\documentclass[11pt]{amsart}
\usepackage{fullpage}
\usepackage{amsmath,amsthm,amssymb}
\usepackage{colonequals}
\usepackage{enumerate}
\usepackage[dvipsnames]{xcolor}
\definecolor{citation}{rgb}{0,.40,.80}
\usepackage[colorlinks,linkcolor=citation,citecolor=citation,urlcolor=citation]{hyperref}
\usepackage{cleveref}% use command \Cref for \ref
\usepackage{array,makecell,subcaption}

\usepackage{tikz}
\usetikzlibrary{cd,arrows,decorations.markings}

\tikzset{% arrow close to the source: the 0.2 determines where the arrow is drawn
  ->-/.style={decoration={markings, mark=at position 0.5 with {\arrow{stealth}}},
              postaction={decorate}},
          }
\tikzset{% arrow close to the source: the 0.2 determines where the arrow is drawn
  ->--/.style={decoration={markings, mark=at position 0.3 with {\arrow{stealth}}},
              postaction={decorate}},
          }
\def\qquad{\quad\quad}

\newcommand \C {\mathbb C}
\newcommand \bA {\mathbb A}
\newcommand{\bC}{\mathbb{C}}

\newcommand{\bP}{\mathbb{P}}

\newcommand{\bZ}{\mathbb{Z}}

\newcommand{\cD}{\mathcal{D}}

\newcommand{\cO}{\mathcal{O}}
\newcommand{\fm}{\mathfrak{m}}
\newcommand{\fn}{\mathfrak{n}}

\newcommand{\eps}{\varepsilon}
\newcommand{\gtw}{G_{12}}

\newcommand \Z {\mathbb Z}
\newcommand{\<}{\langle}
\renewcommand{\>}{\rangle}
\newcommand{\ep}{\epsilon}
\newcommand{\dbg}{D^G}
\def\refmain{\hyperlink{main}{Theorem A}}
\def\refcor{\hyperlink{cormain}{Corollary B}}

\let\phi\varphi

\DeclareMathOperator \can {can}

\DeclareMathOperator \End {End}
\DeclareMathOperator \GL {GL}
\DeclareMathOperator \HHilb {\textrm{$H$-Hilb}}
\newcommand{\hhilbc}{\HHilb(\C^2)}
\DeclareMathOperator \GHilb {\textrm{$G$-Hilb}}
\DeclareMathOperator \Hom {Hom}
\DeclareMathOperator \Hilb {Hilb}
\DeclareMathOperator \Ind {Ind}
\DeclareMathOperator \nat {nat}

\DeclareMathOperator \SL {SL}
\DeclareMathOperator \Spec {Spec}

\numberwithin{equation}{section}
\theoremstyle{plain}
\newtheorem{thm}[equation]{Theorem}
\newtheorem*{thma}{Theorem~A}
\newtheorem*{corb}{Corollary~B}
\newtheorem{prop}[equation]{Proposition}
\newtheorem{lem}[equation]{Lemma}
\newtheorem{cor}[equation]{Corollary}
\newtheorem{conj}[equation]{Conjecture}

\theoremstyle{definition}
\newtheorem{dfn}[equation]{Definition}

\theoremstyle{remark}
\newtheorem{rmk}[equation]{Remark}

\title{An explicit derived McKay correspondence for some complex reflection groups of rank two}

\author{Anirban Bhaduri}
\author{Yael Davidov}
\author{Eleonore Faber}
\author{Katrina Honigs}
\author{Peter McDonald}
\author{C.~Eric Overton-Walker}
\author{Dylan Spence}

\address{Department of Mathematics, University of South Carolina, Columbia, SC 29208, USA}
\address{Department of Mathematical Sciences,
University of Delaware, 501 Ewing Hall, 
Newark, DE 19716, USA}
\address{Institut f\"ur Mathematik und Wissenschaftliches Rechnen,
Universit\"at Graz,
Heinrichstr.~36,
A-8010 Graz, Austria and School of Mathematics, University of Leeds, LS2 9JT Leeds, UK }
\address{Department of Mathematics, Simon Fraser University, 8888 University Drive, Burnaby, British Columbia V5A 1S6, Canada}
\address{Department of Mathematics, Statistics, and CS, University of Illinois at Chicago, Chicago, IL 60607, USA}
\address{Department of Mathematical Sciences, 309 SCEN, University of Arkansas, Fayetteville, AR 72701}
\address{Department of Mathematics, University of Wisconsin-Whitewater, 237 N. Prince Street, Whitewater, WI 53190, USA}

\date{\today}
\allowdisplaybreaks
\begin{document}

\begin{abstract}
  In this paper, we explore the derived McKay correspondence for several reflection groups, namely reflection groups of rank two generated by reflections of order two.
  We prove that for each of the reflection groups  $G=G(2m,m,2)$, $G_{12}$, $G_{13}$, or $G_{22}$, there is a semiorthogonal decomposition of the following form, where
$B_1,\ldots,B_r$ are the normalizations of the irreducible components of the branch divisor $\C^2\to \C^2/G$ and 
  $E_1,\ldots,E_n$ are exceptional objects:
  \[
\dbg(\C^2)\cong   \langle
E_1,\ldots,E_n,D(B_1),\ldots, D(B_r), D(\C^2/G)
  \rangle.
\]
We verify that the pieces of this decomposition correspond to
the irreducible representations of $G$, verifying the Orbifold Semiorthogonal Decomposition Conjecture of Polishchuk and Van den Bergh.
Due to work of Potter  on the group $G(m,m,2)$, this conjecture is now proven for all finite groups $G\leq \GL(2,\bC)$ that are generated by order $2$ reflections.
Each of these groups contains, as a subgroup of index $2$, a distinct finite group $H\leq \SL(2,\bC)$.
A key part of our work is an explicit computation of the action of $G/H$ on the $H$-Hilbert scheme $\hhilbc$.
\end{abstract}  

\maketitle

\section{Introduction}\label{sec:intro}

Let $G$ be a finite subgroup of $\SL(2,\bC)$. The classical McKay correspondence connects the representation theory of $G$, the algebra of the ring of invariants $\bC[x,y]^G$, and the geometry of the exceptional divisor of the minimal resolution of singularities $\pi\colon Y\to \bC^2/G$.
Later, Kapranov and Vasserot \cite{Kapranov-Vasserot:2000}
showed that the McKay correspondence may be realized as the following equivalence between the bounded derived category
of $G$-equivariant coherent sheaves on $\bC^2$ and the derived category 
of coherent sheaves on $Y$:
$$\dbg(\bC^2)\simeq D(Y).$$

This result has also been extended to the case of small finite subgroups $G \leq \GL(2,\bC)$, that is, subgroups containing no pseudo-reflections (elements that fix a codimension one subspace), in \cite{Ishii:2002}. This paper is focused on a further extension of the derived McKay correspondence to some finite reflection groups $G\leq \GL(2,\bC)$. 
The Chevalley--Shephard--Todd Theorem \cite{Chevalley,Shephard_Todd}
tells us that in this case $\bC^2/G$ is smooth, so there is no singularity to resolve, making the geometric picture quite different from the classical case. Instead of a derived equivalence as above, there is an embedding of $D(\bC^2/G)$ into $\dbg(\bC^2)$, and moreover a semiorthogonal decomposition described by 
Polishchuk and Van den Bergh's Orbifold Semiorthogonal Decomposition Conjecture:

\begin{conj}[{{\cite{PVdB19}}}]\label{PVdB_conj}
  Suppose $G$ is a finite group acting effectively on a smooth variety $X$, and that for all $\lambda\in G$ the geometric quotient $\bar{X}^\lambda=X^\lambda/C(\lambda)$, where $C(\lambda)$ is the centralizer of $\lambda$ in $G$, is smooth.
  Then there is a semiorthogonal decomposition of $D^G(X)$ 
  whose components $C_{[\lambda]}$ are in bijection with conjugacy classes and $C_{[\lambda]}\cong D(\bar{X}^\lambda)$.
\end{conj}

Our main result gives a geometric description of a semiorthogonal decomposition of $D(\bC^2/G)$ for the reflection groups $G=G(2m,m,2)$ for $m\geq 3$, $G_{12}$, $G_{13}$, and $G_{22}$, which we show correspond to the components in the above conjecture.
Note that the intersections of $G_{12}$, $G_{13}$, and $G_{22}$ with $\SL(2,\C)$ are, respectively, the 
binary tetrahedral, octahedral, and icosahedral groups with singularities of type $E_6$, $E_7$ and $E_8$.

\begin{thma}\label{main}\hypertarget{main}{}
Let $G=G(2m,m,2)$ for $m\geq 3$, $G_{12}$, $G_{13}$, or $G_{22}$. 

There is a semiorthogonal decomposition of the following form,
where $B_1,\ldots,B_r$ are the normalizations of the irreducible components of the branch divisor of $\C^2\to \C^2/G$, $E_1,\ldots,E_n$ are exceptional objects
and $r+n+1$ is the number of distinct irreducible representations of $G$:
\[
\dbg(\C^2)\cong   \langle
D(B_1),\ldots, D(B_r), E_1,\ldots,E_n, D(\C^2/G)
  \rangle.
\]
\end{thma}

\begin{corb}\label{maincor}\hypertarget{maincor}{}
For each group $G$ in \refmain{} acting on $\bC^2$, \Cref{PVdB_conj} holds. \end{corb}  

Our proof strategy is inspired by work of Potter, who proved the analogous result for the dihedral groups $G(m,m,2)$ in \cite{Pot18} (see \cite{Capellan:2024} for further analysis), building on work of Ishii and Ueda \cite{IU15} who gave a semiorthogonal decomposition for finite small subgroups of $\GL(2,\bC)$. An essential step in our arguments is to compute, for each group $G$ appearing in
\refmain, the action of $G/H$ on the $H$-Hilbert scheme $\HHilb(\C^2)$ for $H:=G\cap\SL(2,\C)$. We do this by working with the explicit description of the Hilbert scheme given in these cases by Ito and Nakamura \cite{Ito-Nakamura:1999}.

We wish to acknowledge several other related works.
\refmain{} has already been proven in the case of $G(4,2,2)$ by Lim and Rota \cite{Lim-Rota:2024}.
They prove \Cref{PVdB_conj} in this case using direct methods 
that do not reference the Hilbert scheme; the authors  prove further orbifold semiorthogonal decompositions for groups acting on abelian varieties. In \cite{Kawamata:2016}, Kawamata proves a version of the derived McKay correspondence for finite subgroups of $\GL(2,\bC)$. Specifically, he shows that for any finite subgroup $G\leq\GL(2,\bC)$, there exists some $m\geq0$ and smooth closed subvarieties $Z_i$ of $\bC^2/G$ for $1\leq i \leq m$ such that $D([\bC^2/G])\simeq\langle D(Z_1),\dots, D(Z_m),D(Y)\rangle$ where $Y\to\bC^2/G$ is the minimal resolution. Our proof strategy is similar, and we are able to enumerate the components of the decomposition. In \cite{Kawamata:2018}, Kawamata gives a similar result for finite subgroups of $\GL(3,\bC)$ and connects the McKay correspondence to his conjecture, the \emph{DK-hypothesis}.

More recently, work of Krug \cite{Krug:2025} proves \Cref{PVdB_conj} for the imprimitive complex reflection groups of rank two, 
$G_4,\dots,G_{37}$, notably avoiding the classical McKay correspondence in the process.
Around the same time, Ishii--Nimura \cite{Ishii-Nimura:2025} deduced \Cref{PVdB_conj} for all complex reflection groups of rank $2$ from the work of \cite{Kawamata:2016}.
They also prove the conjecture for real reflection groups of rank $3$.
We would also like to direct the reader to work of \cite{Hu:2025} describing the fixed locus of anti-Poisson involutions on the minimal resolutions of $\bC^2/H$.
In each case we study, the action of $G/H$ on $\hhilbc$ is one of these anti-Poisson involutions, and there is one additional anti-Poisson involution
for the groups $G(2m,m,2)$ and the group $G_{12}$.

We were also inspired by recent work of Buchweitz, Faber, and Ingalls, who give an algebraic McKay correspondence  
in terms of maximal Cohen--Macaulay modules of the discriminant of the reflection group \cite{BFI20}.
% More precisely, let $J\subseteq S:=\bC[x,y]$ be the defining ideal of the hyperplane arrangement and $\Delta\subseteq R:=S^G$ be the defining ideal of the discriminant, the image hyperplane arrangement in the quotient $\bC^2/G$. They show that there is a one-to-one correspondence between nontrivial irreducible $G$-representations, indecomposable projective $\End_{R/\Delta}(S/J)$-modules, and isomorphism classes of graded $R/\Delta$-direct summands of $S/J$.

It would be interesting to explore these connections further for the groups we study. To prove \refcor{}, we establish that the components of the semiorthogonal decomposition in \refmain{} are isomorphic to those predicted by \Cref{PVdB_conj}, but we do not produce an explicit isomorphism. In particular, we do not establish a correspondence between the components of \refmain{} and conjugacy classes of $G$, though the proof of \refmain{} shows some patterns suggesting possible relationships between geometry and the representations of $G$ (see~\Cref{rmk.proof}).

\subsection*{Outline of paper}
In Section \ref{sec:historical_overview}
we give further background on the McKay correspondence and in
\Cref{sec:groups}, we introduce the finite reflection groups $G$
in $\GL(2,\C)$ that are generated by order~$2$ reflections.
In \Cref{sec:hhilb}, for $H:=G\cap \SL(2,\C)$, we  discuss the $H$-Hilbert scheme, which is the primary setting for our computations.
Then in \Cref{sec:SOD}, we discuss the results on semiorthogonal decompositions of categories that we will need to prove our results.
Then, we give proofs of \refmain{} for each of the groups in question in \Cref{sec:proof} and prove \refcor{}.
We gather the specifics of our computations of fixed points in the Hilbert scheme $\HHilb(\C^2)$ under the action of $A:=G/H$ in \Cref{appendix}.

\subsection*{Accompanying code}
The code file \texttt{FixedLocus.m2} associated with this paper can be found at the github repository \cite{code} and as an ancillary file with the arxiv preprint.

\subsection{Conventions and notation}
We work over the base field $\C$. For any variety or stack $X$, we denote the bounded derived category of coherent sheaves on $X$ by $D(X)$. Given a group action $G$ on a variety $X$, $D^G(X)$ is the bounded derived category of $G$-equivariant coherent sheaves on $X$.

Sections \ref{sec:groups}-\ref{sec:proof} and 
\Cref{appendix} are in a common setting; we establish notation for this setting here. Given a reflection group $G\leq \GL(2,\C)$ as in \refmain{}, we write $H:=G\cap\SL(2,\bC)$, which in this setting is index $2$, making $A:= G/H\simeq \Z/2$. We write $Y:=\HHilb(\bC^2)$ for the $H$-orbit Hilbert scheme of $\bC^2$, which is a crepant resolution for $\C^2/H$. The action of $A$ on $\C^2/H$ extends to $Y$.
In the following diagram, the vertical maps are quotient morphisms and the horizontal maps are resolutions.
We note that $\C^2/G$ is smooth, and so the map $Y/A\to \C^2/G$ is not resolving any singularities, but is a resolution in the sense of being proper and birational with a smooth source.
\begin{equation}\label{setting}
  \begin{tikzcd}
    \bC^2\arrow[d]&  \\
     \bC^2/H\arrow[d] & Y\arrow[l]\arrow[d]  \\
     \bC^2/G & Y/A \arrow[l]  
   \end{tikzcd} 
\end{equation}

\subsection*{Acknowledgements} This collaboration started at the American Mathematical Society Mathematics Research Community called ``Derived Categories, Arithmetic and Geometry''. We thank the AMS and National Science Foundation for funding this program.
We also thank Bronson Lim for helpful comments in an early stage of this project.
E.F., K.H., and P.M. thank the SFU math department and PIMS for their hospitality and support.
Work of E.F.~was supported by EPSRC grant EP/W007509/1. This material is based upon work supported by the NSF under Grant No. DMS-1928930 and by the Alfred P. Sloan Foundation under grant G-2021-16778, while E.F.~ was in residence at the Simons Laufer Mathematical Sciences Institute (formerly MSRI) in Berkeley, California, during the Spring 2024 semester.
A.B. was partially supported by the NSF under Award No.~2302263.
K.H.\ is funded by an NSERC Discovery grant. 
D.S. was partially supported by the UW-Whitewater Mathematics Department Strategic Priorities Fund.

\section{Historical Overview of the McKay correspondence}\label{sec:historical_overview}

McKay showed in \cite{McK80} that for any finite subgroup of $\SL(2,\bC)$, there is a one-to-one correspondence between nontrivial irreducible representations of $G$ and irreducible components of the exceptional divisor of the minimal resolution of singularities $\pi\colon Y\to\bC^2/G$. More specifically, he showed that there was an isomorphism of quivers between the Dynkin diagram for the representation theory of $G$ and the resolution graph of $Y$. Gonzalez-Sprinberg and Verdier made this correspondence more geometrically explicit in \cite{GSV83} by constructing a vector bundle for each representation whose first Chern class transversely intersects exactly one irreducible component of the exceptional divisor. As a consequence they determined that the $G$-equivariant Grothendieck group of $\bC^2$ is isomorphic to the ordinary Grothendieck group of $Y$ \cite[Theorem~2.2]{GSV83}.

Soon after, Auslander provided an algebraic version of this correspondence in \cite{Aus86}. Let $S=\bC[x,y]$ and $R=S^G$, the coordinate ring of $\bC^2/G$.
By work of Herzog \cite{Herzog}, there is a 1-1 correspondence between indecomposable reflexive $R$-modules, and indecomposable $R$-summands of $S$. Auslander showed there is an isomorphism $S*G\cong\End_R(S)$, and extended the 1-1 correspondence to indecomposable projective modules over this skew group ring and thus to irreducible representations of $G$. 

In \cite{Ito-Nakamura:1996,Ito-Nakamura:1999}, Ito and Nakamura further developed this correspondence by using Hilbert schemes to construct minimal resolutions. For a finite group $G\leq\GL(r,\bC)$ of order $n$, the $G$-Hilbert scheme $\GHilb(\bC^r)$ is a subscheme of $\Hilb^n(\bC^r)$ that parametrizes certain $G$-invariant $n$-points in $\bC^r$. Ito and Nakamura showed that for $G\leq \SL(2,\bC)$, $\GHilb(\bC^2)$ is not only the minimal resolution of $\bC^2/G$, but a crepant resolution. Furthermore, they give an explicit correspondence between nontrivial irreducible representations of $G$ and the components of the exceptional divisor of $\GHilb(\bC^2)$. Using this moduli-theoretic description of the McKay correspondence, Kapranov and Vasserot in \cite{Kapranov-Vasserot:2000} gave a derived version of the McKay correspondence, proving there is a  derived equivalence $\dbg(\bC^2)\simeq D(Y)$, which gave the previous statement on $K_0$ by \cite{GSV83} and, by passing to a suitable enhancement, higher $K$-groups simultaneously.

Extending the McKay correspondence generally proceeds in two directions. One option is to increase the dimension and consider actions of finite subgroups of $\SL(r,\bC)$ on $\bC^r$. Another is to remain in dimension $2$ but expand our focus to consider other finite subgroups of $\GL(2,\bC)$. Some work has been done in both directions; we briefly survey them here.

In the case where $G$ is a finite small subgroup of $\GL(2,\bC)$, a special version of the McKay correspondence has long been known. Auslander's results \cite{Aus86} still hold, giving a one-to-one correspondence between irreducible representations of $G$ and indecomposable reflexive $\bC[x,y]^G$-modules. However, the geometric picture becomes somewhat more nuanced, as the exceptional divisor of the minimal resolution $\pi:Y\to\bC^2/G$ does not have as many components as there are irreducible representations of $G$. Wunram introduced special representations in \cite{Wunram:1988} and showed that there is a bijection between irreducible components of the exceptional divisor of the minimal resolution of $\bC^2/G$ and nontrivial irreducible special representations.

This work was extended by Ishii in \cite{Ishii:2002}, who showed that $Y=\GHilb(\bC^2)$ is the minimal resolution of $\bC^2/G$ and gave an explicit correspondence between the nontrivial irreducible special representations of $G$ and the irreducible components of the exceptional divisor of $Y$, \`a la \cite{Ito-Nakamura:1999}. Ishii also produced a fully faithful embedding of $D(Y)\hookrightarrow D^G(\bC^2)$. Several years later, Ishii and Ueda gave the appropriate derived version of the correspondence in \cite{IU15}. In this setting, the embedding of $D(Y)$ inside of $D^G(\bC^2)$ from \cite{Ishii:2002} is part of a semiorthogonal decomposition. The other subcategories in this semiorthogonal decomposition are generated by exceptional objects corresponding to the non-special representations of $G$.

In dimension $3$, Bridgeland, King, and Reid show that, similar to the $2$-dimensional case, for finite subgroups $G\leq\SL(3,\bC)$, $Y=\GHilb(\bC^3)$ is a crepant resolution of singularities of $\bC^3/G$ and $\dbg(\bC^3)\simeq D(Y)$ \cite{Bridgeland-King-Reid:2001}.
Moreover, the Craw--Ishii conjecture holds: every projective crepant resolution of $\bC^3/G$ is isomorphic to a moduli space of stable $G$-constellations \cite{MR4871720}.
Beyond dimension $3$, it is unknown in general if a crepant resolution of $\bC^r/G$ exists. However, Nakamura has conjectured that when a crepant resolution exists, $Y:=\GHilb(\bC^r)$ is such a resolution. As noted before, Kawamata has also proved a version of the derived McKay correspondence for finite subgroups of $\GL(3,\bC)$ \cite{Kawamata:2018}.

\section{Finite reflection subgroups of \texorpdfstring{$\GL(2,\C)$}{GL2C} containing \texorpdfstring{$-1$}{-1}}\label{sec:groups}

In this section, we will examine the reflection groups that appear in \refmain. These are precisely the complex reflection groups
in $\GL(2,\C)$ that are generated by order $2$ reflections, where we take a reflection to be any (order $2$) element of $\GL(2,\C)$ whose fixed locus is codimension $1$.
Moreover, these groups are in one-to-one correspondence with the finite subgroups of $\SL(2,\C)$, up to conjugation. 
This correspondence is given by intersecting with $\SL(2,\C)$. 
Each of these reflection groups $G$ has the intersection
$H:=G\cap \SL(2,\C)$ as an index~$2$ subgroup; $G$ is an extension of $H$ by $-1$.
%has the intersection a subgroup of $\SL(2,\C)$, given by   In fact, intersecting with $\SL(2,\C)$ gives a bijection between finite reflection subgroups of $\GL(2,\C)$ that contain $-1$ and finite subgroups of $\SL(2,\C)$, up to conjugacy.
The primary source for this section is \cite{GSV83}. 
The correspondence we discuss is also shown in \cite{Shephard_Todd,Knorrer}; see \cite{BFI} for a more recent examination.

The exceptional reflection groups are denoted by $G_{12}$, $G_{13}$, and $G_{22}$. Their intersections with $\SL(2,\C)$ are the binary tetrahedral, octahedral, and icosahedral groups with singularities of type $E_6$, $E_7$ and $E_8$.
The intersection of the reflection group $G(2m,m,2)$ with $\SL(2,\C)$ is isomorphic to the binary dihedral group of order $4m$.
The quotient $\C^2/H$ has a surface singularity of $D_{m+2}$ type; this notation is not to be mistaken for the underlying dihedral group.
To complete the classification, there are the reflection groups $G(m,m,2)$, which are isomorphic to the dihedral group of order $2m$. Their intersections with $\SL(2,\C)$ are cyclic of order $m$, and their quotients have surface singularities of type $A_{m-1}$.

\begin{rmk}
  The representations of the reflection groups $G$ have a close relationship with those of their subgroups $H$. Each reflection group has a nontrivial $1$-dimensional representation $\ep$ whose restriction to $H$ is trivial. The quotient group $A:=G/H$ acts on the representations of $H$ by sending them to their contragredients. This involution on representations of $H$ relates to the involution on representations of $G$ given by $\ep\otimes(-)$ via induction and restriction. By \cite[Prop.~3.4]{GSV83}, if $\rho$ is a representation of $H$ isomorphic to its contragredient, then $\Ind_G\rho$ is the sum of two distinct irreducible representations $\rho'$ and $\rho'\otimes\ep$. Each of these representations restrict to $\rho$. If $\rho$ is a representation of $H$ that is not isomorphic to its contragredient, then $\Ind_G\rho$ is irreducible and isomorphic to its tensor product with $\ep$. 
\end{rmk}

  \begin{figure}[h!]
    \begin{subfigure}[t]{0.49\textwidth}
      \centering
      \begin{tikzpicture}
        \node at (-3,0) {$E_6\colon$};
        \node at (-2,1) {$\bullet$};
        \node at (-2,-1) {$\bullet$};
        \node at (-1,1) {$\bullet$};
        \node at (-1,-1) {$\bullet$};
        \node at (0,0) {$\bullet$};
        \node at (1,0) {$\bullet$};
        \node at (2,0) {$\times$};
        \draw (-2,1)--(-1,1);
        \draw (-2,-1)--(-1,-1);
        \draw (-1,1)--(0,0);
        \draw (-1,-1)--(0,0);
        \draw (0,0)--(1,0);
        \draw (1,0)--(2,0);
        \node[above] at (-2,1) {$\rho_1'$};
        \node[below] at (-2,-1) {$\rho_1''$};
        \node[above] at (-1,1) {$\rho_2'$};
        \node[below] at (-1,-1) {$\rho_2''$};
        \node[above] at (0,0) {$\rho_3$};
        \node[above] at (1,0) {$\rho_2$};
        \node[above] at (2,0) {$\rho_0$};
      \end{tikzpicture}
    \end{subfigure}
    \begin{subfigure}[t]{0.49\textwidth}
      \centering
      \begin{tikzpicture}
        \node at (4,0) {$E_6'\colon$};
        \node at (5,0) {$\bullet$};
        \node at (6,0) {$\bullet$};
        \node at (7,1) {$\bullet$};
        \node at (7,-1) {$\bullet$};
        \node at (8,1) {$\bullet$};
        \node at (8,-1) {$\bullet$};
        \node at (9,1) {$\times$};
        \node at (9,-1) {$\bullet$};
        \draw(5,0)--(6,0);
        \draw (6,0)--(7,1);
        \draw (6,0)--(7,-1);
        \draw[->-] (8,1)--(7,1);
        \draw[->--] (7,-1)--(8,1);
        \draw[->--] (7,1)--(8,-1);
        \draw[->--] (8,-1)--(9,1);
        \draw[->--] (8,1)--(9,-1);
        \draw[->-] (8,-1)--(7,-1);
        \draw[->-] (9,1)--(8,1);
        \draw[->-] (9,-1)--(8,-1);
        \node[above] at (5,0) {$\rho_1'$};
        \node[above] at (6,0) {$\rho_2'$};
        \node[above] at (7,1) {$\rho_3$};
        \node[below] at (7,-1) {$\ep\rho_3$};
        \node[above] at (8,1) {$\rho_2$};
        \node[below] at (8,-1) {$\ep\rho_2$};
        \node[above] at (9,1) {$\rho_0$};
        \node[below] at (9,-1) {$\ep\rho_0$};
      \end{tikzpicture}
    \end{subfigure}
  \caption{$E_6$ and $E_6'$ Dynkin diagrams}\label{E6.dynkincopy}
  \end{figure}

The McKay quivers for $E_6$ and $E_6'$ are shown in \Cref{E6.dynkincopy}; the quivers for 
each of the groups in question can be found in \Cref{appendixb}. For the subgroups of $\SL(2,\C)$, these are the extended $A$, $D$, and $E$ Dynkin diagrams. We label the quivers of the corresponding reflection groups as $A'$, $D'$, and $E'$. In the quivers, each vertex corresponds to an irreducible representation. Edges are determined using the natural representation $\rho_{\nat}$ given by the inclusion of the group in $\GL(2,\C)$:
an arrow from $\rho_i$ to $\rho_j$ indicates that $\rho_j$ is a summand of $\rho_{\nat}\otimes\rho_i$. In cases where this relationship is symmetric, we consolidate arrows in both directions to a single undirected edge.
Since we will use results in \cite{Ito-Nakamura:1999} for our computations, we follow their labelling system for representations of the subgroups of $\SL(2,\C)$. The subscripts on the representations in the $E_6$, $E_7$, and $E_8$ diagrams indicate their dimensions, but this is not the case in the diagrams of type $A$ and $D$.
We then somewhat abuse this notation by reusing it for the representations of the corresponding reflection group, given their close relationship.
However, we follow \cite{GSV83} for the particular positioning  of the diagrams; these  place contragredients and tensor products with $\ep$ across a middle, horizontal axis of reflection from one another, to the extent that it is possible without sacrificing compactness of presentation.

Finally, in \Cref{disc.table}, we give the ring of invariants $\bC[x,y]^G$, the branch locus $B$ of the map $\bC^2/H\to\bC^2/G$ (up to a suitable change of coordinates), and the number of connected components of $B$, for each reflection group $G$. This information is from \cite{Bannai:1976}.

\begin{figure}
\renewcommand{\arraystretch}{1.6}
\begin{tabular}{c|c|c|c}
  Group & $\bC[x,y]^G$ & $B$ & \makecell{\# comp.\ \\ of $B$}
\\\hline  
  $G(m,m,2)$ & $\bC[xy,x^m+y^m]$ & $V(z_1^m-z_2^2)$ &
\makecell{$2$ ($m$ even)\\$1$ ($m$ odd)}
  \\\hline
  $G(2m,m,2)$ & $\bC[x^2y^2,x^{2m}+y^{2m}]$ & $V(z_1(z_1^m-z_2^2))$ &
\makecell{$3$ ($m$ even)\\$2$ ($m$ odd)}
  \\\hline
  $G_{12}$  & $\bC[x^8+14x^4y^4+y^8,xy(x^4-y^4)]$ & $V(z_1^3-z_2^4)$ & $1$
  \\\hline 
  $G_{13}$  & $\bC[x^8+14x^4y^4+y^8,(xy)^2(x^4-y^4)^2]$ & $V(z_1(z_1^2-z_2^3))$ & $2$
  \\\hline
  $G_{22}$  &
\makecell{$\bC[xy(x^{10}+11x^5y^5-y^{10})$,\\
$-x^{20}-y^{20}-494x^{10}y^{10}$\\
$+228(x^{15}y^{5}-x^5y^{15})]$}
  & $V(z_1^3-z_2^5)$ & $1$
\end{tabular}
\caption{Table of invariant rings, branch divisors, and number of components in branch divisors}\label{disc.table}
\end{figure}

\section{\texorpdfstring{$\hhilbc$}{H-HilbC}}\label{sec:hhilb}

In this section we examine the structure of $\hhilbc$ and $\hhilbc/A$, concluding with a proof that $\hhilbc/A$ is smooth.

\subsection{Structure of \texorpdfstring{$\hhilbc$}{H-HilbC}}
Let $H$ be a finite subgroup of $\SL(2,\C)$. In \cite{Ito-Nakamura:1999}, Ito and Nakamura give a new perspective on the classical McKay correspondence via moduli theory, using the \emph{$H$-orbit Hilbert scheme}.
An essential part of their work is an explicit description of $\hhilbc$ for each such $H$, which we will use in our computations.
We include an overview here for the convenience of the reader.

\begin{dfn}[{{\cite[Theorem~9.3]{Ito-Nakamura:1999}}}] Let $n=|H|$. The $H$-orbit Hilbert scheme of $\bC^2$, denoted $\HHilb(\bC^2)$, is the unique component of the fixed locus $\Hilb^n(\bC^2)^H\subset \Hilb^n(\bC^2)$ dominating $\bC^2/H$ via the Hilbert--Chow morphism
  $\Hilb^n(\bC^2)^H\to S^n(\bC^2)^H\simeq \bC^2/H$.
  Furthermore, $\HHilb(\bC^2)$ is a crepant (equivalently minimal) resolution of $\bC^2/H$.
\end{dfn}

As a moduli space, $\Hilb^n(\bC^2)^H$
parametrizes $H$-invariant length $|H|$ subschemes of $\bC^2$, equivalently, $H$-invariant ideals of the coordinate ring $\bC[x,y]$.
Then $Y:=\HHilb(\bC^2)$ paramatrizes such subschemes that furthermore correspond to ideals $I$ so that there is an isomorphism of $H$-modules between
$\cO_{\bC^2}/I$ and $\bC[H]$, the regular representation of $H$. 
The \emph{exceptional locus} $E$ of $Y\to \bC^2/H$ consists of the $H$-invariant subschemes supported at the origin.
We call the open subset of $Y$ outside the exceptional locus
$N\subseteq Y$ for \emph{non-exceptional}.
$H$ acts freely on $\bC^2$ outside of the origin,
and thus the set of $H$-orbits of $\bC^2\setminus \{(0,0)\}$ is isomorphic to $\bC^2\setminus \{(0,0)\}$, so $N\simeq \bC^2\setminus \{(0,0)\}$.

%Outside of the exceptional locus, the points of $Y$ consist of $H$-orbits of $\bC^2$. %and thus are isomorphic to $\bC^2\setminus \{(0,0)\}$.
%We call this open subset .

We will generally refer to points in $\hhilbc$ using their corresponding ideals, which we now describe.

\subsubsection{Exceptional locus of $Y$}

Fixing coordinates so that $\bC^2=\Spec\bC[x,y]$, let $\fm=(x,y)$ be the maximal ideal of the origin in $\bC^2$, $\fm_S$ be the maximal ideal at the origin in $\C^2/H$, and 
$\fn:=\fm_S\cdot \C[x,y]$, i.e., 
the ideal generated by $H$-invariant polynomials in $\C[x,y]$.

Let $I$ be an ideal corresponding to a point in $E\subset Y$.
Because $I$ is supported at the origin, $I\subseteq\fm$ and $\fn\subseteq I$ \cite[Corollary~9.6]{Ito-Nakamura:1999}.
Following the notation of \cite{Ito-Nakamura:1999},
we define for convenience the following finite $H$-module: %, which is isomorphic to a minimal $H$-submodule of $I/\fn$ that generates $I/\fn$ over $\bC[x,y]$:
\[V(I):=I/(\fm I+\fn).\]

Let $V(\rho)$ be the nontrivial irreducible $H$-module corresponding to the nontrivial irreducible representation $\rho:H\to\GL(V(\rho))$. We define the following loci in $E\subset Y$:
\begin{align*}
E(\rho)&:=\{I\in \hhilbc\colon V(\rho)\subseteq V(I)\},\\
P(\rho,\rho')&:=\{I\in \hhilbc \colon V(\rho)\oplus V(\rho')\subseteq V(I)\}.\end{align*}
Ito and Nakamura prove the McKay correspondence by showing that the assignment $\rho\mapsto E(\rho)$ gives a bijection between nontrivial irreducible representations of $H$ and the irreducible components $E(\rho)\simeq\bP^1$ of $E$. The locus $P(\rho,\rho')$ is the intersection of $E(\rho)$ and
$E(\rho')$; it is nonempty if and only if $\rho$ and $\rho'$ are adjacent in the Dynkin diagram of $H$.
If $P(\rho,\rho')\neq\emptyset$, then it consists of a single reduced point at which $E(\rho)$ and $E(\rho')$ intersect transversally \cite[Theorem~10.4]{Ito-Nakamura:1999}.

For each finite group $H\leq \SL(2,\bC)$, 
Ito and Nakamura identify 
the ideals corresponding to points in the exceptional locus as certain submodules of $\fm/\fn$. We recount the details in \Cref{IN.main} in \Cref{appendix} before using it for our own computations.

\subsubsection{Outside of the exceptional locus of $Y$}

For each finite subgroup $H$, the ideal $\fn\subseteq \C[x,y]$ can be generated by three polynomials $f_1,f_2,f_3\in\bC[x,y]$. The points of $N$ correspond to points of $\bC^2\setminus\{(0,0)\}$. In particular, $(a,b)\in \bC^2\setminus\{(0,0)\}$ corresponds to the following ideal:
\[
I_{(a,b)}:=(f_1(x,y)-f_1(a,b),f_2(x,y)-f_2(a,b),f_3(x,y)-f_3(a,b)).
\]
\begin{rmk}
  The ideals $I_{(a,b)}$ are products of the maximal ideals of all points in the $H$-orbit of $(a,b)$. Since these points are distinct, the product of these maximal ideals is equal to their intersection, hence 
the ideals  $I_{(a,b)}$ are radical.
 Thus, if some $f\in \C[x,y]$ does not vanish at any of the points in the $H$-orbit, $f$ is a unit in $\bC[x,y]/I_{(a,b)}$. 
\end{rmk}

\subsection{Structure of \texorpdfstring{$\hhilbc/A$}{H-HilbC/A}}

It is significant to our proof of \refmain{} that the quotient of $Y$ by
$A:=G/H\simeq \Z/2$ is smooth. The geometry of the branch locus of the quotient $Y\to Y/A$ also plays a crucial role. 

In order to elucidate both of these, we will calculate the fixed locus of the action of $A$ on $Y$. In \Cref{appendix} we compute the fixed points of the action of $A$ on the exceptional locus of $Y$.
For each group $G$, these computations  show that there are points of $Y$ fixed by $A$ that are isolated \emph{within the exceptional locus}. In this section, we will show that these fixed points are not isolated in $Y$ by examining the parts of the fixed locus of the action of $A$ on $Y$ that extend outside of the exceptional locus.

Therefore, we wish to compute which ideals $I_{(a,b)}$ are fixed by the action of $A$ in each case. We will arrange our choice of $f_1$, $f_2$, and $f_3$ in each case so that $f_1$ and $f_2$ are fixed by the action of $A$ and $f_3$ is sent to $-f_3$. These choices of $f_1$, $f_2$ and $f_3$ are all given in the subsections below; the actions of $A$ in each case are given in \Cref{appendix}.

The kernel of the following polynomial map is principally generated by a polynomial of the form $w^2+p(u,v)$:
\begin{equation}\label{uvw}
\C[u,v,w]\to\C[x,y], \quad u\mapsto f_1,\ v\mapsto f_2,\ w\mapsto f_3
\end{equation}  
The above map surjects onto $\C[x,y]^H$ and gives the relation
$f_3^2+p(f_1,f_2)=0$.
In this setting we have the following result:

\begin{lem}\label{puv}
Each irreducible factor of $p(u,v)$ corresponds to a component of the ramification locus of $\bC^2/H\to \bC^2/G$, which in turn corresponds to a component of the fixed locus of $A$ acting on $Y$ that has nontrivial intersection with $N$. 
\end{lem}   

\begin{proof}
For any $(a,b)\in \C^2\setminus \{(0,0)\}$ that satisfies
$f_3(a,b)=0$, we see directly from the definition that the ideal $I_{(a,b)}$ must be fixed under the action of $A$.
Conversely, the sum of $I_{(a,b)}$ and its image under $A$ is:
$$(f_1(x,y)-f_1(a,b),f_2(x,y)-f_2(a,b),f_3(x,y)-f_3(a,b),f_3(x,y)+f_3(a,b)),$$
which contains $f_3(a,b)$. Thus $I_{(a,b)}$ is fixed by $A$ if and only if $f_3(a,b)=0$.

The fixed locus of $A$ inside of $\C^2/H$, or equivalently, the ramification locus of $\C^2/H\to \C^2/G$, has the following coordinate ring:
$$\C[x,y]^H/(f_3)\cong\C[f_1,f_2,f_3]/(f_3^2+p(f_1,f_2),f_3)\cong \C[f_1,f_2]/(p(f_1,f_2)),$$
so the irreducible components of the fixed locus of $A$ inside of $\bC^2/H$ correspond to the irreducible factors of $p(u,v)$.
\end{proof}  

\begin{rmk}\label{rmk.puv}
Since $\bC^2 \to \bC^2/H$ is ramified only at the origin, the branch divisors of $\bC^2 \to \bC^2/G$ are the same as $\bC^2/H \to \bC^2/G$.
  Note that $\bC[f_1,f_2]\cong\bC[x,y]^G$, and the polynomials
$p(u,v)$ are the branch divisors of $\bC^2\to \bC^2/G$; in each case we treat, we see the polynomials $p(u,v)$ are consistent with the branch divisors shown in \Cref{disc.table}.
\end{rmk}

Now, for each reflection group $G$ that appears in \refmain, we will examine the fixed locus of $A$ acting on $Y$ that has nontrivial intersection with $N$.
For each case below, we provide a family of ideals corresponding to points in $Y$ which is flat over $\bC^*$. 
By \cite[III.9.8]{Hartshorne}, there is a unique extension to a flat family over $\mathbb{A}^1$. We identify the point added in the extension as the intersection of the fixed locus of $A$ and the exceptional set $E$.

\subsubsection{$G(2m,m,2)$}\label{int.G2mm2}
We choose the $H$-invariants of $\C[x,y]$ as in \cite[\S13.1]{Ito-Nakamura:1999}
\[
f_1=x^{2m}+y^{2m},\quad
f_2=x^2y^2,\quad
f_3=xy(x^{2m}-y^{2m}).
\]
In this case, the kernel of the map \eqref{uvw} is generated by
$w^2-v(u^2-4v^m)$
and thus by \Cref{puv}, the ramification locus of $\C^2/H\to \C^2/G$ has two components when $m$ is odd and three components when $m$ is even. 

When $m$ is odd, these components are $V(f_2)$ and $V(4f_2^m-f_1^2)$. When $m$ is even, there are three components: $V(f_2)$, $V(2f_2^{m/2}-f_1)$, and $V(2f_2^{m/2}+f_1)$. 

First, we consider the intersection of the component $V(f_2)$ with the exceptional locus; the same argument works for any value of $m$.
Each ideal $I_{(a,b)}$ corresponding to a point in 
$V(f_2)\cap N$ is radical and thus contains $(xy)$, and, furthermore, in the limit as points $(a,b)$ in this component approach $0$, 
any ideal in the intersection of this component with $E$ must contain $xy$.
Any ideal in the intersection of this component with $E$
must also contain $f_1$, $f_2$, and $f_3$ (see \cite[Cor.~9.6,\S10.1]{Ito-Nakamura:1999}).
There is a unique ideal in the exceptional locus
containing $f_1$, $f_2$, $f_3$, and $xy$, which corresponds to the module
$V_2(\rho_1')=(xy)$ (see \Cref{AIsec} for explanation of notation). We identify this point in 
\Cref{prop.Dn}(e) as a fixed point of $A$ in $E$ that is contained in  $E(\rho_1')$.

We treat the remaining components separately for $m$ even or odd.

Consider the case where $m$ is odd. 
The two components of the fixed locus of the $A$-action on $Y$ that intersect with $N\simeq \C^2\setminus \{(0,0)\}$
are $V(f_2)$ and $V(4f_2^m-f_1^2)$.

The points in $N$ are ideals $I_{(a,b)}$ that
are products of the maximal ideals corresponding to points in the orbit of $(a,b)\in\bC^2$. For any $t\in \bC^*$, the ideal $I_{(t,t)}$ is a point in the component $V(4f_2^m-f_1^2)$ and we write it as an ideal in $\bC[x,y]^H$:
\[I_{(t,t)}:=(f_1(x,y)-2t^{2m},f_2(x,y)-t^4,f_3(x,y)).\]
If we furthermore consider the ideals $I_{(t,t)}$ these generate in $\bC[x,y]$ (abusing notation),
because these ideals are radical and contain $f_1^2-4f_2^m=(x^{2m}-y^{2m})^2$, $I_{(t,t)}$ contains $x^{2m}-y^{2m}$. Thus we can write: 
\[I_{(t,t)}=(x^{2m}+y^{2m}-2t^{2m},x^2y^2-t^4,x^{2m}-y^{2m}).\]
We claim that this ideal contains $x^{m+1}-t^2y^{m-1}$ and $y^{m+1}-t^2x^{m-1}$.
To see this, we first note that since $x^2y^2-t^4\in I_{(t,t)}$, all solutions to $x^{m-1}y^{m-1}-t^{2m-2}$ are contained in the vanishing of $I_{(t,t)}$.
Since $I_{(t,t)}$ is radical, it must contain $x^{m-1}y^{m-1}-t^{2m-2}$, and thus the following equality implies that $x^{m-1}(x^{m+1}-t^2y^{m-1})\in I_{(t,t)}$:
\[x^{m-1}(x^{m+1}-t^2y^{m-1})+t^2(x^{m-1}y^{m-1}-t^{2m-2})=x^{2m}-t^{2m}\in I_{(t,t)}.
\]
Because $x^{m-1}$ is not contained in any of the maximal ideals which intersect to give $I_{(t,t)}$ (note $I_{(a,0)}$, $I_{(0,b)}$ are in $V(f_2)$) this implies $x^{m+1}-t^2y^{m-1}\in I_{(t,t)}$. By a similar argument, we see $y^{m+1}-t^2x^{m-1}\in I_{(t,t)}$. Then we see that  $\lim_{t\to0}I_{(t,t)}$ must contain
$x^{m+1}$ and $y^{m+1}$ as well as
$f_1$, $f_2$, and $f_3$.  There is a unique ideal in the exceptional locus with this property,
which corresponds to the module $V_{m+1}'(\rho_m)=(x^{m+1},y^{m+1})$
in the notation of \Cref{AIsec}.
In \Cref{prop.Dn}(e), we identify this point as a fixed point of $A$ in $E$ that is contained in  $E(\rho_m)$.

Now, let $m$ be even.
The three components of the fixed locus of $A$ on $Y$ that intersect with $N\simeq \C^2\setminus \{(0,0)\}$
are $V(f_2)$, $V(2f_2^{m/2}-f_1)$, and $V(2f_2^{m/2}+f_1)$.

For any $t\in \bC^*$, the ideal $I_{(t,t)}$ is a point in the component
$V(2f_2^{m/2}-f_1)$ and we write it as an ideal in $\bC[x,y]^H$:
\[I_{(t,t)}=(f_1(x,y)-2t^{2m},f_2(x,y)-t^4,f_3(x,y)).\]
Consider the ideals $I_{(t,t)}$ these generate in $\bC[x,y]$,
because this ideal is radical and contains $f_1-2f_2^{m/2}=(x^{m}-y^{m})^2$, it contains $x^{m}-y^{m}$. Thus we can write our ideals
\[I_{(t,t)}=(x^{2m}+y^{2m}-2t^{2m},x^2y^2-t^4,x^{m}-y^{m}).\]
Then, $\lim_{t\to0}I_{(t,t)}$ must contain $x^m-y^m$ as well as $f_1$, $f_2$, $f_3$.

By a similar argument,
noting that $2f_2^{m/2}+f_1=(x^m+y^m)^2$,
the intersection of component $V(2f_2^{m/2}+f_1)$ with the exceptional locus is an ideal containing $x^m+y^m$ as well as $f_1$, $f_2$, and $f_3$. 

There are unique ideals in the exceptional locus
that contain $f_1,f_2,f_3,x^m-y^m$ and $f_1,f_2,f_3,x^m+y^m$, respectively,
though which ideal depends on whether $m$ is divisible by~$4$. 
If $m$ is not divisible by $4$, these ideals correspond to the modules
$V_m(\rho_{m+1}')=(x^m-i^{m+2}y^m)$ and $V_m(\rho_{m+2}')=(x^m+i^{m+2}y^m)$ (see \Cref{AIsec} for notation), respectively, and vice versa if $4|m$.

We identify these points in 
\Cref{prop.Dn}(e) as fixed points of $A$ in $E$ that are contained in  $E(\rho_{m+1}')$ and $E(\rho_{m+2}')$.

\subsubsection{$G_{12}$}\label{int.G12}
We choose the invariants of $\C[x,y]^H$ as in \cite{Slodowy:1980}, which 
coincides with \cite[\S14]{Ito-Nakamura:1999} except for the choice of $f_3$ (for unexplained notation see Appendix \ref{gtw.comp}):
\begin{align*}
f_1&=x^5y-xy^5=p_1p_2p_3\\
f_2&=x^8+14x^4y^4+y^8=\phi\psi=(p_2^2+4\omega p_3^2)(p_2^2+4\omega^2 p_3^2)\\
f_3&=x^{12}-33x^8y^4-33x^4y^8+y^{12}
\end{align*}
The kernel of the map \eqref{uvw} is generated by
$w^2-(v^3-108u^4)$ 
and thus by \Cref{puv}, the ramification locus of $\C^2/H\to \C^2/G$ has one component, corresponding to one component of the fixed locus of the action of $A$ on $Y$. Its intersection with $N\simeq \C^2\setminus \{(0,0)\}$
is $V(f_2^3-108f_1^4)$.

Let $\varepsilon:=e^{2\pi i/8}$. For any $t\in \bC^*$, the ideal $I_{(t,\varepsilon t)}$ is a point in the component $V(f_2^3-108f_1^4)$:
\[
I_{(t,\varepsilon t)}:=(f_1-2\varepsilon t^6,f_2+12t^8,f_3).
\]
By \Cref{G12_A_summary}(c), the exceptional locus of $\HHilb(\bC^2)$ has a single isolated fixed point under the $A$-action on $E(\rho_3)$ corresponding to the following $H$-module:
\[V(I)=((x^2-y^2)(-\omega^2\varphi+\psi),xy(-\varphi+\psi),(x^2+y^2)(-w\varphi+\psi),f_1,f_2,f_3).\]
We are able to verify computationally using \texttt{macaulay2} in the code file \texttt{FixedLocus.m2} \cite{code},
for each $t$
and each generator of $V(I)$
$I_{(t,\varepsilon t)}$ contains an element that is the sum of 
the generator of $V(I)$ and an element that is a multiple of $t$. 
Thus $\lim_{t\to 0}I_{(t,\varepsilon t)}$ must contain all of the generators listed in $V(I)$, hence it is the point in $E$ corresponding to $V(I)$.

\subsubsection{$G_{13}$}\label{int.G13}
We choose the invariants of $\C[x,y]^H$ as in \cite{Slodowy:1980}, which 
differs from but is consistent with that in \cite[\S15]{Ito-Nakamura:1999} in this case:
\[
f_1=(x^5y-xy^5)^2,\quad
f_2=x^8+14x^4y^4+y^8,\quad
f_3=(x^{12}-33x^8y^4-33x^4y^8+y^{12})(x^5y-xy^5).
\]
In this case, the kernel of the map \eqref{uvw} is generated by
$w^2-u(v^3-108u^2)$,
and thus by \Cref{puv}, the ramification locus of $\C^2/H\to \C^2/G$ has two components, corresponding to two components of the fixed locus of the action of $A$ on $Y$. Their intersections with $N\simeq \C^2\setminus \{(0,0)\}$
are $V(f_1)$ and $V(f_2^3-108f_1^2)$.

We consider the intersection of the first component of the fixed locus with $E$. Each ideal $I_{(a,b)}$ corresponding to a point in 
$V(f_1)\cap N$ is radical and thus contains $(x^5y-xy^5)$.
In the limit as points $(a,b)$ in this component approach $0$, we see that 
any ideal in the intersection of this component with $E$ must contain $(x^5y-xy^5)$ as well as 
$f_1$, $f_2$, and $f_3$. There is a unique ideal in the exceptional locus with this property, namely, in the notation of
\Cref{AIsec},
$V_6(\rho_1')=(x^5y-xy^5)$. We identify this point in 
\Cref{G13_A_summary}(c) as a fixed point of $A$ in $Y\cap E$ that is contained in  $E(\rho_1')$.

Similarly, each ideal
$I_{(a,b)}$ corresponding to a point in 
$V(f_2^3-108f_1^2)\cap N$ is radical and thus contains $x^{12}-33x^8y^4-33x^4y^8+y^{12}$. The intersection of this component of the fixed locus of $A$ in $Y$ with $E$ must be the unique ideal $V_8(\rho_2'')=(x^{12}-33x^8y^4-33x^4y^8+y^{12})$
that we identify in \Cref{G13_A_summary}(c) as a point on $E(\rho_2'')$.

\subsubsection{$G_{22}$}\label{int.G22}
We choose the following invariants of $\C[x,y]^H$ (see \cite{Bannai:1976}, \cite[p.~55]{Klein}, \cite{Slodowy:1980}):
\begin{align*}
f_1&=xy(x^{10}+11x^5y^5-y^{10})\\
f_2&= x^{20}-228x^{15}y^5+494x^{10}y^{10}+228x^5y^{15}+y^{20}\\
f_3&=x^{30}+522x^{25}y^5-10005x^{20}y^{10}-10005x^{10}y^{20}-522x^5y^{25}+y^{30}
\end{align*}
The kernel of the map \eqref{uvw} is generated by
$w^2-(1728u^5+v^3)$, 
and thus by \Cref{puv}, the ramification locus of $\C^2/H\to \C^2/G$ has one component, corresponding to one component of the fixed locus of the action of $A$ on $Y$. Its intersections with $N\simeq \C^2\setminus \{(0,0)\}$
is $V(1728 f_1^5+f_2^3)$.

Let $\gamma:=e^{2\pi i/20}$. For any $t\in \bC^*$, the ideal $I_{(t,\gamma t)}$ is a point in the component $V(1728 f_1^5+f_2^3)$:
\[
I_{(t,\gamma t)}:=(f_1- (2+11i)\gamma t^{12},f_2+(492+456i)t^{20},f_3)
\]
By \Cref{G22_A_summary}(c), the exceptional locus of $\HHilb(\bC^2)$ has a single isolated fixed point on $E(\rho_3)$ corresponding to the following $H$-module:
\[V_{14}(\rho_3'')=(x^{14}-14x^9y^5+49x^4y^{10},7x^{12}y^2-48x^7y^7-7x^2y^{12}, 49x^{10}y^4+14x^5y^9+y^{14})\]
We are able to verify in our code file that 
for each $t$
and each generator of $V_{14}(\rho_3'')$, 
$I_{(t,\gamma t)}$ contains an element that is the sum of 
the generator of $V(I)$ and an element that is a multiple of $t$. 
Thus $\lim_{t\to 0}I_{(t,\varepsilon t)}$ must contain all of the generators listed in $V_{14}(\rho_3'')$, hence it is the point in $E$ corresponding to $V_{14}(\rho_3'')$.

\subsubsection{Smoothness of $Y/A$}\label{smoothYA}
We now use the above results to show $Y/A$ is smooth and draw conclusions about the branch locus of $Y\to Y/A$.

\begin{prop}\label{no_isolated}
Let $G\subseteq\GL(2,\bC)$ be a complex reflection group as in \refmain.
The fixed locus of the action of $A$ on $Y$ does not have any isolated fixed points.
\end{prop}

\begin{proof}
By the arguments shown in sections \ref{int.G2mm2}-\ref{int.G22}, 
each fixed point of the action of $A$ on $Y$ is contained in a component of codimension $1$ and in particular there are no isolated fixed points.
\end{proof}  

By the Chevalley--Shephard--Todd Theorem \cite{Chevalley,Shephard_Todd}, \Cref{no_isolated} implies the following corollary.

\begin{cor}\label{YAsmooth}
The quotient $Y/A$ is smooth.
\end{cor}

\begin{rmk}
In \cite{Pot18}, smoothness of $Y/A$ in the $G=G(m,m,2)$ case is shown by analyzing affine charts of the toric minimal resolution of $\C^2/G$, which is not available in the cases we treat.
\end{rmk}  

We may further identify the geometry of the components of the branch locus of $Y\to Y/A$:

\begin{cor}\label{smooth_lemma}
The branch locus of $Y\to Y/A$ that extends outside the exceptional locus
of $Y/A\to \C^2/G$ 
is isomorphic to a disjoint union of affine lines. 
\end{cor}

\begin{proof}
By \Cref{corollary_quotient_stack_is_root_stack} (cf.~\cite[Cor.~5.3.4]{Pot18}),
since $Y/A$ is smooth and hence coincides with $(Y/A)^{\can}$, the branch locus of $Y\to Y/A$ consists of smooth divisors.
  
Let $B$ be an irreducible component of the branch locus of $Y\to Y/A$ that extends outside the exceptional locus of $Y/A\to \C^2/G$.
Consider the diagram \eqref{setting} in the Conventions section. Since it commutes, $B$ is the strict transform (and therefore the normalization) of a branch divisor of $\bC^2/H \to \bC^2/G$.

Each component of the branch locus of $\bC^2/H\to \bC^2/G$ is given in terms of an explicit equation in sections \ref{int.G2mm2}-\ref{int.G22}. Our equations all normalize to lines.
\end{proof}

\section{Semiorthogonal Decompositions of Equivariant Derived Categories}\label{sec:SOD}

An important tool in our proof of \refmain{} is a collection of results concerning semiorthogonal decompositions of equivariant derived categories. These results were first combined and studied by Potter in his thesis \cite[Chapter 6]{Pot18}. There are essentially three moving pieces to this result. The first result is a semiorthogonal decomposition of the derived category of a root stack, proven in \cite{IU15} and \cite{BLS16} independently. The second result is another semiorthogonal decomposition of the derived category of the canonical stack associated to a surface. Finally, the third is a description of a quotient stack as an iterated root stack along the branching divisors of a group action, proven in \cite{Geraschenko-Satriano17}. Combined, these results allow us to realize the derived category of equivariant sheaves as the derived category of sheaves on an iterated root stack over the canonical stack of a surface, allowing us to obtain the claimed semiorthogonal decomposition using the first two results mentioned above.

We first remind the reader about the notion of a root and canonical stack. Then, we examine the topic of semiorthogonal decompositions. Finally, we assemble these ideas to outline the method we will use to prove \refmain.

\subsection{Root and Canonical Stacks} The original definition of a root stack was given independently in \cite{Cad07} and \cite{AGV08}, while the notion of a canonical stack is implicit in \cite[2.9, proof of 2.8]{Vistoli-Intersection_theory_on_algebraic_stacks_and_moduli_spaces} (see also \cite{FantechiEtAl-Smooth_toric_DM_stacks} for discussion). We will only need the basics of both, and refer the readers to the above sources for more details. 

We begin the discussion with the basics of root stacks. Let $X$ be a smooth Deligne--Mumford stack and let $\mathbf{D} = (D_1,\ldots,D_n)$ be $n$ effective Cartier divisors on $X$. Further, let $\mathbf{r} = (r_1,\ldots,r_n) \in \Z_{>0}^n$ be a tuple of positive integers. Recall that $\mathbf{D}$ determines a morphism $X \to [\mathbb{A}^n/\mathbb{G}_m^n]$, which we also denote by $\mathbf{D}$, and let $\theta_{\mathbf{r}}$ be the morphism $[\mathbb{A}^n/\mathbb{G}_m^n] \to [\mathbb{A}^n/\mathbb{G}_m^n]$ determined by $x_i\mapsto x_i^{r_i}$, and $\lambda_i \mapsto \lambda_i^{r_i}$.

A definition of the root stack of $X$ along $\mathbf{D}$ in terms of its generalized points is straightforward but unnecessary for our purposes (cf. \cite[Remark 2.2.2]{Cad07}). We alternatively define the root stack of $X$ along $\mathbf{D}$ as the fiber product 
\begin{center}
  \begin{tikzcd}
X[\sqrt[\mathbf{r}]{\mathbf{D}}] \arrow[r] \arrow[d,]  & \left[\mathbb{A}^n/\mathbb{G}_m^n \right] \arrow[d, "\theta_\mathbf{r}"] \\
X \arrow[r, "\mathbf{D}"] &  \left[\mathbb{A}^n/\mathbb{G}_m^n\right] 
\end{tikzcd}
\end{center}

\begin{rmk} \label{remark_embedding of root stack strict transform}
  The root stack behaves much like the notion of a blow-up in classical birational geometry. Indeed if $D$ is the union of the $D_i$, then the restriction $X[\sqrt[\mathbf{r}]{\mathbf{D}}]|_{X\setminus D}$ is isomorphic to $X \setminus D$, however its restriction to $D$ is much more interesting (see \cite{Cad07}). Intuitively,
in the case where $\mathbf{r}=(r)$ and $\mathbf{D}=(D)$, 
the $\mathbf{r}$th root stack of $X$ along $\mathbf{D}$ modifies $X$ only
on %along
$D$, resulting in a stack with stabiliser groups of $\mu_r$ along $D$. The objects of the root stack serve as $r$th roots of the line bundle $\cO_X(D)$.
\end{rmk}

\begin{rmk}\label{remark_iterated root stacks}
  It will be useful later on to recognize that by definition of the root stack $X [ \sqrt[\mathbf{r}]{\mathbf{D}}]$, there is a canonical isomorphism $$X [ \sqrt[\mathbf{r}]{\mathbf{D}}] \cong  X [\sqrt[r_1]{D_1}]  [\sqrt[r_2]{D_2}] \cdots [\sqrt[r_n]{D_n}].$$ In other words, we may take root stacks iteratively.
\end{rmk}

Canonical stacks were first studied by Vistoli in \cite{Vistoli-Intersection_theory_on_algebraic_stacks_and_moduli_spaces} as a means of attaching a smooth Deligne--Mumford stack to a scheme with tame quotient singularities. Our primary reference is \cite[\S 4]{FantechiEtAl-Smooth_toric_DM_stacks}

\begin{dfn}
    Let $X$ be a smooth Deligne--Mumford stack with coarse moduli space $Y$. We say that $X$ is canonical if the locus where the map $X \to Y$ is not an isomorphism has codimension at least two. 
\end{dfn}
\begin{rmk}
    In \cite[Proposition 2.8, 2.9]{Vistoli-Intersection_theory_on_algebraic_stacks_and_moduli_spaces}, it was shown that any scheme of finite type over a field with tame quotient singularities is the coarse moduli space of a canonical stack. The question of uniqueness is resolved by \cite[Theorem 4.6]{FantechiEtAl-Smooth_toric_DM_stacks}, where canonical stacks are shown to be terminal with respect to dominant, codimension-preserving morphisms to their coarse moduli spaces. This property is stable under base change by \'etale morphisms, thus via descent every algebraic stack with tame quotient singularities has a canonical stack. We denote the canonical stack associated to such a stack $X$ by $X^{\can}$.
\end{rmk}

\begin{rmk}\label{rmk.smooth}
    As a byproduct of the definition of canonical stacks, one can check directly that if $Y^{\can}$ is a smooth canonical Deligne--Mumford stack with coarse moduli space $Y$, the locus where $\pi:Y^{\can} \to Y$ is an isomorphism is precisely $\pi^{-1}(Y_{sm})$, where $Y_{sm}$ is the smooth locus of $Y$. Clearly, if $Y$ is itself smooth, then $Y^{\can} \cong Y$ \cite[\S 4]{FantechiEtAl-Smooth_toric_DM_stacks}.
\end{rmk}

The main result concerning root and canonical stacks is a result of Geraschenko and Satriano, which shows that under suitable assumptions, a stack can be built from its coarse moduli space by repeating the canonical and root stack constructions.

\begin{thm}[{\cite[Theorem 1]{Geraschenko-Satriano17}}]\label{theorem_coursemoduli_is_rootstack_morphism}
    Let $X$ be a smooth separated tame Deligne--Mumford stack with trivial generic stabilizer. Let $Y$ be its coarse moduli space and $Y^{\can}$ the canonical stack associated to $Y$. Let $D \subset Y$ be the branch divisor of the map $\pi:X \to Y$, and $\mathcal{D}$ the pullback of $D$ to $Y^{\can}$. Let $r_i$ be the ramification index of $\pi$ over the irreducible component $\mathcal{D}_i$ of $\mathcal{D}$, then for $\mathbf{r}=(r_1,\ldots,r_n)$ and $\mathbf{D} = (\mathcal{D}_1,\ldots,\mathcal{D}_n)$, the stack $Y^{\can}[\sqrt[\mathbf{r}]{\mathbf{D}}]$ has tame quotient singularities and the map $\pi$ factors as $$X \cong \left( Y^{\can}\left[\sqrt[\mathbf{r}]{\mathbf{D}}\right] \right)^{\can} \to Y^{\can} \left[ \sqrt[\mathbf{r}]{\mathbf{D}} \right] \to Y^{\can} \to Y.$$
\end{thm}

For our purposes, the following corollary is more directly useful.

\begin{cor}[{{\cite[Corollary~5.6]{GS}}}]
  \label{corollary_quotient_stack_is_root_stack}
Suppose that $X$ is a smooth quasi-projective variety and $G$ a finite abelian group whose order is coprime to the characteristic of $k$. Then the induced map $$[X/G] \to (X/G)^{\can}$$ is a root stack morphism along a collection of smooth connected divisors with simple normal crossings.
\end{cor}

\subsection{Semiorthogonal Decompositions}
Before we state the main results needed regarding semiorthogonal decompositions, we first briefly remind the reader of the definition.

\begin{dfn}\label{def_semiorthogonal_decomp}
  Let $\mathcal{T}$ be a triangulated category. A semiorthogonal decomposition, written as $$\mathcal{T} = \< \mathcal{A}_1,\ldots, \mathcal{A}_n\>,$$ is a collection of full triangulated subcategories $\mathcal{A}_1,\ldots,\mathcal{A}_n$, referred to as the components of the decomposition, such that 
\begin{enumerate}
\item For any $T \in \mathcal{A}_i$ and $S \in \mathcal{A}_j$, if $i>j$ then $\Hom(T,S) = 0$.
\item for any $T \in \mathcal{T}$, there is a sequence of morphisms $$0 = F_n \to F_{n-1} \to \cdots \to F_1 \to F_0 =T$$ such that $\operatorname{Cone}(F_i \to F_{i-1}) \in \mathcal{A}_i$.
  \end{enumerate}
\end{dfn}

\begin{rmk}\label{remark_semiorthogonal_is_filtration}
One should think of a semiorthogonal decomposition as giving each object of a category a distinguished filtration, whose intermediate factors belong to each of the specified subcategories. Further, condition (1) implies that the intermediate factors are unique and functorial. 
\end{rmk}

\begin{rmk}\label{remark_exceptional_objects}
  When the subcategories appearing in \Cref{def_semiorthogonal_decomp} are as simple as possible, that is, $\mathcal{A}_i \cong D(k-\text{v.s.})$ the derived category of $k$-vector spaces, $\mathcal{A}_i$ is generated by a single object called an exceptional object and we call $\mathcal{A}_i$ exceptional.
  In general an exceptional object is an object $E$ in a triangulated category such that $$\operatorname{Hom}(E,E[i]) = \begin{cases} k & i=0 \\ 0 & i\neq 0. \end{cases}$$ When a triangulated category admits a semiorthogonal decomposition as above, and some of the $\mathcal{A}_i$ are exceptional %generated by exceptional objects
  we say that the collection of those $\{\mathcal{A}_i\}$  form an exceptional collection. 
If all $\mathcal{A}_i$ are exceptional, then we say they are a full exceptional collection.
\end{rmk}

Semiorthogonal decompositions are usually very interesting from the perspective of noncommutative geometry, but are typically difficult to establish. However, in a handful of cases, much is known, for example the well-known semiorthogonal decomposition of projective space \cite{Beilinson} (cf.~\cite[Corollary~8.29]{Huybrechts}):
\begin{equation}\label{Pn}D(\bP^n) = \< \cO,\cO(1),...,\cO(n)\>.\end{equation}
If $X \to Y$ is a blow-up with smooth center $Z \subset Y$ of codimension $c$, then it is a result of Orlov~\cite{Orlov} (cf.~\cite[Prop.~11.18]{Huybrechts}), that we have the following semiorthogonal decomposition of $X$:
\begin{equation}\label{orlov.blowup}
  D(X) = \langle D(Z)(1-c),\ldots,D(Z)(-1), D(Y)\rangle.
\end{equation}  

From the viewpoint of the derived category, the root stack along a divisor behaves much like a blow-up of a variety along a smooth center. Before we make this analogy precise, it will be useful for us to know explicitly the embedding functors appearing in the above semiorthogonal decomposition, so we devote some time to discuss them here. Additionally, although it is possible to weaken the assumptions, we only work with smooth Deligne--Mumford stacks for the rest of the paper. 

Let $X$ be a smooth Deligne--Mumford stack and $\mathbf{D} = (D_1,...,D_n)$ a collection of effective Cartier divisors and $\mathbf{r} \in \Z_{>0}^n$. Let $D$ be the union of the $D_i$, then one can think of the stack $\cD := D[\sqrt[\mathbf{r}]{\mathbf{D}}]$ as the ``strict transform'' of $D$ in the root stack. Indeed there is a commutative diagram
\begin{center}
    \begin{tikzcd}
      \cD \arrow[r,"j"] \arrow[d,"\pi_D"] & X[\sqrt[\mathbf{r}]{\mathbf{D}}] \arrow[d,"\pi_X"] \\
        D \arrow[r, "\overline{j}"] & X 
    \end{tikzcd}
\end{center}
where $\overline{j}$ is the closed embedding of the divisor $D$, and $j$ is the composition of the closed embedding $\cD \to X[\sqrt[\mathbf{r}]{\mathbf{D}}]|_D$ (see \cite[Appendix B]{AGV08}) and the inclusion $X[\sqrt[\mathbf{r}]{\mathbf{D}}]|_D \subset X[\sqrt[\mathbf{r}]{\mathbf{D}}]$. Now define the functors, for $\ell \in \Z$:
$$ \pi_X^*: D(X) \to D(X[\sqrt[\mathbf{r}]{\mathbf{D}}])$$ $$\Phi_\ell := \cO_{X}(\ell \mathcal{D}) \otimes  j_* \pi_D^*(-):  D(D) \to D(X[\sqrt[\mathbf{r}]{\mathbf{D}}]).$$ 

The following result was originally proven in \cite{BLS16} and \cite{IU15} independently. A version for iterated root stacks also appears in \cite{BLS16}, but 
we will not need its full statement here.

\begin{thm}[{\cite[Theorem 4.7]{BLS16}}, {\cite[Proposition 6.1]{IU15}}]\label{root.SOD}
    Let $X$ be a smooth Deligne--Mumford stack and $D \subset X$ an effective Cartier divisor. Fix a positive integer $r$, and let $\pi : X[\sqrt[r]{D}] \to X$ be the $r$th root stack along $D$. Then the functors $\pi_X^*$ and $\Phi_\ell$ are fully faithful and there is a semiorthogonal decomposition into admissible subcategories: $$D\left(X[\sqrt[r]{D}]\right) = \< \Phi_{r-1}(D(D)),\ldots, \Phi_{1}(D(D)), \pi_X^*D(X)\>.$$
\end{thm}

Since equivariant derived categories appear in the statement of our results we take a moment to note that they are equivalent to derived categories of quotient stacks.
Given a variety $X$ and a group $G$ acting on $X$, we have the following equivalence (see \cite[Exercise~9.H]{Olsson}):
$$D^G(X):=D([X/G]).$$
The reader interested in reading further about $G$-equivariant sheaves may wish to see, for example \cite[Section 4]{Bridgeland-King-Reid:2001}. 

In light of the result of the discussion above, notably \Cref{corollary_quotient_stack_is_root_stack} and \Cref{root.SOD}, given a finite abelian group acting on the variety $X$, we now have a prescription to obtain a semiorthogonal decomposition of the equivariant derived category on $X$, whose components consist of the derived categories of the irreducible components of the branching divisors $[X/G] \to (X/G)^{\can}$ as well as the derived category of the canonical stack $(X/G)^{\can}$. If $X$ is a quasi-projective surface, then in fact \cite[Theorem 1.4]{IU15} provides a way to decompose $D((X/G)^{\can})$ in terms of the minimal resolution of $X/G$ and a collection of exceptional objects coming from intersections of the branching divisors.
In our case we treat the action of $A:=\bZ/2$ on $Y:=\GHilb(\bC^2)$ and
by \Cref{YAsmooth}
the quotient $Y/A$ is smooth, thus we have $(Y/A)^{\can} \cong Y/A$ (see \Cref{rmk.smooth}).
These observations were first assembled by Potter in his thesis \cite{Pot18}, for convenience we include the full statement below.

\begin{thm}[{\cite[Corollary 6.1.2]{Pot18}}]
  Let $X$ be a quasi-projective surface over $k$ and $G$ a finite abelian group acting faithfully on $X$. Let $D = \sum D_i$ be the branch divisor of the coarse moduli space morphism $\pi:[X/G] \to X/G$, and $Y$ the minimal resolution of $X/G$. Then there is a semiorthogonal decomposition $$D^G(X) = \langle E_1,\ldots,E_k, \{D(D_1)\}_{i=1}^{r_1},\ldots,\{D(D_n)\}_{i=1}^{r_n}, D(Y)\rangle $$ consisting of $D(Y)$, multiple copies of $D(D_i)$, the number of which is the order of the stabilizer group of $D_i$, and exceptional objects arising from the intersections of $D_i$ and $D_j$ where the stabilizer group jumps and non-special representations of $G$ acting at an isolated point.
\end{thm}

\subsection{Strategy for the proof of Theorem A}\label{strategy}

We now outline how results on derived categories and semiorthogonal decompositions will be used to prove \refmain. 

Let $G$ be a reflection group as in \refmain. 

The first step we make toward a semiorthgonal decomposition of $D^G(\bC^2)$ is to use the following equivalence of categories, which is valid for any choice of nontrivial finite group $G\leq \GL(2,\bC)$ and the other notation is defined analogously to above:

\begin{thm}[{{\cite[Theorem~4.1]{IU15}}}]\label{zeroth}
$D^G(\C^2)\simeq D^A(Y)$
\end{thm}  

More strongly, the underlying categories of coherent sheaves are equivalent. 

The remainder of the strategy proceeds now as follows. By \Cref{corollary_quotient_stack_is_root_stack}, the map to the canonical stack $[Y/A]\to (Y/A)^{\can}$ is a root stack. Since $Y/A$ is a smooth quotient
(\Cref{YAsmooth}), as discussed in \Cref{rmk.smooth}, we may simply replace the canonical stack $(Y/A)^{\can}$ with the quotient $Y/A$.
The quotient $Y/A$
is a 2nd root stack along the branch divisor of $[Y/A]\to (Y/A)$.
In each case we treat, the branch divisor of $[Y/A]\to (Y/A)$
 consists of disjoint  curves $D_1,\ldots,D_n$.
Then \Cref{root.SOD} and \Cref{remark_iterated root stacks} implies that:
\begin{equation}\label{first}
D^A(Y)\simeq\langle D(D_1),\ldots, D(D_n), D(Y/A) \rangle.
\end{equation}

Finally, divisors in $Y/A$ can be blown down to obtain $\bC^2/G$
and thus we apply the blow-up formula \eqref{orlov.blowup} using $c=2$, obtaining a semiorthogonal decomposition of the following form:
\begin{equation}\label{second}
D(Y/A)\simeq\langle E_1,\ldots, E_m, D(\C^2/G) \rangle,
\end{equation}
where $E_1,\ldots, E_m$ are exceptional objects, one for each divisor we blow down. 

Thus, combining \eqref{first} and \eqref{second}, we obtain a semiorthogonal decomposition of the form given in \refmain. We note that the (subcategories generated by) exceptional objects and subcategories generated by branch divisors, since they are embedded via Fourier--Mukai transforms, are admissible and thus can be permuted in any order using mutations; see \cite{BK}, \cite[Section~2]{Bondal}.

In order to also count the number of exceptional objects present in this semiorthogonal decomposition, we must identify, for each $G$, the fixed locus of the action of $A$ on $Y$ and the divisors in $Y/A$ that will be blown down to obtain $\C^2/G$. We address these details in the next section.

\section{Proof of Theorem A}\label{sec:proof}

We now combine the strategy outlined in \Cref{strategy} with the computational results shown in the appendix to prove \refmain.
In Sections \ref{sec.2mm2}-\ref{sec.g22}, we identify the components of the semiorthogonal decomposition of $\dbg(\C^2)$, treating each reflection group that appears in \refmain{} separately.
We illustrate these cases of the proof with Figures~\ref{fig_three}-\ref{fig_seven}, which show the exceptional locus of $\hhilbc$; the red curves indicate where the discriminant curve intersects the exceptional locus and dashed arrows indicate a non-identity action of $A$.
We also discuss the $G(m,m,2)$ case in \Cref{sec.gmm2} for the reader's convenience. Then in \Cref{OSOD}, we conclude by explaining why the components in the semiorthogonal decomposition of $\dbg(\C^2)$,
are in bijection with the conjugacy classes of $G$, verifying the Orbifold Semiorthogonal Decomposition Conjecture.

\subsection{$G(2m,m,2)$}\label{sec.2mm2}

\begin{figure}[h!]
  \begin{subfigure}[h]{0.99\textwidth}
    \centering 
    \begin{tikzpicture}[scale=3]
      % Exceptional curves and labels
      \draw(-2,0) to [in=160, out=20] (-1,0); %1st curve
      \node at (-2.2,0) [above] {$E(\rho_{m-1})$}; %label for 1st curve
      %\node at (-1.3,0.21) {$E(\rho_{m-1})$}; %label for 1st curve
      \draw(-1.25,0) to [in=160, out=20] node [midway, above] {$E(\rho_{m-2})$} (-0.25,0); %2nd curve
      \node at (0,0) {$\cdots$};
      \draw(0.25,0) to [in=160, out=20] node [midway, above] {$E(\rho_2)$} (1.25,0); %3rd curve
      \draw(1,0) to [in=160, out=20]  node [midway, above] {$E(\rho_1')$} (2,0); %4th curve
      \draw(-1.5,-.35) to [in=250, out=110] node [at end, above right] {$E(\rho_m)$} (-1.5,.65); %5th (vertical) curve
      \draw(-2,.65) to [in=180, out=320] (-1,.35); %6th curve
      \node at (-2,.65) [above] {$E(\rho_{m+1}')$}; %label for 6th curve
      \draw(-2,-.35) to [in=160, out=30]  (-1,-.25); %7th curve
      \node at (-2.2,-.35) [below] {$E(\rho_{m+2}')$}; %label for 7th curve

      % Discriminant curve
      \draw [ultra thick, red] (1.8,-0.5) to [in=250, out=110] (1.8,0.5); %vertical disc. component
      \draw [ultra thick, red] (-2,0.3) to [out=30, in=-90] (-1.7,0.8); %disc. component intersecting 6th curve
      \draw [ultra thick, red] (-2,-0.25) to [out=-30, in=90] (-1.8, -0.75); %disc. component intersecting 7th curve
  
      % Arrows indicating action of Z/2Z (these are smaller than in the other cases)
      \draw[->,dashed] (-1.8,-.01) to [in=320,out=220,looseness=5] (-1.7,-.01);
      \draw[->,dashed] (1.5,-.01) to [in=320,out=220,looseness=5] (1.6,-.01);
      \draw[->,dashed] (-1.8,0.44) to [in=320,out=220,looseness=5] (-1.7,0.4);
      \draw[->,dashed] (-1.8,-.35) to [in=320,out=220,looseness=5] (-1.7,-.36);
    \end{tikzpicture}
    \subcaption{Exceptional locus in the $G(2m,m,2)$ ($m$ even) case.}
  \end{subfigure}

  \begin{subfigure}[h]{0.99\textwidth}
    \centering 
    \begin{tikzpicture}[scale=3]
      % Exceptional curves and labels
      \draw(-2,0) to [in=160, out=20] (-1,0); %1st curve
      \node at (-2.2,0) [above] {$E(\rho_{m-1})$}; %label for 1st curve
      %\node at (-1.3,0.21) {$E(\rho_{m-1})$}; %label for 1st curve
      \draw(-1.25,0) to [in=160, out=20] node [midway, above] {$E(\rho_{m-2})$} (-0.25,0); %2nd curve
      \node at (0,0) {$\cdots$};
      \draw(0.25,0) to [in=160, out=20] node [midway, above] {$E(\rho_2)$} (1.25,0); %3rd curve
      \draw(1,0) to [in=160, out=20]  node [midway, above] {$E(\rho_1')$} (2,0); %4th curve
      \draw(-1.5,-.35) to [in=250, out=110] node [at end, above right] {$E(\rho_m)$} (-1.5,.65); %5th (vertical) curve
      \draw(-2,.65) to [in=180, out=320] (-1,.35); %6th 
      \node at (-2,.65) [above] {$E(\rho_{m+1}')$}; %label for 6th
      \draw(-2,-.35) to [in=160, out=30]  (-1,-.25); %7th 
      \node at (-2.2,-.35) [below] {$E(\rho_{m+2}')$}; %label for 6th

      % Discriminant curve
      \draw [ultra thick, red] (1.8,-0.5) to [in=250, out=110] (1.8,0.5); %vertical disc. component
      \draw [ultra thick, red] (-2,0.3) to [out=30, in=140] (-1.3,0.2); %disc. component intersecting 5th curve
  
      % Arrows indicating action of Z/2Z (these are smaller than in the other cases)
      \draw[->,dashed] (-0.8,-.01) to [in=320,out=220,looseness=5] (-0.7,-.01);
      \draw[->,dashed] (1.5,-.01) to [in=320,out=220,looseness=5] (1.6,-.01);
      \draw[->,dashed] (-1.55,0.05) to [in=50,out=320,looseness=5] (-1.55, 0.15);
      \draw[<->,dashed] (-2.2, 0.65) to [in=140,out=220,looseness=1.5] (-2.2,-0.35);
    \end{tikzpicture}
    \subcaption{Exceptional locus in the $G(2m,m,2)$ ($m$ odd) case. }
  \end{subfigure}
  \caption{$G(2m,m,2)$ cases}\label{fig_three}
\end{figure}

Let $m$ be even. By \Cref{prop.Dn}(c),(e) and \Cref{disc.table} (cf.~\Cref{rmk.puv}), the fixed locus of the action of $A$ on $Y$ consists of $\frac{m}{2}$ exceptional curves $E(\rho_2),E(\rho_4),\dots, E(\rho_m)$, as well as three curves that extend outside the exceptional locus.
We claim that the images of $E(\rho_2),E(\rho_4),\dots, E(\rho_m)$ in $Y/A$ are also copies of $\bP^1$: by~\Cref{corollary_quotient_stack_is_root_stack} their images are smooth, and so we may conclude the claim by Riemann--Hurwitz. Here and in the other cases below, we abuse notation and give the fixed exceptional curves in $Y$ and their images in $Y/A$ the same name.
Let $B_1$, $B_2$, and $B_3$ be 
the images in $Y/A$ of the three curves  that extend outside the exceptional locus.
By \Cref{smooth_lemma}, $B_1$, $B_2$, and $B_3$ are affine lines.

Thus by \eqref{first},
\[D^A(Y)\simeq\langle
D(B_1),D(B_2), D(B_3),
D(E(\rho_2)),D(E(\rho_4)),\dots, D(E(\rho_m)),  
D(Y/A)
\rangle.\]
Since each $E(\rho_k)\simeq \bP^1$, by \eqref{Pn}, there exists a semiorthogonal decomposition as follows, where the $E_i$ are exceptional objects:
\begin{equation}\label{above.1}
  D^A(Y)\simeq\langle
D(B_1),D(B_2), D(B_3), 
  E_1,\ldots, E_{m},
  D(Y/A)
  \rangle.
\end{equation}  
By \Cref{prop.Dn}(a), $A$ acts as an automorphism on each of the $m+2$ exceptional curves in $Y$ and therefore $Y/A$ contains $m+2$ curves that must be blown down to obtain $\bC^2/G$.
Combining \eqref{second} and \Cref{zeroth} with \eqref{above.1}, we have the following semiorthogonal decomposition, where the $E_i$ are exceptional objects:
\[
  D^G(\bC^2)\simeq\langle
D(B_1),D(B_2), D(B_3), 
  E_1,\ldots, E_{2m+2},
  D(\bC^2/G)
  \rangle.
\]
We see in \Cref{Dn.dynkin} that $G$ has $2m+6$ irreducible representations, and thus we have proved \refmain{} in this case.

Let $m$ be odd. By \Cref{prop.Dn}(c),(e) and \Cref{disc.table}, the fixed locus of the action of $A$ on $Y$ consists of $\frac{m-1}{2}$ exceptional curves
$E(\rho_2),E(\rho_4),\dots, E(\rho_{m-1})$ and two curves that extend outside the exceptional locus of $Y$ whose images $B_1$, $B_2$ in $Y/A$ are affine lines. 

By \Cref{prop.Dn}(b), $A$ acts as an automorphism on each of the $m$ exceptional curves
$E(\rho_1)$, \ldots, $E(\rho_m)$ and exchanges $E(\rho_{m+1})$ with $E(\rho_{m+2})$.
Therefore $Y/A$ contains $m+1$ curves that must be blown down to obtain $\bC^2/G$.

Applying \Cref{zeroth}, \eqref{first}, and \eqref{second}, we conclude that
there exists a semiorthogonal decomposition of the following form, where the $E_i$ are exceptional objects:
\[D^G(\bC^2)\simeq\langle
  D(B_1), D(B_2),
  E_1,\ldots, E_{2m},  
  D(\bC^2/G)  \rangle.\]
We see in \Cref{Dn.dynkin} that $G$ has $2m+3$ irreducible representations, proving \refmain{} in this case.

\subsection{$\gtw$}

\begin{figure}[h!]
  \centering 
  \begin{tikzpicture}[scale=3]
    % Exceptional curves and labels
    \draw(-2,0) to [in=160, out=20] node [midway, above] {$E(\rho_1')$} (-1,0); %1st curve
    \draw(-1.25,0) to [in=160, out=20] node [midway, above] {$E(\rho_2')$} (-0.25,0); %2nd curve
    \draw(-0.5,0) to [in=160, out=20] node [midway, above, xshift=0.3cm] {$E(\rho_3)$} (0.5,0); %3rd curve
    \draw(0.25,0) to [in=160, out=20] node [midway, above] {$E(\rho_2'')$} (1.25,0); %4th curve
    \draw(1,0) to [in=160, out=20] node [midway, above] {$E(\rho_1'')$} (2,0); %5th curve
    \draw(0,-.05) to [in=250, out=110] node [near end, above left] {$E(\rho_2)$} (0,.75); %6th (vertical) curve
  
    % Discriminant curve
    \draw [ultra thick, red] (-.5,.5) to [in=110, out=270] (-0.1,-0.025);
  
    % Arrows indicating action of Z/2Z
    \draw[<->,dashed] (-1.5,-.05) to [in=250,out=290] (1.5,-.05); %outer pair
    \draw[<->,dashed] (-.75,-.05) to [in=250,out=290] (.75,-.05); %inner pair
    \draw[->,dashed] (-.05,-.05) to [in=320,out=220,looseness=10] (.05,-.05); %center curve
  \end{tikzpicture}
  \caption{Exceptional locus in the $G_{12}$ case. }\label{fig_four}
\end{figure}
By \Cref{G12_A_summary}(c) and \Cref{disc.table}, the fixed locus of the action of $A$ on $Y$ consists of one exceptional curve $E(\rho_2)$ and
one curve extending outside the exceptional locus of $Y$, whose image, $B$, in $Y/A$ is an affine line.

By \Cref{G12_A_summary}(a),(b), $A$ acts as an automorphism on two exceptional curves
$E(\rho_2)$, $E(\rho_3)$ and exchanges two pairs of exceptional curves.
Therefore $Y/A$ contains $4$ curves that must be blown down to obtain $\bC^2/G$.

Applying \Cref{zeroth}, \eqref{first}, and \eqref{second}, we conclude that
there exists a semiorthogonal decomposition of the following form, where the $E_i$ are exceptional objects:
\[D^G(\bC^2)\simeq\langle
D(B),E_1,\ldots, E_6, 
  D(\bC^2/G) \rangle.\] 
We see in \Cref{E6.dynkin} that $G$ has $8$ irreducible representations, proving \refmain.

\subsection{\texorpdfstring{$G_{13}$}{G13}}

\begin{figure}[h!]
  \centering
  \begin{tikzpicture}[scale=3]
    % Exceptional curves and labels
    \draw(-2,0) to [in=160, out=20] node [midway, above] {$E(\rho_2)$} (-1,0); %1st curve
    \draw(-1.25,0) to [in=160, out=20] node [midway, above] {$E(\rho_3)$} (-0.25,0); %2nd curve
    \draw(-0.5,0) to [in=160, out=20] node [midway, above, xshift=0.3cm] {$E(\rho_4)$} (0.5,0); %3rd curve
    \draw(0.25,0) to [in=160, out=20] node [midway, above] {$E(\rho_3')$} (1.25,0); %4th curve
    \draw(1,0) to [in=160, out=20] node [midway, above] {$E(\rho_2')$} (2,0); %5th curve
    \draw(1.75,0) to [in=160, out=20] node [midway, above] {$E(\rho_1')$} (2.75,0); %6th curve
    \draw(0,-.05) to [in=250, out=110] node [near end, above left] {$E(\rho_2'')$} (0,.75); %7th (vertical) curve

    % Discriminant curve (Is this in the right place?)
    \draw [ultra thick, red] (-.5,.5) to [in=210, out=330] (0.5,.5); %disc. component intersecting 1st curve
    \draw [ultra thick, red] (2.5,-0.5) to [in=250, out=110] (2.5,0.5); %vertical disc. component

    % Arrows indicating action of Z/2Z
    \draw[->,dashed] (-.8,-.05) to [in=320,out=220,looseness=10] (-.7,-.05);
    \draw[->,dashed] (.7,-.05) to [in=320,out=220,looseness=10] (.8,-.05);
    \draw[->,dashed] (2.2,-.05) to [in=320,out=220,looseness=10] (2.3,-.05);
    \draw[->,dashed] (-0.05,0.8) to [in=40,out=140,looseness=10] (0.05,0.8);
  \end{tikzpicture}
  \caption{Exceptional locus in the $G_{13}$ case.}\label{fig_five}
\end{figure}

By \Cref{G13_A_summary}(b),(c) and \Cref{disc.table},
the fixed locus of the action of $A$ on $Y$ consists of the three exceptional curves $E(\rho_2)$, $E(\rho'_2)$ and $E(\rho_4)$ 
as well as two curves
extending outside the exceptional locus of $Y$ whose images $B_1$ and $B_2$ in $Y/A$ are affine lines.

By \Cref{G13_A_summary}(a), $G_{13}$ acts as an automorphism on each of the seven exceptional curves, and therefore $Y/A$ contains $7$ curves that must be blown down to obtain $\bC^2/G$.

Applying  \Cref{zeroth}, \eqref{first}, and \eqref{second}, we conclude that
there exists a semiorthogonal decomposition of the following form, where the $E_i$ are exceptional objects:
\[D^G(\bC^2)\simeq\langle D(B_1), D(B_2), 
E_1,\ldots, E_{13}, 
  D(\bC^2/G)
  \rangle.\] 
We see in \Cref{E7.dynkin} that $G$ has $16$ irreducible representations, proving \refmain.

\subsection{\texorpdfstring{$G_{22}$}{G22}}\label{sec.g22}

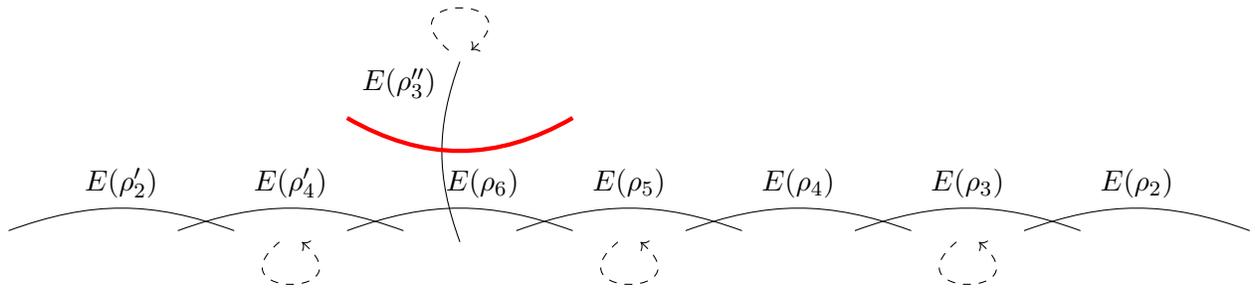
\begin{figure}[h!]
  \centering
  \begin{tikzpicture}[scale=3]
    % Exceptional curves and labels
    \draw(-2,0) to [in=160, out=20] node [midway, above] {$E(\rho_2')$} (-1,0); %1st curve
    \draw(-1.25,0) to [in=160, out=20] node [midway, above] {$E(\rho_4')$} (-0.25,0); %2nd curve
    \draw(-0.5,0) to [in=160, out=20] node [midway, above, xshift=0.3cm] {$E(\rho_6)$} (0.5,0); %3rd curve
    \draw(0.25,0) to [in=160, out=20] node [midway, above] {$E(\rho_5)$} (1.25,0); %4th curve
    \draw(1,0) to [in=160, out=20] node [midway, above] {$E(\rho_4)$} (2,0); %5th curve
    \draw(1.75,0) to [in=160, out=20] node [midway, above] {$E(\rho_3)$} (2.75,0); %6th curve
    \draw(2.5,0) to [in=160, out=20] node [midway, above] {$E(\rho_2)$} (3.5,0); %7th curve
    \draw(0,-.05) to [in=250, out=110] node [near end, above left] {$E(\rho_3'')$} (0,.75); %8th (vertical) curve

    % Discriminant curve
    \draw [ultra thick, red] (-.5,.5) to [in=210, out=330] (0.5,.5);

    % Arrows indicating action of Z/2Z
    \draw[->,dashed] (-.8,-.05) to [in=320,out=220,looseness=10] (-.7,-.05);
    \draw[->,dashed] (.7,-.05) to [in=320,out=220,looseness=10] (.8,-.05);
    \draw[->,dashed] (2.2,-.05) to [in=320,out=220,looseness=10] (2.3,-.05);
    \draw[->,dashed] (-0.05,0.8) to [in=40,out=140,looseness=10] (0.05,0.8);
  \end{tikzpicture}
  \caption{Exceptional locus in the $G_{22}$ case.}\label{fig_six}
\end{figure}

By \Cref{G22_A_summary}(b),(c) and \Cref{disc.table},
the fixed locus of the action of $A$ on $Y$ consists of the exceptional curves $E(\rho_2)$, $E(\rho'_2)$, $E(\rho_4)$, and $E(\rho_6)$
as well as one curve
extending outside the exceptional locus of $Y$ whose image $B$ in $Y/A$ is an affine line.

By \Cref{G13_A_summary}(a), $G_{22}$ acts as an automorphism on each exceptional curve, and therefore $Y/A$ contains $8$ curves that must be blown down to obtain $\bC^2/G$.

Applying  \Cref{zeroth}, \eqref{first}, and \eqref{second}, we conclude that
there exists a semiorthogonal decomposition of the following form, where the $E_i$ are exceptional objects:
\[D^G(\bC^2)\simeq\langle
D(B),E_1,\ldots, E_{16},
D(\bC^2/G)
\rangle.\] 
We see in \Cref{E8.dynkin} that $G$ has $18$ irreducible representations, proving \refmain.

\subsection{\texorpdfstring{$G(m,m,2)$}{G(m,m,2)}} \label{sec.gmm2}

\begin{figure}[h!]
  \begin{subfigure}[h]{0.99\textwidth}
    \centering 
    \begin{tikzpicture}[scale=3]
      % Exceptional curves and labels
      \draw(-2.75,0) to [in=160, out=20] node [midway, above] {$E(\rho_1)$} (-1.75,0); %1st curve
      \node at (-1.5,0) {$\cdots$};
      \draw(-1.25,0) to [in=160, out=20] node [midway, above] {$E(\rho_{\frac{m}{2}-1})$} (-0.25,0); %2nd curve
      \draw(-0.5,0) to [in=160, out=20] node [midway, above] {$E(\rho_{\frac{m}{2}})$} (0.5,0); %middle curve
      \draw(0.25,0) to [in=160, out=20] node [midway, above] {$E(\rho_{\frac{m}{2}+1})$} (1.25,0); %4th curve
      \node at (1.5,0) {$\cdots$};
      \draw(1.75,0) to [in=160, out=20] node [midway, above] {$E(\rho_{m-1})$} (2.75,0); %5th curve
  
      % Discriminant curve
      \draw [ultra thick, red] (-.5,.5) to [in=110, out=270] (-0.1,-0.025);
      \draw [ultra thick, red] (0.5,.5) to [in=70, out=270] (0.1,-0.025);
  
      % Arrows indicating action of Z/2Z
      \draw[<->,dashed] (-2.25,-.05) to [in=250,out=290] (2.25,-.05); %outer pair
      \draw[<->,dashed] (-.75,-.05) to [in=250,out=290] (.75,-.05); %inner pair
      \draw[->,dashed] (-.05,-.05) to [in=320,out=220,looseness=10] (.05,-.05); %center curve
    \end{tikzpicture}
    \subcaption{Exceptional locus in the $G(m,m,2)$ ($m$ even) case.}
  \end{subfigure}
  \begin{subfigure}[h]{0.99\textwidth}
    \centering 
    \begin{tikzpicture}[scale=3]
      % Exceptional curves and labels
      \draw(-2.375,0) to [in=160, out=20] node [midway, above] {$E(\rho_1)$} (-1.375,0); %1st curve
      \node at (-1.125,0) {$\cdots$};
      \draw(-0.875,0) to [in=160, out=20] node [midway, above] {$E(\rho_{\frac{m-1}{2}})$} (0.125,0); %2nd curve
      \draw(-0.125,0) to [in=160, out=20] node [midway, above] {$E(\rho_{\frac{m+1}{2}})$} (0.875,0); %3rd curve
      \node at (1.125,0) {$\cdots$};
      \draw(1.375,0) to [in=160, out=20] node [midway, above] {$E(\rho_{m-1})$} (2.375,0); %4th curve
  
      % Discriminant curve
      \draw [ultra thick, red] (0,.5) to (0,-0.2);
  
      % Arrows indicating action of Z/2Z
      \draw[<->,dashed] (-1.825,-.05) to [in=250,out=290] (1.825,-.05); %outer pair
      \draw[<->,dashed] (-.375,-.05) to [in=250,out=290] (.375,-.05); %inner pair
    \end{tikzpicture}
    \subcaption{Exceptional locus in the $G(m,m,2)$ ($m$ odd) case.}
  \end{subfigure}
  \caption{$G(m,m,2)$ cases}\label{fig_seven}
\end{figure}

Finally, for the reader's convenience, we will discuss the proof of \refmain{} for the group $G(m,m,2)$, which was shown in \cite[\S6.5]{Pot18}.

Using explicit charts for the toric minimal resolution of the singularity $A_{m-1}$, Potter computes that 
the fixed locus of the action of $A$ on $Y$ consists only of
the strict transform of the ramification locus $\C^2/H\to \C^2/G$.
This lies over the branch locus $V(z_1^2-z_2^m)$ in $\C^2/G$
and thus has one component $B$ when $m$ is odd and two components, $B_1$ and $B_2$, when $m$ is even.
He furthermore shows that when $m$ is odd, $A$ interchanges pairs of exceptional curves, leaving $\frac{m-1}{2}$ curves in $Y/A$ that must be blown down to obtain $\bC^2/G$. When $m$ is even, $A$ interchanges $\frac{m}{2}-1$ pairs of exceptional curves and acts as a nonidentity involution on $E(\rho_{m/2})$, implying there are
$\frac{m}{2}$ curves in $Y/A$ that must be blown down to obtain $\bC^2/G$.
The curves that are interchanged with one another correspond to the representations directly above and below one another in the Dynkin diagram on the left-hand side of \Cref{An.dynkin}.

Thus,
there exist semiorthogonal decompositions of the following form, where the $E_i$ are exceptional objects:
\begin{align*}
  D^G(\bC^2)&\simeq\langle D(B),  
E_1,\ldots, E_{\frac{m-1}{2}},              D(\bC^2/G)
              \rangle \quad\textrm{($m$ odd),}
  \\
  D^G(\bC^2)&\simeq\langle D(B_1), D(B_2),
E_1,\ldots, E_{\frac{m}{2}}, 
              D(\bC^2/G) \rangle \quad\textrm{($m$ even).}  
\end{align*}
It is illustrated in \Cref{An.dynkin} that $G(m,m,2)$ has $\frac{m+1}{2}$ nontrivial irreducible representations when $m$ is odd and $\frac{m}{2}+2$ when $m$ is even.

\begin{rmk} \label{rmk.proof}
In each case of the proof of \refmain, the number of exceptional objects we add in the first step of the proof, where we use \eqref{first}, is always equal to the number of representations of $G$ that contain an $\ep$ in our Dynkin diagrams. In each case, whether $A$
acts as an involution on exceptional curves or
exchanges them corresponds precisely to whether the corresponding representation is  or is not isomorphic to its contragredient.
\end{rmk}

\subsection{Proof of Corollary~B}\label{OSOD}
As we noted in the introduction, the semiorthogonal decompositions we construct are predicted by Polishchuck and Van den Bergh's Orbifold Semiorthogonal Decomposition Conjecture. We have already shown that the
number of components 
of our semiorthogonal decomposition of $D^G(\bC^2)$
is equal to
the irreducible representations and thus conjugacy classes of $G$, so it remains to show that the components are isomorphic to $D((\bC^2)^g/C(g))$ for each conjugacy class $[g]$, where $C(g)$ is the centralizer of $G$.

We will show that these components are isomorphic, but we will not produce explicit isomorphisms.
For a more explicit approach to constructing this semiorthogonal decomposition and showing it satisfies the conjecture of Polischuck--Van den Bergh, see \cite[Section 3]{Lim-Rota:2024} which covers the details for $G(4,2,2)$.

By \refmain, each semiorthogonal decomposition we construct consists of one copy of $D(\bC^2/G)$, a copy of $D(\bA^1)$ for each of the $r$ components of the branch divisor of $\bC^2\to\bC^2/G$, and $n$ exceptional objects each generating a subcategory isomorphic to $D(\Spec\bC)$, such that $n+r+1$ equals the number of representations of $G$.

In parallel, each group $G$ has three types of conjugacy class: $[I]$ where $I$ is the identity matrix, $[s]$ where $s$ is a reflection, and $[g]$ for all remaining non-identity, non-reflection elements $g\in G$. We have the following isomorphisms:
\[
D((\bC^2)^I/C(I))\simeq D(\bC^2/G),\quad
D((\bC^2)^s/C(s))\simeq D(\bA^1),\quad
D((\bC^2)^g/C(g))\simeq D(\Spec\bC).
\]
Thus we just need to show that the conjugacy classes of reflections are in bijection with the components of the branch divisor of $\bC^2/H\to\bC^2/G$, which is the image of the hyperplanes fixed by reflections in $G$.
Suppose $gs_1g^{-1}=s_2$, where $s_1$ and $s_2$ are reflections fixing hyperplanes $H_1$ and $H_2$. Then $gH_1=H_2$, so $H_1$ and $H_2$ map to the same component of the branch divisor of $\bC^2\to\bC^2/G$ and we are done.

\appendix
\section{Computing the action of \texorpdfstring{$A$}{A} on \texorpdfstring{$\hhilbc$}{H-HilbC}}\label{appendix}

A key component of our proofs of \refmain{} for each of the reflection groups $G$ is our computation of the fixed points of the action of $A$ 
on the exceptional locus of $Y:=\hhilbc$. 
Ito and Nakamura have given explicit algebraic descriptions of the points of $Y$ for each $H\leq \SL(2,\C)$ \cite{Ito-Nakamura:1999}. To perform our computations in this setting, we find embeddings $G\leq \GL(2,\C)$ that are extensions of those given in \cite{Ito-Nakamura:1999}, and in particular we choose an element $\alpha\in G\setminus H$
that we use to explicitly compute the action of $A$.
Ito and Nakamura give the exceptional curves of $Y$ and their intersections in terms of certain $H$-invariant submodules, which in some cases we must identify precisely and parametrize in order to determine the fixed points of the action. In this appendix, we show these computational details and summarize our findings, handling the different groups $G$ in each section.

Some computations discussed in sections~\ref{gtw.comp},~\ref{app.gthir}, and~\ref{app.gtwtw} are verified using \texttt{macaulay2} in the code file \texttt{FixedLocus.m2}, which is available at  \cite{code} or as an ancillary file with the arxiv preprint.

\subsection{}\label{AIsec}

For the groups $H\leq \SL(2,\C)$ we treat in this appendix, 
Ito and Nakamura identify the ideals corresponding to points in the exceptional locus of $Y$ as certain submodules of $\fm/\fn$
where $\fm$ is the maximal ideal of the origin in $\bC^2$ and $\C[x,y]^H\simeq\fn\subseteq\fm\subseteq\C[x,y]$ is the ideal generated by $H$-invariant polynomials.

For each irreducible representation $\rho$ of $H$, they define 
the modules $V_i(\rho)$:
\[V_i(\rho):= \rho\text{-summands of the homogenous degree } i \text{ part of } \fm/\fn.\]
Their theorem below particularly refers to modules $V_{h\pm d(\rho)}(\rho)$, where $h$ is the Coxeter number of $H$, listed in the table below, and $d(\rho)$ is the distance of $\rho$ from the representation at the center in the Dynkin diagram of $H$. For $H$ of type $D_{m+2}$, $E_6$, $E_7$, or $E_8$, the Dynkin diagram is star-shaped with a unique center; these centers are labelled in the Dynkin diagrams in \Cref{appendixb} as $\rho_m$, $\rho_3$, $\rho_4$, and $\rho_6$, respectively.

\renewcommand{\arraystretch}{1.5}
\begin{center}
\begin{tabular}{c|ccccc}
Group $G$ & $G(m,m,2)$ & $G(2m,m,2)$ & $G_{12}$ & $G_{13}$ & $G_{22}$
  \\\hline
Coxeter number of $H$& $m$ & $2m+2$ & 12 & 18 & 30
\end{tabular}
\end{center}

By \cite[Theorem~10.6]{Ito-Nakamura:1999}, we have that $V_{\frac{h}{2}}(\rho)\simeq\rho^2$ for $\rho$ the center of the Dynkin diagram and $V_{\frac{h}{2}-d(\rho)}(\rho)\simeq V_{\frac{h}{2}+d(\rho)}(\rho)\simeq\rho$ otherwise. The structure of the exceptional set for groups of type $D_n$, $E_6$, $E_7$, and $E_8$ is then given by:

\begin{thm}[{{\cite[Theorem 10.7]{Ito-Nakamura:1999}}}]\label{IN.main}
Let $H\leq \SL(2,\C)$ a finite group such that $\bC^2/H$ is of type $D_n$, $E_6$, $E_7$, or $E_8$. Then:
\begin{enumerate}
\item Assume $\rho$ is an endpoint of the Dynkin diagram. Then $I$ is a point on $E(\rho)$ and no other exceptional component if and only if $V(I)$ is a nonzero irreducible $H$-submodule ($\simeq\rho$) of $V_{\frac{h}{2}-d(\rho)}(\rho)\oplus V_{\frac{h}{2}+d(\rho)}(\rho)$ different from $V_{\frac{h}{2}+d(\rho)}(\rho)$
\item Assume $\rho$ is not an endpoint or center of the Dynkin diagram. Then $I$ is a point on $E(\rho)$ and no other exceptional component if and only if $V(I)$ is a nonzero irreducible $H$-submodule ($\simeq\rho$) of $V_{\frac{h}{2}-d(\rho)}(\rho)\oplus V_{\frac{h}{2}+d(\rho)}(\rho)$ different from $V_{\frac{h}{2}-d(\rho)}(\rho)$ and $V_{\frac{h}{2}+d(\rho)}(\rho)$
\item Assume $\rho$ is the center of the Dynkin diagram. Then $I$ is a point on $E(\rho)$ and no other exceptional component if and only if $V(I)$ is a nonzero irreducible component of $V_{\frac{h}{2}}(\rho)$ different from $\{S_1\cdot V_{\frac{h}{2}-1}(\rho')\}[\rho]$ for any $\rho'$ adjacent to $\rho$
\item Assume $\rho$ and $\rho'$ are adjacent with $d(\rho')=d(\rho)+1\geq2$ (so neither are the center). Then $I\in P(\rho,\rho')$ if and only if 
\[V(I)=V_{\frac{h}{2}-d(\rho)}(\rho)\oplus V_{\frac{h}{2}+d(\rho')}(\rho')\]
\item Assume $\rho$ is the center of the Dynkin diagram and $\rho'$ is adjacent to $\rho$. Then $I\in P(\rho,\rho')$ if and only if
\[V(I)=\{S_1\cdot V_{\frac{h}{2}-1}(\rho')\}[\rho]\oplus V_{\frac{h}{2}+1}(\rho').\]
\end{enumerate}
Here $S_1\cdot V_i(\rho)$ means the $H$-submodule of $\fm/\fn$ generated by $(x,y)$ and $V_i(\rho)$, and $W[\rho]$ refers to the $\rho$-summands of $W$, a homogenous $H$-submodule of $\fm/\fn$.
\end{thm}

\subsection{\texorpdfstring{$G(2m,m,2)$}{G(2m,m,2)}}\label{appendixdn}
\numberwithin{equation}{subsection}

Let $m$ be an integer greater than or equal to $3$ and let $G:=G(2m,m,2)$.
The singularity $\C^2/H$ is of type $D_{m+2}$.
A presentation of $G$ is given in \cite[p.~36]{LehrerTaylor}, where $G$ is generated by
\begin{equation}\label{G.Dm.gens}
r_1=\begin{pmatrix} 0&1\\1&0\end{pmatrix},\quad t^m=\begin{pmatrix} -1 & 0  \\ 0 & 1 \end{pmatrix},\quad\text{and}\quad s=\begin{pmatrix}0&\varepsilon^{-1}\\ \varepsilon&0\end{pmatrix}.
\end{equation}
A presentation of $H$ is given in \cite[Ch. 13]{Ito-Nakamura:1999}, where $H$ is generated by
\[
\sigma=\begin{pmatrix} \varepsilon & 0 \\ 0 & \varepsilon^{-1} \end{pmatrix},\quad \tau=\begin{pmatrix} 0 & 1 \\ -1 & 0 \end{pmatrix},
\]
where $\varepsilon:=e^{\pi i / m}$ is an $2m$-th root of unity. Note that these presentations of $H$ and $G$ agree because $\sigma=r_1s$ and $\tau=r_1t^m$.

In order to construct our semiorthogonal decomposition, we need to know the ramification locus of $\C^2/H\to C^2/G$ as well as the $A$-action on $Y$.
In particular, we need to understand which points in $Y$ are fixed by $A$
and
which exceptional components of $Y$ are fixed or exchanged by $A$, allowing us to deduce how many exceptional components are in $Y/A$.
%when two exceptional components in $Y$ are exchanged by $A$ how many exceptional components in $Y$ lie over in $H$-$\Hilb(\C^2)/A$.

In this section, we will compute the following information about the action of $A$, which we summarize in a proposition. The module $V_{m+1}(\rho_m)$ consists of two $\rho_m$-summands, which are denoted by $V_{m+1}'(\rho_m)$ and
$V_{m+1}''(\rho_m)$.

\begin{prop}\label{prop.Dn} Let $G:=G(2m,m,2)$. The following statements describe the action of $A$ on the exceptional locus of $Y$.
\begin{enumerate}[(a)]
\item When $m$ is even, the action of $A$ maps each of the $m+2$ exceptional curves to itself. 
\item When $m$ is odd, 
  $A$ interchanges $E(\rho_{m+1}')$ and $E(\rho_{m+2}')$, and maps each of the other $m$ exceptional curves to themselves.
\item When $m$ is odd or even, $A$ fixes $E(\rho_k)$ pointwise for each even value of $k$ with $2\leq k\leq m$.
\item When $m$ is odd or even, $A$ fixes exactly two points of $E(\rho_k)$ for each odd value of $k$ with $1\leq k\leq m$. When $k\neq1,m$, these are the points of intersection with the adjacent exceptional curves. When $k=1$, one fixed point is the intersection of $E(\rho_1)$ with $E(\rho_2)$ and when $k=m$ is odd one fixed point is the intersection of $E(\rho_m)$ with $E(\rho_{m-1})$.
\item Within the exceptional locus, we have the following isolated fixed points:
\begin{enumerate}[(i)]
\item When $m$ is odd or even, there an isolated fixed point on $E(\rho_1)$ given by $V_2(\rho_1')=(xy)$.
\item When $m$ is odd, there is an isolated fixed point on $E(\rho_m)$ given by $V_{m+1}'(\rho_m)=(x^{m+1},y^{m+1})$.
\item When $m$ is even, $E(\rho_{m+1}')$ and $E(\rho_{m+2}')$ each contain an isolated fixed point given by $V_m(\rho_{m+1}')=(x^m-i^{m+2}y^m)$ and $V_m(\rho_{m+2}')=(x^m+i^{m+2}y^m)$ respectively.
\end{enumerate}
\end{enumerate}
\end{prop}

To compute the $A$-action, we need a representative $\alpha$ of the nontrivial element of $A$, which we can conveniently take to be $t^m$ from the generators of $G$ \eqref{G.Dm.gens}, which acts by sending $x$ to $-x$ and fixing $y$.

\subsubsection{$E(\rho_1')$} $\rho_1'$ is an endpoint of the Dynkin diagram a distance of $m-1$ from the center. By \Cref{IN.main}(1), the points of $E(\rho_1')$ away from its intersection with $E(\rho_2)$ are given by proper $H$-invariant submodules of $V_2(\rho_1')\oplus V_{2m}(\rho_1')$ other than $V_{2m}(\rho_1')$. By \Cref{IN.main}(4), its intersection with $E(\rho_2)$ is given by $V_{2m}(\rho_1')\oplus V_3(\rho_2)$. By \cite{Ito-Nakamura:1999}, we have 
\[
V_2(\rho_1')=(xy),\qquad V_{2m}(\rho_1')=(x^{2m}-y^{2m}),\qquad V_3(\rho_2)=(x^2y,xy^2).
\]
We may directly verify that $\alpha$ fixes $V_2(\rho_1')\oplus V_{2m}(\rho_1')$,
$V_2(\rho_1')$, and $V_{2m}(\rho_1')\oplus V_3(\rho_2)$, 
and thus $\alpha$ is an involution of $E(\rho_1')$ that fixes its point of intersection with $E(\rho_2)$ as well as the point given by $V_2(\rho_1')=(xy)$.
To determine whether $\alpha$ fixes all points on $E(\rho_1')$, we check its action on a third point.
We check that $(xy+(x^{2m}-y^{2m}))$ is an $H$-module in \texttt{FixedLocus.m2} and thus corresponds to a point on $E(\rho_1')$, and we confirm that:
\[
\alpha(xy+(x^{2m}-y^{2m})) = (xy-(x^{2m}-y^{2m})) \neq (xy+(x^{2m}-y^{2m})).
\]
Thus $A$ does not fix $E(\rho_1')$ pointwise.

\subsubsection{$E(\rho_k)$ for $2\leq k\leq m-1$}
In these cases, 
$\rho_k$  is not an endpoint of the Dynkin diagram and is a distance of $m-k$ from the center. By \Cref{IN.main}(1), the points of $E(\rho_k)$ away from its intersections with $E(\rho_{k-1})$ and $E(\rho_{k+1})$ are given by proper $H$-invariant submodules of $V_{k+1}(\rho_k)\oplus V_{2m-k+1}(\rho_k)$ other than $V_{k+1}(\rho_k)$ and $V_{2m-k+1}(\rho_k)$. By \Cref{IN.main}(4), its intersection with $E(\rho_{k-1})$ is given by $V_{k+1}(\rho_k)\oplus V_{2m-k+2}(\rho_{k-1})$ and its intersection with $E(\rho_{k+1})$ is given by $V_{2m-k+1}(\rho_k)\oplus V_{k+2}(\rho_{k+1})$, except in the case of $k=m-1$ in which case the intersection of $E(\rho_{m-1})$ with $E(\rho_m)$ is given by $V_{m+2}(\rho_{m-1})\oplus V_{m+1}''(\rho_m)$. By \cite{Ito-Nakamura:1999}, we have 
\[
V_{k+1}(\rho_k)=(x^ky,xy^k),\quad V_{2m-k+1}(\rho_k)=(x^{2m-k+1},y^{2m-k+1}),\quad V_{m+1}''=(x^my,xy^m).
\]
We may directly verify that $\alpha$ fixes each of the above ideals, and thus it is an involution of
$E(\rho_k)$ that fixes its two points of intersection with other exceptional components. 
To determine whether $\alpha$ fixes every point in  $E(\rho_k)$, we examine the action of $\alpha$ on a third point on $E(\rho_k)$: 
We check that $(x^ky+y^{2m-k+1}, xy^k+(-1)^kx^{2m-k+1}) \subsetneq V_{k+1}(\rho_k)\oplus V_{2m-k+1}(\rho_k)$
is an $H$-module in \texttt{FixedLocus.m2} and we confirm that
\[
\alpha((x^ky+y^{2m-k+1}, xy^k+(-1)^kx^{2m-k+1})) = ((-1)^kx^ky+y^{2m-k+1}, xy^k+x^{2m-k+1}),
\]
and thus $\alpha$ fixes this point when $k$ is even and does not fix it $k$ is odd. Thus $A$ fixes $E(\rho_k)$ pointwise when $k$ is even, but does not when $k$ is odd.

\subsubsection{$E(\rho_{m+1}')$ and $E(\rho_{m+2}')$} $\rho_{m+1}'$ and $\rho_{m+2}'$ are endpoints of the Dynkin diagram a distance of~$1$ from the center. By \Cref{IN.main}(1) the points on $E(\rho_{m+1}')$ and $E(\rho_{m+2}')$ away from their intersection with $E(\rho_m)$ are given by nontrivial, proper $H$-submodules of $V_m(\rho_{m+1}')\oplus V_{m+2}(\rho_{m+1}')$ other than $V_{m+2}(\rho_{m+1}')$ and nontrivial, proper $H$-submodules of $V_m(\rho_{m+2}')\oplus V_{m+2}(\rho_{m+2}')$ other than $V_{m+2}(\rho_{m+2}')$, respectively. By \ref{IN.main}(5), the intersections of
$E(\rho_{m+1}')$ and $E(\rho_{m+2}')$
with $E(\rho_m)$ are given by $V_{m+2}(\rho_{m+1}')\oplus[S_1\cdot V_{m}(\rho_{m+1}')]$ and $V_{m+2}(\rho_{m+2}')\oplus[S_1\cdot V_{m}(\rho_{m+2}')]$. By \cite[p.~217]{Ito-Nakamura:1999}, we have 
\begin{align*}
V_m(\rho_{m+1}')&=(x^m-i^{m+2}y^m),\qquad V_{m+2}(\rho_{m+1}')=(xy(x^m+i^{m+2}y^m)),\\
V_m(\rho_{m+2}')&=(x^m+i^{m+2}y^m),\qquad V_{m+2}(\rho_{m+2}')=(xy(x^m-i^{m+2}y^m)).
\end{align*}
We may directly verify that when $m$ is odd, $\alpha$ exchanges 
$V_m(\rho_{m+1}')\oplus V_{m+2}(\rho_{m+1}')$ with $V_m(\rho_{m+2}')\oplus V_{m+2}(\rho_{m+2}')$ but, when $m$ is even, $\alpha$ fixes each of them.

Thus, when $m$ is odd, $\alpha$ exchanges the components $E(\rho_{m+1}')$ and $E(\rho_{m+2}')$, but when $m$ is even, $\alpha$ acts on an involution on each of $E(\rho_{m+1}')$ and $E(\rho_{m+2}')$.

In the case when $m$ is even, we may verify directly that the following points are fixed by $\alpha$:
$V_m(\rho_{m+1}')=(x^m-i^{m+2}y^m)$ and $V_m(\rho_{m+2}')=(x^m+i^{m+2}y^m)$.
However, to determine whether $\alpha$ is a nontrivial involution on
$E(\rho_{m+1}')$ and $E(\rho_{m+2}')$, we check additional points on each of them.
We confirm that $((x^m-i^{m+2}y^m)+xy(x^m+i^{m+2}y^m))$ is an $H$-module and thus a point on $E(\rho_{m+1}')$. Then, because $m$ is even,
\begin{align*}
\alpha((x^m-i^{m+2}y^m)+xy(x^m+i^{m+2}y^m))&=((x^m-i^{m+2}y^m)-xy(x^m+i^{m+2}y^m))\\
&\neq ((x^m-i^{m+2}y^m)+xy(x^m+i^{m+2}y^m)).
\end{align*}
We also confirm that $((x^m+i^{m+2}y^m)+xy(x^m-i^{m+2}y^m))$ is an $H$-module and thus corresponds to a point on $E(\rho_{m+2}')$. 
\begin{align*}
\alpha((x^m+i^{m+2}y^m)+xy(x^m-i^{m+2}y^m))&=((x^m+i^{m+2}y^m)-xy(x^m-i^{m+2}y^m))\\
&\neq ((x^m+i^{m+2}y^m)+xy(x^m-i^{m+2}y^m)).
\end{align*}
Thus, when $m$ is even, $A$ does not pointwise fix $E(\rho_{m+1}')$ and $E(\rho_{m+2}')$.

\subsubsection{$E(\rho_m)$}
We may deduce the nature of the action of $A$ on $E(\rho_m)$ from the previous sections.
For any value of $m$, the action is an involution. 

For odd $m$, since $E(\rho_{m+1}')$ and $E(\rho_{m+2}')$ are exchanged, the involution is nontrivial. The point of intersection between $E(\rho_m)$ and $E(\rho_{m-1})$ is fixed, and we may verify directly that the other fixed point is
$V_{m+1}'(\rho_m)=(x^{m+1},y^{m+1})$.

For even $m$, 
since $\alpha$ fixes the points of intersection of $E(\rho_m)$ with its three adjacent components, the action of $A$ fixes $E(\rho_m)$.

\subsection{\texorpdfstring{$\gtw$}{G12}}\label{gtw.comp}

In this section, we set $G:=\gtw$. 
The following are generators of $G$ \cite{LehrerTaylor} where $\varepsilon:=e^{2\pi i / 8}$ is an $8$-th root of unity:
\[
  r_3=\frac{1}{\sqrt{2}}\begin{pmatrix} 1 & -1 \\ -1 & -1 \end{pmatrix},\quad
  r_3'=\frac{1}{\sqrt{2}}\begin{pmatrix} 1 & 1 \\ 1 & -1 \end{pmatrix},\quad
  r_3''=\begin{pmatrix} 0 & \varepsilon\\ \varepsilon^7& 0 \end{pmatrix}.
\]  
The intersection $H:=G\cap\SL(2,\C)$ is
the binary tetrahedral group and the singularity $\bC/H$ is of type $E_6$.
The group $H$ is 
generated by the following elements,
matching the notation given in \cite[\S 14]{Ito-Nakamura:1999}:
\[
  \sigma=\begin{pmatrix} i & 0 \\ 0 & -i \end{pmatrix},\quad
  \tau=\begin{pmatrix} 0 & 1 \\ -1 & 0 \end{pmatrix},\quad
  \mu=\frac{1}{\sqrt{2}}\begin{pmatrix} \varepsilon^7 & \varepsilon^7 \\ \varepsilon^5 & \varepsilon \end{pmatrix}.
\]  
We note that these embeddings are compatible because $\tau=r_3r_3'$, $\tau^2=-I$, $\sigma=-(r_3r_3'r_3'')^2$, and $\mu=(r_3'r_3)(r_3'r_3'')^2$. We choose the following element $\alpha\in G\setminus H$, which we will use to compute the action of $A$. We have chosen this particular representative for convenience since it is diagonal and so acts by scaling each of $x$ and $y$; it is not a reflection, but is the product of a reflection with a rotation, as we show:
\[
\alpha=
\begin{pmatrix}
\eps&0  \\0&\eps^3
\end{pmatrix}=
\begin{pmatrix} 0 & \eps \\ \eps^{-1} & 0 \end{pmatrix}\cdot \begin{pmatrix} 0 & \eps^4 \\ 1 & 0 \end{pmatrix}
  \]
  The matrices act on the coordinate vectors $x$ and $y$ in Ito and Nakamura's description of the points of $Y$.

The exceptional locus of $Y$ consists of six exceptional curves: $E(\rho_1')$, $E(\rho_1'')$, $E(\rho_2')$, $E(\rho_2'')$, $E(\rho_2)$, and $E(\rho_3)$.
Their intersections are given by the Dynkin diagram \Cref{E6.dynkin}.

The results of this section are summarized in the following proposition:

\begin{prop}\label{G12_A_summary}
  Let $G:=G_{12}$. The action of $A$ on the exceptional locus of $Y$:
\begin{enumerate}[(a)]
\item exchanges $E(\rho_1')$ with $E(\rho_1'')$ and $E(\rho_2')$ with $E(\rho_2'')$.
\item restricts to an involution on $E(\rho_2)$ and $E(\rho_3)$.
\item fixes each point on $E(\rho_2)$ and fixes one additional point on $E(\rho_3)$ given by:
  $$V(I)=(p_1(-\omega^2\phi+\psi),p_3(-\phi+\psi),p_2(-\omega\phi+\psi)).$$
\end{enumerate}
\end{prop}

In this case,
we use the following notation for generators of $H$-modules  in the exceptional locus of $Y$ \cite[p.\ 227]{Ito-Nakamura:1999}, where $\omega:=e^{2\pi i/3}$:
\begin{align*}
  p_1&=x^2-y^2, & p_2&=x^2+y^2, & p_3&=xy,
& T&=p_1p_2p_3,\ W=\varphi\psi.
  \\
q_1&=x^3+(2\omega+1)xy^2,&q_2&=y^3+(2\omega+1)x^2y,&s_1&=x^3+(2\omega^2+1)xy^2,&s_2&=y^3+(2\omega^2+1)x^2y,\\
\varphi&=p_2^2+4\omega p_3^2,&\psi&=p_2^2+4\omega^2p_3^2,&
\gamma_1&=x^5-5xy^4,&\gamma_2&=y^5-5x^4y,\\
\end{align*}
The action of $\alpha$ fixes and exchanges some of these generators up to scaling. In particular,
we will use that $\alpha(\varphi)=-\psi$, $\alpha(\psi)=-\varphi$,
$\alpha(s_1)=(t^5-t) q_1$, and 
$\alpha(s_2)=t^3q_2$,
$\alpha(\gamma_1)=\eps^5\gamma_1$, and
$\alpha(\gamma_2)=\eps^7\gamma_2$,
where $t:=e^{\frac{2\pi i}{24}}$, a primitive $24$-th root of unity. 

\subsubsection{$E(\rho_1')$ and $E(\rho_1'')$} Both $\rho_1'$ and $\rho_1''$ are endpoints of the Dynkin diagram \Cref{E6.dynkin} a distance of~$2$ from the center. By \Cref{IN.main}(1), the points of $E(\rho_1')$ and $E(\rho_1'')$, away from their intersections, are given by proper $H$-invariant submodules of $V_4(\rho_1')\oplus V_8(\rho_1')$ and
$V_4(\rho_1'')\oplus V_8(\rho_1'')$ other than $V_8(\rho_1')$ and $V_8(\rho_1'')$. By \Cref{IN.main}(4), the intersections of $E(\rho_1')$ with $E(\rho_2')$ and of $E(\rho_1'')$ with $E(\rho_2'')$ are given by
$V_5(\rho_2')\oplus V_8(\rho_1')$ and $V_5(\rho_2'')\oplus V_8(\rho_1'')$. 
In terms of the generators above, we have:
\begin{align*}
  &V_4(\rho_1')\oplus V_8(\rho_1')=(\varphi)\oplus(\psi^2)&
  &V_4(\rho_1'')\oplus V_8(\rho_1'')=(\psi)\oplus(\varphi^2)
  \\
  &V_5(\rho_2')\oplus V_8(\rho_1')=(x\varphi,y\varphi)\oplus(\psi^2) &
  &V_5(\rho_2'')\oplus V_8(\rho_1'')=(x\psi,y\psi)\oplus(\varphi^2) 
\end{align*}
By our observations above on the action of $\alpha$, $\alpha$ exchanges the generators of $V_4(\rho_1')$, $V_8(\rho_1')$, and $V_4(\rho_2')$ with those of $V_4(\rho_1'')$, $V_8(\rho_1'')$, and $V_4(\rho_2'')$ up to multiplication by scalars. Thus, the action of $A$ exchanges $E(\rho_1')$ with $E(\rho_1'')$, including the points of intersection with $E(\rho_2')$ and $E(\rho_2'')$.

\subsubsection{$E(\rho_2')$ and $E(\rho_2'')$} Now consider the curves $E(\rho_2')$ and $E(\rho_2'')$. The representations $\rho_2'$ and $\rho_2''$ are both a distance of $1$ from the center of the Dynkin diagram and so by \Cref{IN.main}(2), their points, away from intersections with other curves, are given by the proper $H$-invariant submodules of the following modules, respectively:
\[
V_5(\rho_2')\oplus V_7(\rho_2')=(x\varphi,y\varphi)\oplus (s_1\psi,s_2\psi),
  \quad
V_5(\rho_2'')\oplus V_7(\rho_2'')=(x\psi,y\psi)\oplus(q_1\varphi,q_2\varphi).
\]
Again, the action of $\alpha$ exchanges, up to scalars, the generators of
$V_5(\rho_2')$ and $V_7(\rho_2')$ with those of
$V_5(\rho_2'')$ and $V_7(\rho_2'')$, and so, 
the action of $A$ exchanges $E(\rho_2')$ with $E(\rho_2'')$.

\subsubsection{$E(\rho_2)$} Before examining the intersections of $E(\rho_3)$ with $E(\rho_2')$ and $E(\rho_2'')$, we analyze $E(\rho_2)$. To carry out this analysis, we parametrize $E(\rho_2)$.
By \Cref{IN.main}(1), the non-intersection points of $E(\rho_2)$ are given by $H$-invariant submodules, other than $V_7(\rho_2)$, of the following:
\begin{equation}\label{E2}
  V_5(\rho_2)\oplus V_7(\rho_2)=(\gamma_1,\gamma_2)\oplus(s_1\varphi,s_2\varphi).
\end{equation}
We parametrize $E(\rho_2)$ by setting each projective point $[a:b]$ with $a\neq 0$ in correspondence with the submodule $(a\gamma_1+bs_2\varphi,a\gamma_2-bs_1\varphi)$, 
and thus $[0:1]$ corresponds to the point of intersection with $E(\rho_3)$.
It remains to prove that $(a\gamma_1+bs_2\varphi,a\gamma_2-bs_1\varphi)$ is $H$-invariant. 
We check that it is fixed by each of $\tau$, $\sigma$, and $\mu$: The action of $\tau$ exchanges the generators up to scalars, the action of $\sigma$ scales each generator, and $\mu$ maps to a linear combination of the generators, as can be verified in \texttt{FixedLocus.m2}:
\begin{align*}
  &\tau(a\gamma_1+bs_2\varphi)=-a\gamma_2+bs_1\varphi,
  & &\sigma(a\gamma_1+bs_2\varphi)=i(a\gamma_1+bs_2\varphi), \\
  &\tau(a\gamma_2-bs_1\varphi)=a\gamma_1+bs_2\varphi,
  & &\sigma(a\gamma_2-bs_1\varphi)=-i(a\gamma_2-bs_1\varphi),\\
  &\mu(a\gamma_1+bs_2\varphi)=
\tfrac{-i+1}{2}(a\gamma_1+bs_2\varphi)-\tfrac{i+1}{2}(a\gamma_2-bs_1\varphi),&&\\
  & \mu(a\gamma_2-bs_1\varphi)=
\tfrac{-i+1}{2}(a\gamma_1+bs_2\varphi)+\tfrac{i+1}{2}(a\gamma_2-bs_1\varphi).&&
\end{align*}
We also verify in \texttt{FixedLocus.m2} that $\alpha$ fixes \eqref{E2}, meaning that the action of $A$ restricts to  an involution of $E(\rho_2)$.
In order to check that $\alpha$ fixes every point on 
$E(\rho_2)$, we furthermore verify that it fixes the three points
 $[1:1]$, $[1:-1]$,
and $[1:0]$.

This argument implies that the intersection point between $E(\rho_2)$ and $E(\rho_3)$ must also be fixed by the action of $A$, which we nevertheless check to aid in our parametrization of $E(\rho_3)$.
By \Cref{IN.main}(5), the intersection
of $E(\rho_2)$ and $E(\rho_3)$
corresponds to the module
denoted $\{S_1\cdot V_5(\rho_2)\}[\rho_3]\oplus V_7(\rho_2)$, where 
$\{S_1\cdot V_5(\rho_2)\}[\rho_3]$ is the summand of $S_1\cdot V_5(\rho_2)$
that is isomorphic to $\rho_3$ as an $H$-module. Since $S_1\cdot V_5(\rho_2)$ has four generators $x\gamma_1$, $y\gamma_1$, $x\gamma_2$, $y\gamma_2$, we need only identify a $3$-dimensional invariant subspace. We claim that:
\[
\{S_1\cdot V_5(\rho_2)[\rho_3]=(y\gamma_1+x\gamma_2,\ x\gamma_1+y\gamma_2,\ x\gamma_1-y\gamma_2).
\]  
In \texttt{FixedLocus.m2}, it is verified this submodule is $H$-invariant and further checked that $\{S_1\cdot V_5(\rho_2)\}[\rho_3]\oplus V_7(\rho_2)$ is fixed by $\alpha$, as expected.

\subsubsection{$E(\rho_3)$}
By \Cref{IN.main}(3), the non-intersection points of $E(\rho_3)$ correspond to proper submodules of $V_6(\rho_3)\simeq\rho_3^{\oplus 2}$ distinct from $\{S_1\cdot V_5(\rho')\}[\rho_3]$ for any $E(\rho')$ intersecting $E(\rho_3)$.
Since $V_6(\rho_3)$ is fixed by $A$ as shown in \texttt{FixedLocus.m2}, the action of $A$ restricts to an involution on $V_6(\rho_3)$.
We parametrize $E(\rho_3)$ by setting each projective point $[a:b]$
with $a\neq 0$, $b\neq 0$ 
in correspondence with the following submodule:
\[
((x^2-y^2)(a\varphi+b\psi), xy(a\omega\varphi+b\psi), (x^2+y^2)(a\omega^2\varphi+b\psi))
\subseteq   V_6(\rho_3)=(x^2,xy,y^2)\cdot(\varphi,\psi).
\]  
Since $V_5(\rho_2')=(x\phi,y\phi)$ and $V_5(\rho_2'')=(x\psi,y\psi)$  the points of intersection of $E(\rho_3)$ with $E(\rho_2')$ and $E(\rho_2'')$ are parametrized by $[1:0]$ and $[0:1]$, respectively. It is checked in \texttt{FixedLocus.m2} that $[1:\omega]$ corresponds to the point of intersection with $E(\rho_2)$.

Finally, we examine the action of $A$ on $E(\rho_3)$. The following actions are checked in \texttt{FixedLocus.m2}:
\begin{align*}
&\alpha(\varphi)=-\psi,\quad
\alpha(\psi)=-\varphi,\quad
\alpha(x^2)=ix^2,\quad
\alpha(xy)=-xy,\quad
\alpha(y^2)=-iy^2
\\  &\alpha((x^2-y^2)(a\varphi+b\psi))
=-i\omega(x^2+y^2)(b\omega^2\varphi+a\omega^2\psi)    
  \\
  &\alpha(xy(a\omega\varphi+b\psi))
=-xy\omega^2(b\omega\varphi+a\omega^2\psi)
  \\
&\alpha((x^2+y^2)(a\omega^2\varphi+b\psi))=
i(x^2-y^2)(b\varphi+a\omega^2\psi)  
\end{align*}
Thus, $A$ acts on $E(\rho_3)$ by sending
$[a:b]$ to $[b\omega:a]$. 
As expected, $A$ exchanges the points of intersection with
$E(\rho_2')$ and $E(\rho_2'')$.
Its fixed points are $[1:\omega]$ (the point of intersection with $E(\rho_2)$) and $[1:t^{20}]=[1:-\omega]$, which is not a point of intersection with another exceptional curve. We list the specific module $V(I)$ in \Cref{G12_A_summary}(c).

\subsection{$G_{13}$}\label{app.gthir} In this section, we set $G:=G_{13}$. A presentation of $G$ is given in \cite[p.~88]{LehrerTaylor}, where $G$ is generated by
\[
r=\begin{pmatrix}1&0\\0&-1\end{pmatrix},\quad r_3=\frac{1}{\sqrt{2}}\begin{pmatrix}1&-1\\-1&-1\end{pmatrix},\quad r_3''=\begin{pmatrix}0&\varepsilon\\\varepsilon^7&0\end{pmatrix}.
\]
The intersection $H:=G\cap\SL(2,\bC)$ is the binary octahedral group and the singularity $\bC^2/H$ is of type $E_7$. A presentation of this group is given in \cite[Ch.~15]{Ito-Nakamura:1999}, where $H$ is generated by
\[
  \sigma=\begin{pmatrix} i & 0 \\ 0 & -i \end{pmatrix},\quad
  \tau=\begin{pmatrix} 0 & 1 \\ -1 & 0 \end{pmatrix},\quad
  \mu=\frac{1}{\sqrt{2}}\begin{pmatrix} \varepsilon^7 & \varepsilon^7 \\ \varepsilon^5 & \varepsilon \end{pmatrix},\quad
  \kappa=\begin{pmatrix}\varepsilon&0\\0&\varepsilon^7\end{pmatrix},
\]  
where $\varepsilon\colon=e^{\pi i/4}$ is a primitive eighth root of unity. We note that these embeddings are compatible because $rr_3r=r_3'$ and so this group contains the embedding of $G_{12}$ used in \Cref{gtw.comp} and thus $\sigma$, $\tau$, and $\mu$. Finally, note that $\kappa=r_3''r\tau$.

In this section, we will compute the following information about the action of $A$, which we summarize in a proposition:

\begin{prop}\label{G13_A_summary}
Let $G:=G_{13}$. The action of $A$ on the exceptional locus of $Y$:
\begin{enumerate}[(a)]
\item maps every exceptional curve to itself;
\item pointwise fixes $E(\rho_2)$, $E(\rho_2')$, and $E(\rho_4)$;
\item fixes exactly two points on $E(\rho_1')$, $E(\rho_2'')$, $E(\rho_3)$, and $E(\rho_3')$, leaving a single isolated (within the exceptional locus) fixed point on $E(\rho_1')$ and $E(\rho_2'')$ given by $V_6(\rho_1')=(T)$ and $V_8(\rho_2'')=(W)$ respectively.
\end{enumerate}
\end{prop}

Since $G_{12}$ is contained in $G_{13}$, we can compute the action of $A$ on the exceptional locus using the same choice of $\alpha\in G\setminus H$ with
\[\alpha=\begin{pmatrix}\varepsilon&0\\0&\varepsilon^3\end{pmatrix}.\]
The exceptional locus of $Y$ consists of seven exceptional curves: $E(\rho_1')$, $E(\rho_2)$, $E(\rho_2')$, $E(\rho_2'')$, $E(\rho_3)$, $E(\rho_3')$, and $E(\rho_4)$. We will show in our computations that
$E(\rho_1')$, $E(\rho_2)$, $E(\rho_2')$, $E(\rho_2'')$, $E(\rho_3)$, and $E(\rho_3')$
are fixed by $\alpha$, though not necessarily pointwise -- $\alpha$ restricts to an involution on each of them.
This result implies that the central component, $E(\rho_4)$, must be pointwise fixed, as its three points of intersection with $E(\rho_3), E(\rho_3')$ and $E(\rho_2'')$ must be fixed.

We will need several of the same invariants from \Cref{G12_A_summary} so we include the relevant ones here for convenience:
\begin{align*}
\varphi&=(x^2+y^2)^2+4\omega x^2y^2,&\psi&=(x^2+y^2)^2+4\omega^2x^2y^2, &W&=\varphi\psi,\\
T&=x^5y-xy^5, &\chi&=x^{12}-33x^8y^4-33x^4y^8+y^{12}&F&=\chi T
\end{align*}
where $\omega=e^{2\pi i/3}$ is a primitive cube root of unity.

\subsubsection{$E(\rho_1')$} $\rho_1'$ is an endpoint of the Dynkin diagram \Cref{E7.dynkin} a distance of $3$ from the center. By \Cref{IN.main}(1), points on $E(\rho_1')$ away from its intersection with $E(\rho_2')$ are given by nontrivial, proper $H$-invariant submodules of $V_6(\rho_1')\oplus V_{12}(\rho_1')$ other than $V_{12}(\rho_1')$. By \Cref{IN.main}(4), its intersection with $E(\rho_2')$ corresponds to $V_{12}(\rho_1')\oplus V_7(\rho_2')$. By \cite[p.235]{Ito-Nakamura:1999}, we have
\[ V_6(\rho_1')=(T)\quad V_{12}(\rho_1')=(\chi)\]
We check in \texttt{FixedLocus.m2} that $\alpha$ fixes $V_6(\rho_1')\oplus V_{12}(\rho_1')$ and thus is an involution on $E(\rho_1')$.
Furthermore, $\alpha$ fixes $V_6(\rho_1')$ and $ V_{12}(\rho_1')\oplus V_7(\rho_2')$, so
to confirm whether $\alpha$ fixes $E(\rho_1')$ pointwise we simply need to check whether it fixes a third point on $E(\rho_1')$. In \texttt{FixedLocus.m2}, we confirm that $(T+\chi)$ is indeed an $H$-submodule of
$V_6(\rho_1')\oplus V_{12}(\rho_1')$
and check  that
\[
\alpha((T+\chi))=(T-\chi)\neq(T+\chi).
\]
Thus, $A$ does not pointwise fix $E(\rho_1')$.

\subsubsection{$E(\rho_2)$} $\rho_2$ is an endpoint of the Dynkin diagram and is a distance of $2$ from the center. By \Cref{IN.main}(1), points on $E(\rho_2)$ away from its intersection with $E(\rho_3)$ are given by nontrivial, proper $H$-invariant submodules of $V_7(\rho_2)\oplus V_{11}(\rho_2)$ other than $V_{11}(\rho_2)$. By \Cref{IN.main}(4), its intersection with $E(\rho_3)$ corresponds to $V_{11}(\rho_2)\oplus V_8(\rho_3)$. By \cite[p.235]{Ito-Nakamura:1999}, we have
\begin{align*}
V_7(\rho_2)&=(7x^4y^3+y^7,-x^7-7x^3y^4)\\
 V_{11}(\rho_2)&=(x^{10}y-6x^6y^5+5x^2y^9,-xy^{10}+6x^5y^6-5x^9y^2)
\end{align*}
For readability, we will set
\begin{align*}
f_1&=7x^4y^3+y^7& f_2&=-x^7-7x^3y^4\\
g_1&=x^{10}y-6x^6y^5+5x^2y^9&g_2&=-xy^{10}+6x^5y^6-5x^9y^2
\end{align*}
We check in \texttt{FixedLocus.m2} that $\alpha$ fixes $V_7(\rho_2)\oplus V_{11}(\rho_2)$ and thus is an involution on $E(\rho_2)$.
To confirm whether $\alpha$ fixes $E(\rho_2)$ pointwise we simply need to check whether it fixes a third point on $E(\rho_2)$.
In \texttt{FixedLocus.m2},
we confirm that $(f_1+g_1,f_2+g_2)$ is indeed an $H$-submodule of $V_7(\rho_2)\oplus V_{11}(\rho_2)$ and check  that
\[
\alpha((f_1+g_1,f_2+g_2))=(f_1+g_1,f_2+g_2).
\]
Thus, $A$ pointwise fixes $E(\rho_2)$.

\subsubsection{$E(\rho_2')$} $\rho_2'$ is not an endpoint of the Dynkin diagram and is a distance of $2$ from the center. By \Cref{IN.main}(2), points on $E(\rho_2')$ away from its intersection with $E(\rho_1')$ and $E(\rho_3')$ are given by nontrivial, proper $H$-invariant submodules of $V_7(\rho_2')\oplus V_{11}(\rho_2')$ other than $V_7(\rho_2')$ and $V_{11}(\rho_2')$. By \Cref{IN.main}(4), its intersection with $E(\rho_1')$ corresponds to $V_{12}(\rho_1')\oplus V_7(\rho_2')$ and its intersection with $E(\rho_3')$ corresponds to $V_{11}(\rho_2')\oplus V_8(\rho_3')$. By \cite[p.235]{Ito-Nakamura:1999}, we have
\[
V_7(\rho_2')=(xT,yT)\quad V_{11}(\rho_2')=(-11x^8y^3-22x^4y^7+y^{11},11x^3y^8+22x^7y^4-x^{11})
\]
For readability, let 
\[v_1=-11x^8y^3-22x^4y^7+y^{11}\quad v_2=11x^3y^8+22x^7y^4-x^{11}\]
We check in \texttt{FixedLocus.m2} that $\alpha$ fixes $V_7(\rho_2')\oplus V_{11}(\rho_2')$ and thus is an involution on $E(\rho_2')$.
To confirm whether $\alpha$ fixes $E(\rho_2')$ pointwise we simply need to check whether it fixes a third point on $E(\rho_2')$.
In \texttt{FixedLocus.m2},
we confirm that $(xT+v_1,yT+v_2)$ is indeed an $H$-submodule of $V_7(\rho_2')\oplus V_{11}(\rho_2')$
and check that
\[
\alpha((xT+v_1,yT+v_2))=(xT+v_1,yT+v_2).
\]
Thus, $A$ pointwise fixes $E(\rho_2')$.

\subsubsection{$E(\rho_2'')$} $\rho_2''$ is an endpoint of the Dynkin diagram \Cref{E7.dynkin} and is a distance of $1$ from the center. By \Cref{IN.main}(1), points on $E(\rho_2'')$ away from its intersection with $E(\rho_4)$ are given by nontrivial, proper $H$-invariant submodules of $V_8(\rho_2'')\oplus V_{10}(\rho_2'')$ other than $V_8(\rho_2'')$. By \Cref{IN.main}(5), its intersection with $E(\rho_4)$ corresponds to $V_{10}(\rho_2'')\oplus [S_1\cdot V_8(\rho_2'')]$. We note that there was a small typo in the generators of $V_{12}(\rho_2)$ in \cite[p.235]{Ito-Nakamura:1999} but with that typo corrected we get 
\[
V_8(\rho_2'')=(\psi^2,-\varphi^2),\quad
V_{10}(\rho_2'')=(\psi T,-\varphi T)
\]
We check in \texttt{FixedLocus.m2} that $\alpha$ fixes $V_8(\rho_2'')\oplus V_{10}(\rho_2'')$ and thus is an involution on $E(\rho_2'')$.
To confirm whether $\alpha$ fixes $E(\rho_2'')$ pointwise we simply need to check whether it fixes a third point on $E(\rho_2'')$.
In \texttt{FixedLocus.m2},
we confirm that $(\psi^2+\phi T,\phi^2+\psi T)$ is indeed an $H$-submodule
of $V_8(\rho_2'')\oplus V_{10}(\rho_2'')$
and check  that
\[
\alpha((\psi^2+\phi T,\phi^2+\psi T))=(\psi^2-\phi T,\phi^2-\psi T)\neq (\psi^2+\phi T,\phi^2+\psi T).
\]
Thus, $A$ does not pointwise fix $E(\rho_2'')$.

\subsubsection{$E(\rho_3)$} $\rho_3$ is not an endpoint of the Dynkin diagram \Cref{E7.dynkin} and is a distance of $1$ from the center. By \Cref{IN.main}(2), points on $E(\rho_3)$ away from its intersection with $E(\rho_2)$ and $E(\rho_4)$ are given by nontrivial, proper $H$-invariant submodules of $V_8(\rho_3)\oplus V_{10}(\rho_3)$ other than $V_8(\rho_3)$ and $V_{10}(\rho_3)$. By \Cref{IN.main}(4) and (5), its intersection with $E(\rho_2)$ corresponds to $V_{11}(\rho_2)\oplus V_8(\rho_3)$ and its intersection with $E(\rho_4)$ corresponds to $V_{10}(\rho_3)\oplus [S_1\cdot V_8(\rho_3)]$. By \cite[p.235]{Ito-Nakamura:1999}, we have
\begin{align*}
V_8(\rho_3)&=(-2xy^7-14x^5y^3,x^8-y^8,2x^7y+14x^3y^5)\\
V_{10}(\rho_3)&=(40x^{10}+60x^6y^4,5x^9y+54x^5y^5+5xy^9,60x^4y^6+4y^{10})
\end{align*}
For readability, we write our generators as
\begin{align*}
h_1&=-2xy^7-14x^5y^3&h_2&=x^8-y^8&h_3&=2x^7y+14x^3y^5\\
j_1&=40x^{10}+60x^6y^4&j_2&=5x^9y+54x^5y^5+5xy^9&j_3&=60x^4y^6+4y^{10}
\end{align*}
We check in \texttt{FixedLocus.m2} that $\alpha$ fixes $V_8(\rho_3)\oplus V_{10}(\rho_3)$ and thus is an involution on $E(\rho_3)$.
To confirm whether $\alpha$ fixes $E(\rho_3)$ pointwise we simply need to check whether it fixes a third point on $E(\rho_3)$.
In \texttt{FixedLocus.m2},
we confirm that $(h_1+j_1,h_2+j_2,h_3+j_3)$ is indeed an $H$-submodule
of $V_8(\rho_3)\oplus V_{10}(\rho_3)$
and check  that
\[
\alpha((h_1+j_1,h_2+j_2,h_3+j_3))=(h_1-j_1,h_2-j_2,h_3-j_3))\neq (h_1+j_1,h_2+j_2,h_3+j_3)).
\]
Thus, $A$ does not pointwise fix $E(\rho_3)$.

\subsubsection{$E(\rho_3')$} $\rho_3'$ is not an endpoint of the Dynkin diagram \Cref{E7.dynkin} and is a distance of $1$ from the center. By \Cref{IN.main}(2), points on $E(\rho_3')$ away from its intersection with $E(\rho_2')$ and $E(\rho_4)$ are given by nontrivial, proper $H$-invariant submodules of $V_8(\rho_3')\oplus V_{10}(\rho_3')$ other than $V_8(\rho_3')$ and $V_{10}(\rho_3')$. By \Cref{IN.main}(4) and (5), its intersection with $E(\rho_2')$ corresponds to $V_{11}(\rho_2')\oplus V_8(\rho_3')$ and its intersection with $E(\rho_4)$ corresponds to $V_{10}(\rho_3')\oplus [S_1\cdot V_8(\rho_3')]$. By \cite[p.235]{Ito-Nakamura:1999}, we have
\begin{align*}
V_8(\rho_3')&=(x^2T,xyT,y^2T)\\
V_{10}(\rho_3')&=(-3x^8y^2-14x^4y^6+y^{10},8x^7y^3+8x^3y^7,x^{10}-14x^6y^4-3x^2y^8)
\end{align*}
For readability we set
\[u_1=-3x^8y^2-14x^4y^6+y^{10}\quad u_2=8x^7y^3+8x^3y^7\quad u_3=x^{10}-14x^6y^4-3x^2y^8\]
We check in \texttt{FixedLocus.m2} that $\alpha$ fixes $V_8(\rho_3')\oplus V_{10}(\rho_3')$ and thus is an involution on $E(\rho_3')$.
To confirm whether $\alpha$ fixes $E(\rho_3')$ pointwise we simply need to check whether it fixes a third point on $E(\rho_3')$.
In \texttt{FixedLocus.m2},
we confirm that $(x^2T+u_1,xyT+u_2,y^2T+u_3)$ is indeed an $H$-submodule
of $V_8(\rho_3')\oplus V_{10}(\rho_3')$
and check  that
\[
\alpha((x^2T+u_1,xyT+u_2,y^2T+u_3))=(x^2T-u_1,xyT-u_2,y^2T-u_3)\neq (x^2T+u_1,xyT+u_2,y^2T+u_3).
\]
Thus, $A$ does not pointwise fix $E(\rho_3')$.

\subsection{$G_{22}$}\label{app.gtwtw}

In this section, we set $G:=G_{22}$. 
Then $H$ is the binary icosahedral group and the singularity $\bC^2/H$ is of type $E_8$. A presentation of $H$ with the following generators is given in \cite[Ch.~15]{Ito-Nakamura:1999}, where
$\varepsilon\colon=e^{2\pi i/5}$ is a primitive fifth root of unity:
\[
  \sigma=-\begin{pmatrix} \varepsilon^3 & 0 \\ 0 & \varepsilon^2 \end{pmatrix},\quad
  \tau=\frac{1}{\sqrt{5}}\begin{pmatrix} -(\varepsilon-\varepsilon^4) & \varepsilon^2-\varepsilon^3 \\ \varepsilon^2-\varepsilon^3 & \varepsilon-\varepsilon^4 \end{pmatrix}.
\]
A presentation of $G$ is given in \cite[p.~88]{LehrerTaylor}, but the embedding of $G\cap\SL(2,\bC)$ is distinct from that given in \cite{Ito-Nakamura:1999}. To extend the embedding of $H$ in \cite{Ito-Nakamura:1999} to an embedding of $G$, we add the 
following choice of $\alpha$ as a generator:
\[\alpha=\frac{i}{\sqrt{5}}\begin{pmatrix}-\varepsilon+\varepsilon^4&\varepsilon^2-\varepsilon^3\\\varepsilon^2-\varepsilon^3&\varepsilon-\varepsilon^4\end{pmatrix}.\]
It is shown in \cite[p.~90]{LehrerTaylor} that the center of $G_{22}$ is the cyclic group of order $4$ and in \cite{group-names} that in fact $G_{22}$ is the unique extension of $H\simeq\SL_2(\mathbb{F}_5)$ by $\Z/2$ whose center contains an element of order $4$. 
Thus to show that $\sigma$, $\tau$ and $\alpha$ in fact generate $G_{22}$, it suffices to observe that $\alpha\tau^{-1}=iI$ (this is verified in \texttt{FixedLocus.m2}).

In this section, we compute the following information about the action of $A$, which we summarize in a proposition:

\begin{prop}\label{G22_A_summary}
Let $G:=G_{22}$. The action of $A$ on the exceptional locus of $Y$:
\begin{enumerate}[(a)]
\item maps every exceptional curve to itself;
\item pointwise fixes $E(\rho_2)$, $E(\rho_2')$, $E(\rho_4)$, and $E(\rho_6)$;
\item fixes exactly two points on $E(\rho_3)$, $E(\rho_3'')$, $E(\rho_4')$, and $E(\rho_5)$, leaving a single isolated (within the exceptional locus) fixed point on $E(\rho_3'')$ given by $V_{14}(\rho_3'')$.
\end{enumerate}
\end{prop}

It will also be convenient to use the following invariants from \cite[p.239]{Ito-Nakamura:1999}:
\begin{align*}
\sigma_1& = x^{10}+66x^5y^5-11y^{10} &\sigma_2& = -11x^{10}-66x^5y^5+y^{10}\\
\tau_1 &= x^{10}-39x^5y^5-26y^{10} & \tau_2 &= -26x^{10}+39x^5y^5+y^{10}
\end{align*}

The exceptional locus of $\HHilb(\bC^2)$ consists of eight curves: $E(\rho_2)$, $E(\rho_2')$, $E(\rho_3)$, $E(\rho_3'')$,$E(\rho_4)$, $E(\rho_4')$, $E(\rho_5)$, and $E(\rho_6)$. Their intersections are shown in \Cref{E8.dynkin}.
We will show computationally that
$E(\rho_2)$, $E(\rho_2')$, $E(\rho_3)$, $E(\rho_3'')$,$E(\rho_4)$, $E(\rho_4')$, and $E(\rho_5)$
are fixed by $\alpha$, though not necessarily pointwise -- $\alpha$ restricts to an involution on each of them.
These results imply that the central component, $E(\rho_6)$, must be pointwise fixed, as its three points of intersection with $E(\rho_3''), E(\rho_4')$ and $E(\rho_5)$ must be fixed.
Furthermore, since every point of intersection is fixed and $\alpha$ acts on each exceptional component $E(\rho)$ as in involution, to check whether $\alpha$ pointwise fixes $E(\rho)$ or not, we simply need to find a non-intersection point on $E(\rho)$ and check whether $\alpha$ fixes that point.

\subsubsection{$E(\rho_2)$} $\rho_2$ is an endpoint of the Dynkin diagram a distance of $4$ from the center. By \Cref{IN.main}(1), points on $E(\rho_2)$ away from its intersection with $E(\rho_3)$ are given by nontrivial, proper $H$-invariant submodules of $V_{11}(\rho_2)\oplus V_{19}(\rho_2)$ other than $V_{19}(\rho_2)$. By \Cref{IN.main}(4), its intersection with $E(\rho_3)$ corresponds to $V_{19}(\rho_2)\oplus V_{12}(\rho_3)$. By \cite[p.241]{Ito-Nakamura:1999}, we have
\begin{align*}
V_{11}(\rho_2)&=(x\sigma_1,-y\sigma_2)\\
V_{19}(\rho_2)&=(-57x^{15}y^4+247x^{10}y^9+171x^5y^{14}+y^{19},-x^{19}+171x^{14}y^5-247x^9y^{10}-57x^4y^{15})
\end{align*}
For readability let
\begin{align*}
c_1&=x\sigma_1& c_2&=-y\sigma_2\\
d_1&=-57x^{15}y^4+247x^{10}y^9+171x^5y^{14}+y^{19}&d_2&=-x^{19}+171x^{14}y^5-247x^9y^{10}-57x^4y^{15}
\end{align*}
We check in \texttt{FixedLocus.m2} that $\alpha$ fixes
$V_{11}(\rho_2)\oplus V_{19}(\rho_2)$
and thus is an involution on $E(\rho_2)$. To confirm whether $\alpha$ fixes $E(\rho_2)$ pointwise we simply need to check whether it fixes a third point on $E(\rho_2)$.
In \texttt{FixedLocus.m2}, we check that $(c_1+d_1,c_2-d_2)$ is an $H$-submodule
of $V_{11}(\rho_2)\oplus V_{19}(\rho_2)$
and we confirm that
\[
\alpha((c_1+d_1,c_2-d_2))=(c_1+d_1,c_2-d_2).
\]
Thus $A$ fixes $E(\rho_2)$ pointwise.

\subsubsection{$E(\rho_2')$} $\rho_2'$ is an endpoint of the Dynkin diagram a distance of $2$ from the center. By \Cref{IN.main}(1), points on $E(\rho_2)$ away from its intersection with $E(\rho_4')$ are given by nontrivial, proper $H$-invariant submodules of $V_{13}(\rho_2')\oplus V_{17}(\rho_2')$ other than $V_{17}(\rho_2')$. By \Cref{IN.main}(4), its intersection with $E(\rho_4')$ corresponds to $V_{17}(\rho_2')\oplus V_{14}(\rho_4')$. By \cite[p.241]{Ito-Nakamura:1999}, we have
\begin{align*}
V_{13}(\rho_2')&=(y^3\tau_2,-x^3\tau_1)\\
V_{17}(\rho_2')&=(x^{17}+119x^{12}y^5+187x^7y^{10}+17x^2y^{15},-17x^{15}y^2+187x^{10}y^7-119x^5y^{12}+y^{17})
\end{align*}
For readability let
\begin{align*}
u_1&=y^3\tau_2& u_2&=-c^3\tau_1\\
v_1&=x^{17}+119x^{12}y^5+187x^7y^{10}+17x^2y^{15}&v_2&=-17x^{15}y^2+187x^{10}y^7-119x^5y^{12}+y^{17}
\end{align*}
We check in \texttt{FixedLocus.m2} that $\alpha$ fixes
$V_{13}(\rho_2')\oplus V_{17}(\rho_2')$
and thus is an involution on $E(\rho_2')$. To confirm whether $\alpha$ fixes $E(\rho_2')$ pointwise we simply need to check whether it fixes a third point on $E(\rho_2')$.
In \texttt{FixedLocus.m2}, we check that $(u_1+v_1,u_2+v_2)$ is an $H$-module of $V_{13}(\rho_2')\oplus V_{17}(\rho_2')$
and we confirm that
\[
\alpha((u_1+v_1,u_2+v_2))=(u_1+v_1,u_2+v_2).
\]
Thus $A$ fixes $E(\rho_2')$ pointwise.

\subsubsection{$E(\rho_3)$} $\rho_3$ is not an endpoint of the Dynkin diagram and is a distance of $3$ from the center. By \Cref{IN.main}(2), points on $E(\rho_3)$ away from its intersections with $E(\rho_2)$ and $E(\rho_4)$ are given by nontrivial, proper $H$-invariant submodules of $V_{12}(\rho_3)\oplus V_{18}(\rho_3)$ other than $V_{12}(\rho_3)$ and $V_{18}(\rho_3)$. By \Cref{IN.main}(4), its intersection with $E(\rho_2)$ corresponds to $V_{19}(\rho_2)\oplus V_{12}(\rho_3)$ and its intersection with $E(\rho_4)$ corresponds to $V_{18}(\rho_3)\oplus V_{13}(\rho_4)$. By \cite[p.241]{Ito-Nakamura:1999}, we have
\begin{align*}
V_{12}(\rho_3)=&(x^2\sigma_1,\frac{xy}{2}(\sigma_1+\sigma_2), y^2\sigma_2)\\
V_{18}(\rho_3)=&(-12x^{15}y^3+117x^{10}y^8+126x^5y^{13}+y^{18},45x^{14}y^4-130x^9y^9-45x^4y^{14},\\
&x^{18}-126x^{13}y^5+117x^8y^{10}+12x^3y^{15})
\end{align*}
For readability let
\begin{align*}
f_1&=x^2\sigma_1&g_1&=-12x^{15}y^3+117x^{10}y^8+126x^5y^{13}+y^{18}\\
f_2&=\frac{xy}{2}(\sigma_1+\sigma_2)& g_2 &=45x^{14}y^4-130x^9y^9-45x^4y^{14}\\
f_3&= y^2\sigma_2& g_3&=x^{18}-126x^{13}y^5+117x^8y^{10}+12x^3y^{15}
\end{align*}
We check in \texttt{FixedLocus.m2} that $\alpha$ fixes
$V_{12}(\rho_3)\oplus V_{18}(\rho_3)$
and thus is an involution on $E(\rho_3)$.
To confirm whether $\alpha$ fixes $E(\rho_3)$ pointwise we simply need to check whether it fixes a third point on $E(\rho_3)$.
In \texttt{FixedLocus.m2}, we check that $(f_1+g_1,f_2+g_2,f_3+g_3)$ is an $H$-module of $V_{12}(\rho_3)\oplus V_{18}(\rho_3)$
 and we confirm that
\[
\alpha((f_1+g_1,f_2+g_2,f_3+g_3))=(f_1-g_1,f_2-g_2,f_3-g_3)\neq(f_1+g_1,f_2+g_2,f_3+g_3).
\]
Thus, $A$ does not fix $E(\rho_3)$ pointwise.

\subsubsection{$E(\rho_3'')$} $\rho_3''$ is an endpoint of the Dynkin diagram a distance of $1$ from the center. By \Cref{IN.main}(1), points on $E(\rho_3'')$ away from its intersection with $E(\rho_6)$ are given by nontrivial, proper $H$-invariant submodules of $V_{14}(\rho_3'')\oplus V_{16}(\rho_3'')$ other than $V_{16}(\rho_3'')$. By \Cref{IN.main}(5), its intersection with $E(\rho_6)$ corresponds to $V_{16}(\rho_3'')\oplus [S_1\cdot V_{16}(\rho_3'')]$. By \cite[p.241]{Ito-Nakamura:1999}, we have
\begin{align*}
V_{14}(\rho_3'')=&(x^{14}-14x^9y^5+49x^4y^{10},7x^{12}y^2-48x^7y^7-7x^2y^{12}, 49x^{10}y^4+14x^5y^9+y^{14})\\
V_{16}(\rho_3'')=&(3x^{15}y-143x^{10}y^6-39x^5y^{11}+y^{16},-25x^{13}y^3-25x^3y^{13},x^{16}+39x^{11}y^5-143x^6y^{10}-3xy^{15})
\end{align*}
For readability let
\begin{align*}
m_1&=x^{14}-14x^9y^5+49x^4y^{10}&n_1&=3x^{15}y-143x^{10}y^6-39x^5y^{11}+y^{16}\\
m_2&=7x^{12}y^2-48x^7y^7-7x^2y^{12}&n_2&=-25x^{13}y^3-25x^3y^{13}\\
m_3&=49x^{10}y^4+14x^5y^9+y^{14}&n_3&=x^{16}+39x^{11}y^5-143x^6y^{10}-3xy^{15}
\end{align*}

We check in \texttt{FixedLocus.m2} that $\alpha$ fixes
$V_{14}(\rho_3'')\oplus V_{16}(\rho_3'')$
and thus is an involution on $E(\rho_3'')$. We also check that $\alpha$ fixes
$V_{14}(\rho_3'')$, which is thus a fixed point on $E(\rho_3'')$.
To confirm whether $\alpha$ fixes $E(\rho_3'')$ pointwise we simply need to check whether it fixes a third point on $E(\rho_3'')$.
In \texttt{FixedLocus.m2}, we check that $(m_1+n_1,m_2+n_2,m_3+n_3)$ is an $H$-module
of $V_{14}(\rho_3'')\oplus V_{16}(\rho_3'')$
and we confirm that
\[
\alpha((m_1+n_2,m_2+n_2,m_3+n_3))=(m_1-n_1,m_2-n_2,m_3-n_3)\neq(m_1+n_1,m_2+n_2,m_3+n_3).
\]
Thus, $A$ does not fix $E(\rho_3'')$ pointwise.

\subsubsection{$E(\rho_4)$} $\rho_4$ is not an endpoint of the Dynkin diagram and is a distance of $2$ from the center. By \Cref{IN.main}(2), points on $E(\rho_4)$ away from its intersections with $E(\rho_3)$ and $E(\rho_5)$ are given by nontrivial, proper $H$-invariant submodules of $V_{13}(\rho_4)\oplus V_{17}(\rho_4)$ other than $V_{13}(\rho_4)$ and $V_{17}(\rho_4)$. By \Cref{IN.main}(4), its intersection with $E(\rho_3)$ corresponds to $V_{18}(\rho_3)\oplus V_{13}(\rho_4)$ and its intersection with $E(\rho_5)$ corresponds to $V_{17}(\rho_4)\oplus V_{14}(\rho_5)$. By \cite[p.241]{Ito-Nakamura:1999}, we have
\begin{align*}
V_{13}(\rho_4)=&(x^3\sigma_1,-3x^{12}y+22x^7y^6-7x^2y^{11},-7x^{11}y^2-22x^6y^7-3xy^{12}, y^3\sigma_2)\\
V_{17}(\rho_4)=&(-2x^{15}y^2+52x^{10}y^7+91x^5y^{12}+y^{17},10x^{14}y^3-65x^9y^8-35x^4y^{13},\\
&-35x^{13}y^4+65x^8y^9+10x^3y^{14},-x^{17}+91x^{12}y^5-52x^7y^{10}-2x^2y^{15})
\end{align*}
For readability let
\begin{align*}
h_1&=x^3\sigma_1&j_1&=-2x^{15}y^2+52x^{10}y^7+91x^5y^{12}+y^{17}\\
h_2&=-3x^{12}y+22x^7y^6-7x^2y^{11}&j_2&=10x^{14}y^3-65x^9y^8-35x^4y^{13}\\
h_3&=-7x^{11}y^2-22x^6y^7-3xy^{12}&j_3&=-35x^{13}y^4+65x^8y^9+10x^3y^{14}\\
h_4&=y^3\sigma_2&j_4&=-x^{17}+91x^{12}y^5-52x^7y^{10}-2x^2y^{15}
\end{align*}

We check in \texttt{FixedLocus.m2} that $\alpha$ fixes
$V_{13}(\rho_4)\oplus V_{17}(\rho_4)$
and thus is an involution on $E(\rho_4)$.
To confirm whether $\alpha$ fixes $E(\rho_4)$ pointwise we simply need to check whether it fixes a third point on $E(\rho_4)$.
We check that $(h_1+j_1,h_2+j_2,h_3+j_3,h_4+j_4)$ is an $H$-module in \texttt{FixedLocus.m2} and we confirm that
\[
\alpha((h_1+j_1,h_2+j_2,h_3+j_3,h_4+j_4))=(h_1+j_1,h_2+j_2,h_3+j_3,h_4+j_4).
\]
Thus, $A$ fixes $E(\rho_4)$ pointwise.

\subsubsection{$E(\rho_4')$} $\rho_4'$ is not an endpoint of the Dynkin diagram and is a distance of $1$ from the center. By \Cref{IN.main}(2), points on $E(\rho_4')$ away from its intersections with $E(\rho_2')$ and $E(\rho_6)$ are given by nontrivial, proper $H$-invariant submodules of $V_{14}(\rho_4')\oplus V_{16}(\rho_4')$ other than $V_{14}(\rho_4')$ and $V_{16}(\rho_4')$. By \Cref{IN.main}(4) and (5), its intersection with $E(\rho_2')$ corresponds to $V_{17}(\rho_2')\oplus V_{14}(\rho_4')$ and its intersection with $E(\rho_6)$ corresponds to $V_{16}(\rho_4')\oplus [S_1\cdot V_{16}(\rho_4')]$. By \cite[p.241]{Ito-Nakamura:1999}, we have
\begin{align*}
V_{14}(\rho_4')=&(xy^3\tau_2,-x^4\tau_1,y^4\tau_2,-x^3y\tau_1)\\
V_{16}(\rho_4')=&(-2x^{15}y+77x^{10}y^6-84x^5y^{11}+y^{16}, 35x^{12}y^4+110x^7y^9+15x^2y^{14},\\
&15x^{14}y^2-110x^9y^7+35x^4y^{12},-x^{16}-84x^{11}y^5-77x^6y^{10}-2xy^{15})
\end{align*}
For readability let
\begin{align*}
a_1&=xy^3\tau_2&b_1&=-2x^{15}y+77x^{10}y^6-84x^5y^{11}+y^{16}\\
a_2&=-x^4\tau_1&b_2&=35x^{12}y^4+110x^7y^9+15x^2y^{14}\\
a_3&=y^4\tau_2&b_3&=15x^{14}y^2-110x^9y^7+35x^4y^{12}\\
a_4&=-x^3y\tau_1&b_4&=-x^{16}-84x^{11}y^5-77x^6y^{10}-2xy^{15}
\end{align*}

We check in \texttt{FixedLocus.m2} that $\alpha$ fixes
$V_{14}(\rho_4')\oplus V_{16}(\rho_4')$
and thus is an involution on $E(\rho_4')$. 
To confirm whether $\alpha$ fixes $E(\rho_4')$ pointwise we simply need to check whether it fixes a third point on $E(\rho_4')$.
In \texttt{FixedLocus.m2}, we check that $(a_1+b_2,a_2+b_1,a_3+b_4,a_4+b_3)$ is an $H$-module
of $V_{14}(\rho_4')\oplus V_{16}(\rho_4')$
and we confirm that
\[
\alpha((a_1+b_2,a_2+b_1,a_3+b_4,a_4+b_3))=(a_1-b_2,a_2-b_1,a_3-b_4,a_4-b_3)\neq(a_1+b_2,a_2+b_1,a_3+b_4,a_4+b_3).
\]
Thus $A$ does not fix $E(\rho_4')$ pointwise.

\subsubsection{$E(\rho_5)$} $\rho_5$ is not an endpoint of the Dynkin diagram and is a distance of $1$ from the center. By \Cref{IN.main}(2), points on $E(\rho_5)$ away from its intersections with $E(\rho_4)$ and $E(\rho_6)$ are given by nontrivial, proper $H$-invariant submodules of $V_{14}(\rho_5)\oplus V_{16}(\rho_5)$ other than $V_{14}(\rho_5)$ and $V_{16}(\rho_5)$. By \Cref{IN.main}(4) and (5), its intersection with $E(\rho_4)$ corresponds to $V_{17}(\rho_4)\oplus V_{14}(\rho_5)$ and its intersection with $E(\rho_6)$ corresponds to $V_{16}(\rho_5)\oplus [S_1\cdot V_{16}(\rho_5)]$. By \cite[p.241]{Ito-Nakamura:1999}, we have
\begin{align*}
V_{14}(\rho_5)=&(x^4\sigma_1,-2x^{13}y+33x^8y^6-8x^3y^{11},-5x^{12}y^2-5x^2y^{12},-8x^{11}y^3-33x^6y^8-2xy^{13},-y^4\sigma_2)\\
V_{16}(\rho_5)=&(64x^{15}y+728x^{10}y^6+y^{16},66x^{14}y^2+676x^9y^7-91x^4y^{12},56x^{13}y^3+741x^8y^8-56x^3y^{13},\\
&91x^{12}y^4+676x^7y^9-66x^2y^{14},x^{16}+728x^6y^{10}-64xy^{15})
\end{align*}
For readability let
\begin{align*}
k_1&=x^4\sigma_1&l_1&=64x^{15}y+728x^{10}y^6+y^{16}\\
k_2&=-2x^{13}y+33x^8y^6-8x^3y^{11}&l_2&=66x^{14}y^2+676x^9y^7-91x^4y^{12}\\
k_3&=-5x^{12}y^2-5x^2y^{12}&l_3&=56x^{13}y^3+741x^8y^8-56x^3y^{13}\\
k_4&=-8x^{11}y^3-33x^6y^8-2xy^{13}&l_4&=91x^{12}y^4+676x^7y^9-66x^2y^{14}\\
k_5&=-y^4\sigma_2&l_5&=x^{16}+728x^6y^{10}-64xy^{15}
\end{align*}

We check in \texttt{FixedLocus.m2} that $\alpha$ fixes
$V_{14}(\rho_5)\oplus V_{16}(\rho_5)$
and thus is an involution on $E(\rho_5)$. To confirm whether $\alpha$ fixes $E(\rho_5)$ pointwise we simply need to check whether it fixes a third point on $E(\rho_5)$.
We check that $(k_1+l_1,k_2+l_2,k_3+l_3,k_4+l_4,k_5-l_5)$ is an $H$-module in \texttt{FixedLocus.m2} and we confirm that $A$ does not fix $E(\rho_5)$ pointwise:
\[
  \alpha((k_1+l_1,k_2+l_2,k_3+l_3,k_4+l_4,k_5-l_5))\neq(k_1+l_1,k_2+l_2,k_3+l_3,k_4+l_4,k_5-l_5).
\]
\goodbreak

\section{McKay Quivers for $G$ and $H = G \cap \SL(2,\bC)$}\label{appendixb}

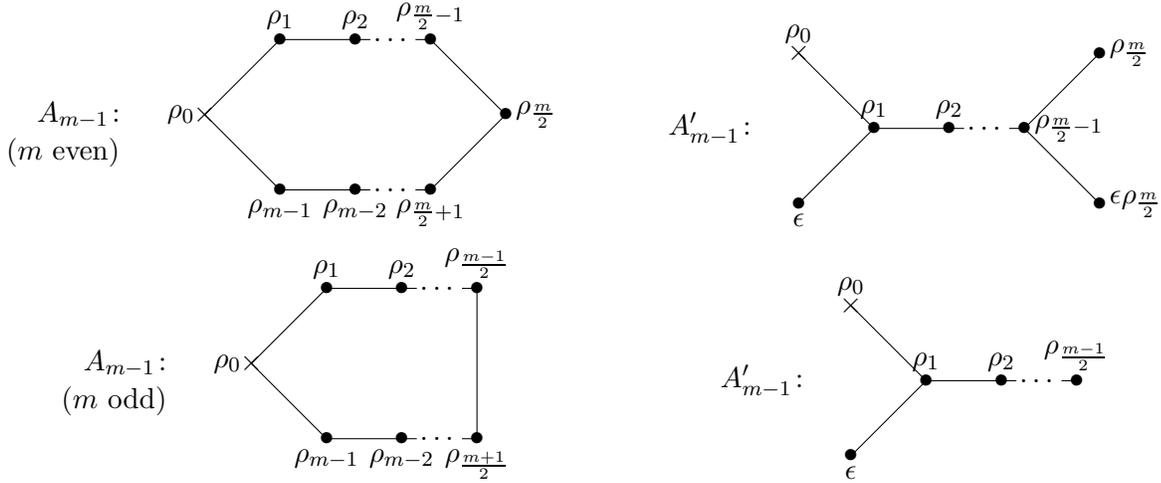
\begin{figure}[h!]
  \begin{subfigure}[t]{0.49\textwidth}
    \centering
    \begin{tikzpicture}
      \node at (-2,0) {\llap{$A_{m-1}\colon$}};
      \node at (-2,-0.5) {\llap{($m$ even) }};
      \node at (-1,0) {$\times$};
      \node at (0,1) {$\bullet$};
      \node at (0,-1) {$\bullet$};
      \node at (1,1) {$\bullet$};
      \node at (1,-1) {$\bullet$};
      \node at (2,1) {$\bullet$};
      \node at (2,-1) {$\bullet$};
      \node at (3,0) {$\bullet$};
      \draw (0,-1)--(-1,0)--(0,1);
      \draw (0,1)--(1,1);
      \draw (0,-1)--(1,-1);
      \draw (2,1)--(3,0)--(2,-1);
      \draw (1,1)--(1.2,1);
      \draw (2,1)--(1.8,1);
      \node at (1.5,1) {$\dots$};
      \draw (1,-1)--(1.2,-1);
      \draw (2,-1)--(1.8,-1);
      \node at (1.5,-1) {$\dots$};
      \node[left] at (-1,0) {$\rho_0$};
      \node[above] at (0,1) {$\rho_1$};
      \node[below] at (0,-1) {$\rho_{m-1}$};
      \node[above] at (1,1) {$\rho_2$};
      \node[below] at (1,-1) {$\rho_{m-2}$};
      \node[above] at (2,1) {$\rho_{\frac{m}{2}-1}$};
      \node[below] at (2,-1) {$\rho_{\frac{m}{2}+1}$};
      \node[right] at (3,0) {$\rho_{\frac{m}{2}}$};
    \end{tikzpicture}
  \end{subfigure}
  \begin{subfigure}[t]{0.49\textwidth}
    \centering
    \begin{tikzpicture}
      \node at (6.5,0) {\llap{$A_{m-1}'\colon$}};
      \node at (7,1) {$\times$};
      \node at (7,-1) {$\bullet$};
      \node at (8,0) {$\bullet$};
      \node at (9,0) {$\bullet$};
      \node at (10,0) {$\bullet$};
      \node at (11,1) {$\bullet$};
      \node at (11,-1) {$\bullet$};
      \draw (7,-1)--(8,0)--(7,1);
      \draw (8,0)--(9,0);
      \draw (11,-1)--(10,0)--(11,1);
      \draw (9,0)--(9.2,0);
      \draw (10,0)--(9.8,0);
      \node at (9.5,0) {$\dots$};
      \node[above] at (7,1) {$\rho_0$};
      \node[below] at (7,-1) {$\ep$};
      \node[above] at (8,0) {$\rho_1$};
      \node[above] at (9,0) {$\rho_2$};
      \node[right] at (10,0) {$\rho_{\frac{m}{2}-1}$};
      \node[right] at (11,1) {$\rho_{\frac{m}{2}}$};
      \node[right] at (11,-1) {$\ep\rho_{\frac{m}{2}}$};
    \end{tikzpicture}
  \end{subfigure}

  \begin{subfigure}[t]{0.49\textwidth}
    \centering
    \begin{tikzpicture}
      \node at (-2,0) {\llap{$A_{m-1}\colon$}};
      \node at (-2,-0.5) {\llap{($m$ odd) }};
      \node at (-1,0) {$\times$};
      \node at (0,1) {$\bullet$};
      \node at (0,-1) {$\bullet$};
      \node at (1,1) {$\bullet$};
      \node at (1,-1) {$\bullet$};
      \node at (2,1) {$\bullet$};
      \node at (2,-1) {$\bullet$};
      \draw (0,-1)--(-1,0)--(0,1);
      \draw (0,1)--(1,1);
      \draw (0,-1)--(1,-1);
      \draw (2,1)--(2,-1);
      \draw (1,1)--(1.2,1);
      \draw (2,1)--(1.8,1);
      \node at (1.5,1) {$\dots$};
      \draw (1,-1)--(1.2,-1);
      \draw (2,-1)--(1.8,-1);
      \node at (1.5,-1) {$\dots$};
      \node[left] at (-1,0) {$\rho_0$};
      \node[above] at (0,1) {$\rho_1$};
      \node[below] at (0,-1) {$\rho_{m-1}$};
      \node[above] at (1,1) {$\rho_2$};
      \node[below] at (1,-1) {$\rho_{m-2}$};
      \node[above] at (2,1) {$\rho_{\frac{m-1}{2}}$};
      \node[below] at (2,-1) {$\rho_{\frac{m+1}{2}}$};
    \end{tikzpicture}
  \end{subfigure}
  \begin{subfigure}[t]{0.49\textwidth}
    \centering
    \begin{tikzpicture}
      \node at (6.5,0) {\llap{$A_{m-1}'\colon$}};
      \node at (7,1) {$\times$};
      \node at (7,-1) {$\bullet$};
      \node at (8,0) {$\bullet$};
      \node at (9,0) {$\bullet$};
      \node at (10,0) {$\bullet$};
      \draw (7,-1)--(8,0)--(7,1);
      \draw (8,0)--(9,0);
      \draw (9,0)--(9.2,0);
      \draw (10,0)--(9.8,0);
      \node at (9.5,0) {$\dots$};
      \node[above] at (7,1) {$\rho_0$};
      \node[below] at (7,-1) {$\ep$};
      \node[above] at (8,0) {$\rho_1$};
      \node[above] at (9,0) {$\rho_2$};
      \node[above] at (10,0) {$\rho_{\frac{m-1}{2}}$};
    \end{tikzpicture}
  \end{subfigure}
  \caption{$A_{m-1}$ and $A_{m-1}'$ Dynkin diagrams}\label{An.dynkin}
  \end{figure}

  \begin{figure}[h!]
    \begin{subfigure}[t]{0.99\textwidth}
      \centering
      \begin{tikzpicture}
        \node at (-2,0) {$D_{m+2}\colon$};
        \node at (-1,1) {$\bullet$};
        \node at (-1,-1) {$\bullet$};
        \node at (0,0) {$\bullet$};
        \node at (1,0) {$\bullet$};
        \node at (2,0) {$\bullet$};
        \node at (3,0) {$\bullet$};
        \node at (4,0) {$\bullet$};
        \node at (3,1) {$\times$};
        \node at (1.5,0) {$\dots$};
        \draw (-1,1)--(0,0);
        \draw (-1,-1)--(0,0);
        \draw (0,0)--(1,0);
        \draw (2,0)--(3,0);
        \draw (3,0)--(4,0);
        \draw (3,0)--(3,1);
        \node[right] at (-1,-1) {$\rho_{m+2}'$};
        \node[right] at (-1,1) {$\rho_{m+1}'$};
        \node[below] at (0,0) {$\rho_m$};
        \node[below] at (1,0) {$\rho_{m-1}$};
        \node[below] at (2,0) {$\rho_3$};
        \node[below] at (3,0) {$\rho_2$};
        \node[below] at (4,0) {$\rho_1'$};
        \node[right] at (3,1) {$\rho_0$};
      \end{tikzpicture}
    \end{subfigure}
  
  \begin{subfigure}[t]{0.49\textwidth}
    \centering 
    \begin{tikzpicture}
      \node at (-2,1.5) {\llap{$D_{m+2}'\colon$}};
      \node at (-2,1) {\llap{($m$ even) }};  
      \node at (-2,0) {$\bullet$};
      \node at (0,1) {$\bullet$};
      \node at (0,-1) {$\bullet$};
      \node at (1,1) {$\bullet$};
      \node at (1,-1) {$\bullet$};
      \node at (2,1) {$\bullet$};
      \node at (2,-1) {$\bullet$};
      \node at (3,1) {$\bullet$};
      \node at (3,-1) {$\bullet$};
      \node at (4,1) {$\bullet$};
      \node at (4,-1) {$\bullet$};
      \node at (-1.33,0) {$\bullet$};
      \node at (-0.67,0) {$\bullet$};
      \node at (0,0) {$\bullet$};
      \node at (2.75,0) {$\bullet$};
      \node at (3.25,0) {$\times$};
      \draw[->-] (0,1)--(-2,0);
      \draw[->-] (-2,0)--(0,-1);
      \draw[->-] (1,1)--(0,1);
      \draw[->-] (1,-1)--(0,-1);
      \draw[->-] (3,1)--(2,1);
      \draw[->-] (3,-1)--(2,-1);
      \draw[->-] (-1.33,0)--(0,1);
      \draw[->-] (0,-1)--(-1.33,0);
      \draw[->-] (-0.67,0)--(0,1);
      \draw[->-] (0,-1)--(-0.67,0);
      \draw[->-] (3,1)--(2.75,0);
      \draw[->-] (2.75,0)--(3,-1);
      \draw[->-] (3.25,0)--(3,1);
      \draw[->-] (3,-1)--(3.25,0);
      \draw[->-] (0,1)--(0,0);
      \draw[->-] (0,0)--(0,-1);
      \draw[->--] (0,-1)--(1,1);
      \draw[->--] (0,1)--(1,-1);
      \draw[->--] (2,-1)--(3,1);
      \draw[->--] (2,1)--(3,-1);
      \draw[->--] (3,-1)--(4,1);
      \draw[->--] (3,1)--(4,-1);
      \draw[->-] (4,-1)--(3,-1);
      \draw[->-] (4,1)--(3,1);
      \draw (2,1)--(1.8,1);
      \draw (1.2,1)--(1,1);
      \node at (1.5,1) {$\dots$};
      \draw (2,-1)--(1.8,-1);
      \draw (1.2,-1)--(1,-1);
      \node at (1.5,-1) {$\dots$};
      \draw[->-] (1,1)--(1.4,0.2);
      \draw[->-] (1,-1)--(1.4,-0.2);
      \draw[->-] (1.6,0.2)--(2,1);
      \draw[->-] (1.6,-0.2)--(2,-1);
      \node at (1.5,0) {$\dots$};
      \node[left] at (-2,0) {$\ep\rho_{m+2}$};
      \node[above] at (-1.33,0) {$\rho_{m+2}$};
      \node[below] at (-0.67,0) {$\rho_{m+1}$};
      \node[above] at (0,0) {$\,\,\,\ep\rho_{m+1}$};
      \node[above] at (2.75,0) {$\ep$};
      \node[below left] at (3.25,0) {$\rho_0\!\!\!$};
      \node[above] at (0,1) {$\rho_m$};
      \node[below] at (0,-1) {$\ep\rho_m$};
      \node[above] at (1,1) {$\rho_{m-1}$};
      \node[below] at (1,-1) {$\ep\rho_{m-1}$};
      \node[above] at (2,1) {$\rho_3$};
      \node[below] at (2,-1) {$\ep\rho_3$};
      \node[above] at (3,1) {$\rho_2$};
      \node[below] at (3,-1) {$\ep\rho_2$};
      \node[above] at (4,1) {$\rho_1'$};
      \node[below] at (4,-1) {$\ep\rho_1'$};
    \end{tikzpicture}
  \end{subfigure}
\vspace*{0.2cm}  
  \begin{subfigure}[t]{0.49\textwidth}
    \centering 
    \begin{tikzpicture}
      \node at (7,1.5) {\llap{$D_{m+2}'\colon$}};
      \node at (7,1) {\llap{($m$ odd) }};
      \node at (7,0) {$\bullet$};
      \node at (8,1) {$\bullet$};
      \node at (8,-1) {$\bullet$};
      \node at (9,1) {$\bullet$};
      \node at (9,-1) {$\bullet$};
      \node at (10,1) {$\bullet$};
      \node at (10,-1) {$\bullet$};
      \node at (11,1) {$\bullet$};
      \node at (11,-1) {$\bullet$};
      \node at (12,1) {$\bullet$};
      \node at (12,-1) {$\bullet$};
      \node at (10.75,0) {$\bullet$};
      \node at (11.25,0) {$\times$};
      \draw (7,0)--(8,1);
      \draw (7,0)--(8,-1);
      \draw[->-] (9,1)--(8,1);
      \draw[->-] (9,-1)--(8,-1);
      \draw[->-] (11,1)--(10,1);
      \draw[->-] (11,-1)--(10,-1);
      \draw[->-] (11,1)--(10.75,0);
      \draw[->-] (10.75,0)--(11,-1);
      \draw[->-] (11.25,0)--(11,1);
      \draw[->-] (11,-1)--(11.25,0);
      \draw[->--] (8,-1)--(9,1);
      \draw[->--] (8,1)--(9,-1);
      \draw[->--] (10,-1)--(11,1);
      \draw[->--] (10,1)--(11,-1);
      \draw[->--] (11,-1)--(12,1);
      \draw[->--] (11,1)--(12,-1);
      \draw[->-] (12,-1)--(11,-1);
      \draw[->-] (12,1)--(11,1);
      \node[left] at (7,0) {$\rho_{m+1}$};
      \node[above] at (8,1) {$\rho_m$};
      \node[below] at (8,-1) {$\ep\rho_m$};
      \node[above] at (9,1) {$\rho_{m-1}$};
      \node[below] at (9,-1) {$\ep\rho_{m-1}$};
      \draw[->-] (9,1)--(9.4,0.2);
      \draw[->-] (9,-1)--(9.4,-0.2);
      \draw[->-] (9.6,0.2)--(10,1);
      \draw[->-] (9.6,-0.2)--(10,-1);
      \node at (9.5,1) {$\dots$};
      \node at (9.5,0) {$\dots$};
      \draw (9,1)--(9.2,1);
      \draw (10,1)--(9.8,1);
      \node at (9.5,-1) {$\dots$};
      \draw (9,-1)--(9.2,-1);
      \draw (10,-1)--(9.8,-1);
      \node[above] at (10,1) {$\rho_3$};
      \node[below] at (10,-1) {$\ep\rho_3$};
      \node[above] at (11,1) {$\rho_2$};
      \node[below] at (11,-1) {$\ep\rho_2$};
      \node[above] at (12,1) {$\rho_1'$};
      \node[below] at (12,-1) {$\ep\rho_1'$};
      \node[above] at (10.75,0) {$\ep$};
      \node[below] at (11.25,0) {$\rho_0$};
      \end{tikzpicture}
  \end{subfigure}
  \caption{$D_{m+2}$ and $D_{m+2}'$ Dynkin diagrams}\label{Dn.dynkin}
\end{figure}
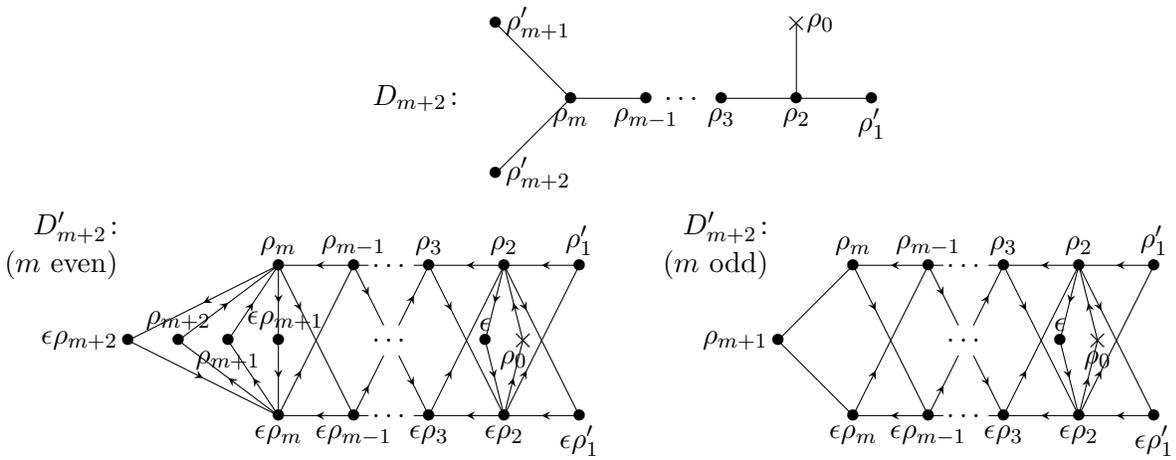

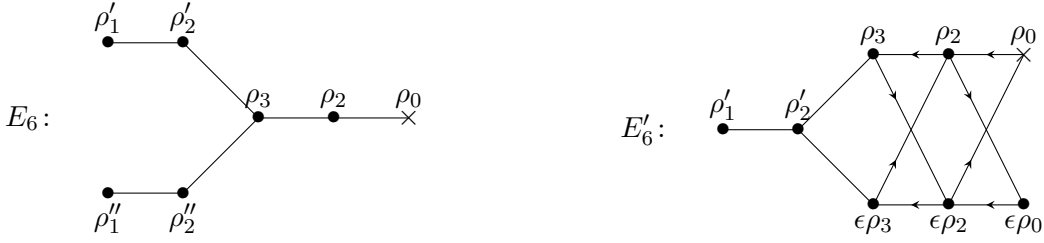
\begin{figure}[h!]
\vspace*{0.2cm}
    \begin{subfigure}[t]{0.49\textwidth}
      \centering
      \begin{tikzpicture}
        \node at (-3,0) {$E_6\colon$};
        \node at (-2,1) {$\bullet$};
        \node at (-2,-1) {$\bullet$};
        \node at (-1,1) {$\bullet$};
        \node at (-1,-1) {$\bullet$};
        \node at (0,0) {$\bullet$};
        \node at (1,0) {$\bullet$};
        \node at (2,0) {$\times$};
        \draw (-2,1)--(-1,1);
        \draw (-2,-1)--(-1,-1);
        \draw (-1,1)--(0,0);
        \draw (-1,-1)--(0,0);
        \draw (0,0)--(1,0);
        \draw (1,0)--(2,0);
        \node[above] at (-2,1) {$\rho_1'$};
        \node[below] at (-2,-1) {$\rho_1''$};
        \node[above] at (-1,1) {$\rho_2'$};
        \node[below] at (-1,-1) {$\rho_2''$};
        \node[above] at (0,0) {$\rho_3$};
        \node[above] at (1,0) {$\rho_2$};
        \node[above] at (2,0) {$\rho_0$};
      \end{tikzpicture}
    \end{subfigure}
    \begin{subfigure}[t]{0.49\textwidth}
      \centering
      \begin{tikzpicture}
        \node at (4,0) {$E_6'\colon$};
        \node at (5,0) {$\bullet$};
        \node at (6,0) {$\bullet$};
        \node at (7,1) {$\bullet$};
        \node at (7,-1) {$\bullet$};
        \node at (8,1) {$\bullet$};
        \node at (8,-1) {$\bullet$};
        \node at (9,1) {$\times$};
        \node at (9,-1) {$\bullet$};
        \draw(5,0)--(6,0);
        \draw (6,0)--(7,1);
        \draw (6,0)--(7,-1);
        \draw[->-] (8,1)--(7,1);
        \draw[->--] (7,-1)--(8,1);
        \draw[->--] (7,1)--(8,-1);
        \draw[->--] (8,-1)--(9,1);
        \draw[->--] (8,1)--(9,-1);
        \draw[->-] (8,-1)--(7,-1);
        \draw[->-] (9,1)--(8,1);
        \draw[->-] (9,-1)--(8,-1);
        \node[above] at (5,0) {$\rho_1'$};
        \node[above] at (6,0) {$\rho_2'$};
        \node[above] at (7,1) {$\rho_3$};
        \node[below] at (7,-1) {$\ep\rho_3$};
        \node[above] at (8,1) {$\rho_2$};
        \node[below] at (8,-1) {$\ep\rho_2$};
        \node[above] at (9,1) {$\rho_0$};
        \node[below] at (9,-1) {$\ep\rho_0$};
      \end{tikzpicture}
    \end{subfigure}
  \caption{$E_6$ and $E_6'$ Dynkin diagrams}\label{E6.dynkin}
  \end{figure}
  
  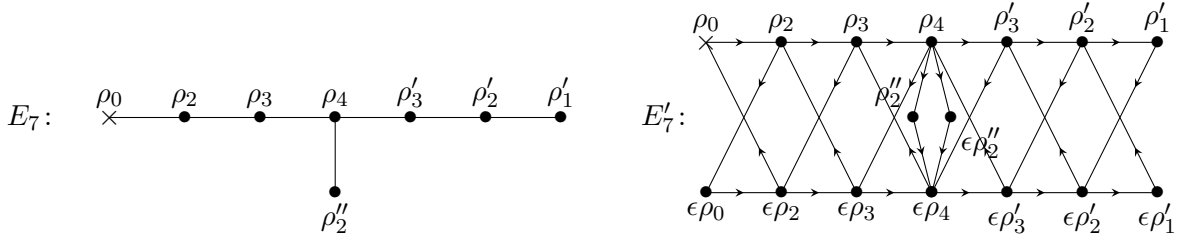
\begin{figure}[h!]
\vspace*{0.2cm}    
    \begin{subfigure}[t]{0.49\textwidth}
      \centering 
      \begin{tikzpicture}
        \node at (-2,0) {$E_7\colon$};
        \node at (-1,0) {$\times$};
        \node at (0,0) {$\bullet$};
        \node at (1,0) {$\bullet$};
        \node at (2,0) {$\bullet$};
        \node at (2,-1) {$\bullet$};
        \node at (3,0) {$\bullet$};
        \node at (4,0) {$\bullet$};
        \node at (5,0) {$\bullet$};
        \draw (-1,0)--(0,0);
        \draw (0,0)--(1,0);
        \draw (1,0)--(2,0);
        \draw (2,0)--(3,0);
        \draw (2,0)--(2,-1);
        \draw (3,0)--(4,0);
        \draw (4,0)--(5,0);
        \node[above] at (-1,0) {$\rho_0$};
        \node[above] at (0,0) {$\rho_2$};
        \node[above] at (1,0) {$\rho_3$};
        \node[above] at (2,0) {$\rho_4$};
        \node[below] at (2,-1) {$\rho_2''$};
        \node[above] at (3,0) {$\rho_3'$};
        \node[above] at (4,0) {$\rho_2'$};
        \node[above] at (5,0) {$\rho_1'$};
      \end{tikzpicture}
    \end{subfigure}
    \begin{subfigure}[t]{0.49\textwidth}
      \centering 
      \begin{tikzpicture}
        \node at (6.5,0) {$E_7'\colon$};
        \node at (7,1) {$\times$};
        \node at (8,1) {$\bullet$};
        \node at (9,1) {$\bullet$};
        \node at (10,1) {$\bullet$};
        \node at (10,-1) {$\bullet$};
        \node at (11,1) {$\bullet$};
        \node at (12,1) {$\bullet$};
        \node at (13,1) {$\bullet$};
        \node at (7,-1) {$\bullet$};
        \node at (8,-1) {$\bullet$};
        \node at (9,-1) {$\bullet$};
        \node at (10,-1) {$\bullet$};
        \node at (11,-1) {$\bullet$};
        \node at (12,-1) {$\bullet$};
        \node at (13,-1) {$\bullet$};
        \draw[->-] (7,1)--(8,1);
        \draw[->--] (8,1)--(7,-1);
        \draw[->-] (8,1)--(9,1);
        \draw[->--] (9,1)--(8,-1);
        \draw[->-] (9,1)--(10,1);
        \draw[->-] (10,1)--(11,1);
        \draw[->--] (10,1)--(9,-1);
        \draw[->--] (11,1)--(10,-1);
        \draw[->-] (11,1)--(12,1);
        \draw[->--] (12,1)--(11,-1);
        \draw[->-] (12,1)--(13,1);
        \draw[->--] (13,1)--(12,-1);
        \draw[->-] (7,-1)--(8,-1);
        \draw[->-] (8,-1)--(9,-1);
        \draw[->--] (8,-1)--(7,1);
        \draw[->-] (9,-1)--(10,-1);
        \draw[->--] (9,-1)--(8,1);
        \draw[->-] (10,-1)--(11,-1);
        \draw[->--] (10,-1)--(9,1);
        \draw[->-] (11,-1)--(12,-1);
        \draw[->--] (11,-1)--(10,1);
        \draw[->-] (12,-1)--(13,-1);
        \draw[->--] (12,-1)--(11,1);
        \draw[->--] (13,-1)--(12,1);
        \node[above] at (7,1) {$\rho_0$};
        \node[above] at (8,1) {$\rho_2$};
        \node[above] at (9,1) {$\rho_3$};
        \node[above] at (10,1) {$\rho_4$};
        \node[above] at (11,1) {$\rho_3'$};
        \node[above] at (12,1) {$\rho_2'$};
        \node[above] at (13,1) {$\rho_1'$};
        \node[below] at (7,-1) {$\ep\rho_0$};
        \node[below] at (8,-1) {$\ep\rho_2$};
        \node[below] at (9,-1) {$\ep\rho_3$};
        \node[below] at (10,-1) {$\ep\rho_4$};
        \node[below] at (11,-1) {$\ep\rho_3'$};
        \node[below] at (12,-1) {$\ep\rho_2'$};
        \node[below] at (13,-1) {$\ep\rho_1'$};
        \node at (9.75,0) {$\bullet$};
        \node at (10.25,0) {$\bullet$};
        \node[above left] at (9.75,0) {$\rho_2''$};
        \node[below right] at (10.25,0) {$\ep\rho_2''$};
        \draw[->-] (10,1)--(9.75,0);
        \draw[->-] (10,1)--(10.25,0);
        \draw[->-] (9.75,0)--(10,-1);
        \draw[->-] (10.25,0)--(10,-1);
      \end{tikzpicture}
    \end{subfigure}
    \caption{$E_7$ and $E_7'$ Dynkin diagrams}\label{E7.dynkin}
  \end{figure}
  
  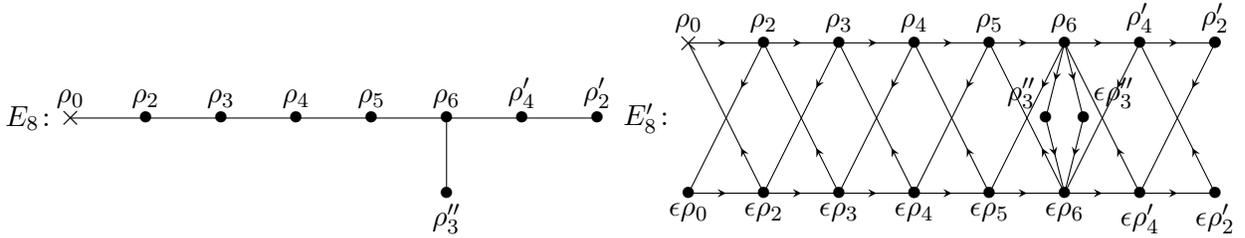
\begin{figure}[h!]
\vspace*{0.2cm}    
    \begin{subfigure}[t]{0.49\textwidth}
      \centering
      \begin{tikzpicture}
        \node at (-1.5,0) {$E_8\colon$};
        \node at (-1,0) {$\times$};
        \node at (0,0) {$\bullet$};
        \node at (1,0) {$\bullet$};
        \node at (2,0) {$\bullet$};
        \node at (4,-1) {$\bullet$};
        \node at (3,0) {$\bullet$};
        \node at (4,0) {$\bullet$};
        \node at (5,0) {$\bullet$};
        \node at (6,0) {$\bullet$};
        \draw (-1,0)--(0,0);
        \draw (0,0)--(1,0);
        \draw (1,0)--(2,0);
        \draw (2,0)--(3,0);
        \draw (4,0)--(4,-1);
        \draw (3,0)--(4,0);
        \draw (4,0)--(5,0);
        \draw (5,0)--(6,0);
        \node[above] at (-1,0) {$\rho_0$};
        \node[above] at (0,0) {$\rho_2$};
        \node[above] at (1,0) {$\rho_3$};
        \node[above] at (2,0) {$\rho_4$};
        \node[above] at (3,0) {$\rho_5$};
        \node[above] at (4,0) {$\rho_6$};
        \node[above] at (5,0) {$\rho_4'$};
        \node[above] at (6,0) {$\rho_2'$};
        \node[below] at (4,-1) {$\rho_3''$};
        \end{tikzpicture}
    \end{subfigure}
    \begin{subfigure}[t]{0.49\textwidth}
      \centering
      \begin{tikzpicture}  
        \node at (6.5,0) {$E_8'\colon$};
        \node at (7,1) {$\times$};
        \node at (8,1) {$\bullet$};
        \node at (9,1) {$\bullet$};
        \node at (10,1) {$\bullet$};
        \node at (10,-1) {$\bullet$};
        \node at (11,1) {$\bullet$};
        \node at (12,1) {$\bullet$};
        \node at (13,1) {$\bullet$};
        \node at (14,1) {$\bullet$};
        \node at (7,-1) {$\bullet$};
        \node at (8,-1) {$\bullet$};
        \node at (9,-1) {$\bullet$};
        \node at (10,-1) {$\bullet$};
        \node at (11,-1) {$\bullet$};
        \node at (12,-1) {$\bullet$};
        \node at (13,-1) {$\bullet$};
        \node at (14,1) {$\bullet$};
        \node at (14,-1) {$\bullet$};
        \draw[->-] (7,1)--(8,1);
        \draw[->--] (8,1)--(7,-1);
        \draw[->-] (8,1)--(9,1);
        \draw[->--] (9,1)--(8,-1);
        \draw[->-] (9,1)--(10,1);
        \draw[->-] (10,1)--(11,1);
        \draw[->--] (10,1)--(9,-1);
        \draw[->--] (11,1)--(10,-1);
        \draw[->-] (11,1)--(12,1);
        \draw[->--] (12,1)--(11,-1);
        \draw[->-] (12,1)--(13,1);
        \draw[->-] (13,1)--(14,1);
        \draw[->-] (13,-1)--(14,-1);
        \draw[->--] (13,1)--(12,-1);
        \draw[->--] (14,1)--(13,-1);
        \draw[->-] (7,-1)--(8,-1);
        \draw[->-] (8,-1)--(9,-1);
        \draw[->--] (8,-1)--(7,1);
        \draw[->-] (9,-1)--(10,-1);
        \draw[->--] (9,-1)--(8,1);
        \draw[->-] (10,-1)--(11,-1);
        \draw[->--] (10,-1)--(9,1);
        \draw[->-] (11,-1)--(12,-1);
        \draw[->--] (11,-1)--(10,1);
        \draw[->-] (12,-1)--(13,-1);
        \draw[->--] (12,-1)--(11,1);
        \draw[->--] (13,-1)--(12,1);
        \draw[->--] (14,-1)--(13,1);
        \node[above] at (7,1) {$\rho_0$};
        \node[above] at (8,1) {$\rho_2$};
        \node[above] at (9,1) {$\rho_3$};
        \node[above] at (10,1) {$\rho_4$};
        \node[above] at (11,1) {$\rho_5$};
        \node[above] at (12,1) {$\rho_6$};
        \node[above] at (13,1) {$\rho_4'$};
        \node[above] at (14,1) {$\rho_2'$};
        \node[below] at (7,-1) {$\ep\rho_0$};
        \node[below] at (8,-1) {$\ep\rho_2$};
        \node[below] at (9,-1) {$\ep\rho_3$};
        \node[below] at (10,-1) {$\ep\rho_4$};
        \node[below] at (11,-1) {$\ep\rho_5$};
        \node[below] at (12,-1) {$\ep\rho_6$};
        \node[below] at (13,-1) {$\ep\rho_4'$};
        \node[below] at (14,-1) {$\ep\rho_2'$};
        \node at (11.75,0) {$\bullet$};
        \node at (12.25,0) {$\bullet$};
        \node[above left] at (11.75,0) {$\rho_3''$};
        \node[above right] at (12.25,0) {$\ep\rho_3''$};
        \draw[->-] (12,1)--(11.75,0);
        \draw[->-] (12,1)--(12.25,0);
        \draw[->-] (11.75,0)--(12,-1);
        \draw[->-] (12.25,0)--(12,-1);
        \end{tikzpicture}
    \end{subfigure}
  \caption{$E_8$ and $E_8'$ Dynkin diagrams}\label{E8.dynkin}
  \end{figure}

\pagebreak

\bibliographystyle{alpha}
\bibliography{nccr}

\end{document}